\documentclass[a4paper]{amsart}
\usepackage{amssymb,stmaryrd}
\usepackage{amsfonts}
\usepackage{amstext}
\usepackage{algorithmic}
\usepackage{algorithm}
\usepackage{color}
\usepackage{graphicx}
\usepackage{epstopdf}
\usepackage[all]{xy}
\usepackage{MnSymbol}
\usepackage{tikz-cd}
\tikzcdset{arrow style=tikz, diagrams={>=stealth}}
\numberwithin{equation}{section}
\newtheorem{theorem}{Theorem}[section]
\newtheorem{lemma}[theorem]{Lemma}
\newtheorem{proposition}[theorem]{Proposition}
\newtheorem{corollary}[theorem]{Corollary}

\theoremstyle{definition}
\newtheorem{definition}[theorem]{Definition}
\newtheorem{example}[theorem]{Example}

\theoremstyle{remark}
\newtheorem{remark}[theorem]{\bf{Remark}}

\newcommand{\Hom}{{\rm{Hom}}}

\newcommand{\R}{{\mathbb{R}}}

\newcommand{\CA}{{\mathcal{A}}}
\newcommand{\C}{{\mathbb{C}}}
\newcommand{\Z}{{\mathbb{Z}}}

\renewcommand{\>}{{\rangle}}

\newcommand{\CC}{{\mathcal{C}}}
\newcommand{\CH}{{\mathcal{H}}}
\newcommand{\CE}{{\mathcal{E}}}
\newcommand{\CD}{{\mathcal{D}}}
\newcommand{\CB}{{\mathcal{B}}}
\newcommand{\CJ}{{\mathcal{J}}}

\newcommand{\cK}{{\mathcal{K}}}
\newcommand{\cH}{{\mathcal{H}}}
\newcommand{\cI}{{\mathcal{I}}}
\newcommand{\CI}{{\mathcal{I}}}
\newcommand{\CR}{{\mathcal{R}}}
\newcommand{\cR}{{\mathcal{R}}}

\newcommand{\CL}{{\mathcal{L}}}
\newcommand{\cL}{{\mathcal{L}}}
\newcommand{\CM}{{\mathcal{M}}}
\newcommand{\wedgeq}{{\wedge\kern-5pt\cdot}}

\newcommand{\cg}{\mathfrak{g}}

\newcommand{\cX}{{\mathfrak{X}}}

\newcommand{\p}{{}_{\scriptscriptstyle{\hat{+}}}}
\newcommand{\np}{{}_{\scriptscriptstyle{\hat{[+]}}}}
\newcommand{\m}{{}_{\scriptscriptstyle{\hat{-}}}}
\newcommand{\nm}{{}_{\scriptscriptstyle{\hat{[-]}}}}

\newcommand{\tens}{\otimes}
\newcommand{\ot}{\otimes}

\newcommand{\tuno}[1]{{#1}{}{}^{\scriptscriptstyle{<1>}}}
\newcommand{\tdue}[1]{{#1}{}{}^{\scriptscriptstyle{<2>}}}

\newcommand{\id}{{\rm id}}

\newcommand{\z}{{}_{\scriptscriptstyle{(0)}}}
\renewcommand{\o}{{}_{\scriptscriptstyle{(1)}}}
\renewcommand{\t}{{}_{\scriptscriptstyle{(2)}}}
\renewcommand{\th}{{}_{\scriptscriptstyle{(3)}}}
\newcommand{\fo}{{}_{\scriptscriptstyle{(4)}}}
\newcommand{\fiv}{{}_{\scriptscriptstyle{(5)}}}
\newcommand{\si}{{}_{\scriptscriptstyle{(6)}}}
\newcommand{\sev}{{}_{\scriptscriptstyle{(7)}}}

\newcommand{\rz}{{}_{\scriptscriptstyle{[0]}}}
\newcommand{\ro}
{{}_{\scriptscriptstyle{[1]}}}
\newcommand{\rt}{{}_{\scriptscriptstyle{[2]}}}

\newcommand{\extd}{{\rm d}}
\newcommand{\del}{{\partial}}
\newcommand{\eps}{\epsilon}
\newcommand{\ev}{{\rm ev}}
\newcommand{\coev}{{\rm coev}}
\newcommand{\und}{\underline}

\newcommand{\Rbullet}{\ \und\bullet\ }

\renewcommand{\imath}{\mathrm{i}}

\newcommand{\la}{{\triangleright}}
\newcommand{\ra}{{\triangleleft}}

\allowdisplaybreaks
\begin{document}

\title{*-Hopf algebroids}
\keywords{Bialgebroid, Hopf algebroid, $*$-algebra, bar monoidal category, noncommutative differential geometry}

\subjclass[2000]{16T05, 20G42, 46L87, 58B32}
\thanks{XH and SM were supported by a Leverhulme Trust project grant RPG-2024-177}

\author{Edwin Beggs${}^1$,  Xiao Han${}^2$ and Shahn Majid${}^2$}
\address{ ${}^1$ Department of Mathematics, Bay Campus, Swansea University, SA1 8EN, UK. ${}^2$ Queen Mary University of London,
School of Mathematical Sciences, Mile End Rd, London E1 4NS, UK}

\email{e.j.beggs@swansea.ac.uk, x.h.han@qmul.ac.uk,    s.majid@qmul.ac.uk}

%\thanks{}
%\date{}

\begin{abstract}
We introduce a theory of $*$-structures for bialgebroids and Hopf algebroids over a $*$-algebra,  defined in such a way that the relevant category of (co)modules is a bar category. We show that if $H$ is a Hopf $*$-algebra then the action Hopf algebroid $A\# H$ associated to a braided-commutative algebra in the category of $H$-crossed modules is a full $*$-Hopf algebroid and the Ehresmann-Schauenburg Hopf algebroid $\CL(P,H)$ associated to a Hopf-Galois extension or quantum group principal bundle  $P$ with fibre $H$ forms a $*$-Hopf algebroid pair, when the relevant (co)action  respects $*$. We also show that Ghobadi's bialgebroid  associated to a $*$-differential structure $(\Omega^1,\extd)$ on $A$ forms a $*$-bialgebroid pair and its quotient in the pivotal case a $*$-Hopf algebroid pair when the pivotal structure is compatible with $*$. We show that when $\Omega^1$ is simultaneously free on both sides, Ghobadi's Hopf algebroid is isomorphic to $\CL(A\#H,H)$ for a smash product by a certain Hopf algebra $H$. \end{abstract}
\maketitle

\section{Introduction}

Bialgebroids and Hopf algebroids over a possible noncommutative base $A$ are the analogue of `quantum groupoids' and have been of interest since the early days of quantum groups. The axioms of a bialgebroid in this context are somewhat settled while for a Hopf algebroid there as several notions of different strength. The strongest, that of a `full Hopf algebroid' \cite{BS}, has an antipode $S$ that is antimultiplicative and its natural formulation involves a pair $(\cL,\cR)$ of resp left and right bialgebroids with the same algebra and connected via $S$. Probably the weakest is the notion\cite{schau1} that certain `Galois maps' are invertible and does not involve an antipode at all. Recently in \cite{HM22} we introduced an intermediate `weak antipode' $S$ and showed that this applies to the Ehresmann-Schauenburg Hopf algebroid $\cL(P,H)$ associated to a Hopf-Galois extension or quantum principal bundle when $H$ is coquasitriangular, and that otherwise this is a Hopf algebroid in the Galois sense. In \cite{HM23} we revisited Lu's construction\cite{Lu} of an action Hopf algebroids $A\# H$ associated to a braided-commutative algebra $A$ in the category of $H$-crossed (or Drinfeld-Yetter) modules and showed that it is a Hopf algebroid in the Galois sense. In the present work we also fill in, in passing, that it is a full Hopf algebroid, see Proposition~\ref{propact}. This result is largely in \cite{Lu} with the same antipode map, but put into the modern context.

Also recently, in \cite{AryGho1}, Ghobadi introduced a third general construction for bialgebroids with potentially noncommutative base, this time associated to a differential algebra $(A,\Omega^1,\extd)$ in the sense of an $A$-bimodule $\Omega^1$ equipped with a derivation $\extd:A\to \Omega^1$ and the assumption that the induced map $A\tens A\to \Omega^1$ given by $a\extd b$ is surjective. Given this data, there is a monoidal category ${}_A\CE_A$ in \cite{BM} of $A$-bimodules equipped with left bimodule connections, and a full submonoidal category which we will denote ${}_A\cI\CE_A$,  where the associated generalised braiding induced by the connection is invertible. There are forgetful functors to $A$-bimodules and by the Eilenberg-Moore reconstruction theorem one expects\cite{BM} bialgebroids whose left modules can be identified with these categories, which are the Ghobadi bialgebroids here denoted $\CL(\Omega^1)$ and $\cI\CL(\Omega^1)$ respectively. The former is generated by $A$,  `right vector fields'  $\cX^R:=\Hom_A(\Omega^1,A)$ and $\cX^R\tens\Omega^1$, while the latter has additional generators $\Omega^1\tens_A\cX^L$ for left vector fields $\cX^L:={}_A\Hom(\Omega^1,A)$. It is assumed that $\Omega^1$ is finitely generated projective (f.g.p.) and that there are chosen covevaluations. Finally, a bimodule is said to be {\em pivotal} if there is a single bimodule $\cX$ providing both a left and right dual, i.e. equipped with both left and right evaluation and coevaluation maps simultaneously (the evaluation maps then render $\cX$ as hom spaces on the two sides). When $\Omega^1$ is pivotal, Ghobadi introduced a quotient Hopf algebroid in the Galois sense which we will denote $\cI\cI\CL(\Omega^1)$ and the modules of which can be identified with a further submonoidal category  ${}_A\cI\cI\CE_A$, where the generalised braiding is {\em binvertible} in a certain sense. As a warm up, Ghobadi in \cite{AryGho1} also introduced bialgebroids $\CB(\Omega^1), \CI\CB(\Omega^1)$ and a full Hopf algebroid which we will denote $\CI\CI\CB(\Omega^1)$, reconstructed from categories ${}_A\CM_A^{\Omega^1}$, ${}_A\CI\CM_A^{\Omega^1}$, ${}_A\CI\CI\CM_A^{\Omega^1}$ with objects pairs $(M,\sigma_M)$, where $M$ is an $A$-bimodule and $\sigma_M:M\tens_A\Omega^1\to \Omega^1\tens_A M$ is a bimodule map, invertible and biinvertible in the latter two cases. Here, $\Omega^1$ is a fixed object in the $A$-bimodule category (not necessarily part of a differential structure) and assumed pivotal for the Hopf algebroid case. Ghobadi's constructions are less well-known than the action and Ehresmann-Schauenburg ones but are important for noncommutative geometry. They have a further quotient  which is rather involved and  generates flat connections of the above types, and hence in the classical limit lands on the algebra of differential operators on a manifold.

What is still missing, however, but rather critical for applications in mathematical physics and (potentially) in quantum computing, is the correct notion of a $*$-structure. This is needed to express and maintain `unitarity' of various constructions. In the present paper, we address this fundamental problem and come up with a notion of a full $*$-Hopf algebroid and a weaker notion of a $*$-Hopf algebroid pair or $*$-bialgebroid pair sufficient to cover the three general constructions above. Here $A$ is a $*$-algebra over $\C$ and for the Ghobadi constructions $(\Omega^1,\extd)$ is a $*$-differential calculus in the standard sense in noncommutative geometry\cite{Con,BM}.  We show in Theorem~\ref{thmact} that if $H$ is a Hopf $*$-algebra in the usual sense (see \cite{BM}, for example) with a compatible action then the action bialgebroid is a full $*$-Hopf algebroid. We also show,  Proposition~\ref{propEScom},  that the  Ehresmann-Schauenburg bialgebroid $\cL(P,H)$ is a full $*$-Hopf algebroid if $H$ is commutative. An example here is provided by the $\theta$-deformed sphere in the sense of Landi and Suijlekom\cite{LS04}, in Section~\ref{secthetasph}.  Section~3 contains further general results, notably Proposition~\ref{propfullbar} that the category of comodules of a full $*$-Hopf algebroid is a bar category in the sense \cite{BM09,BM}, generalising the case of usual Hopf $*$-algebras. We also show, see Section~\ref{sec:cross}, that the recently introduced crossed or Drinfeld-Yetter modules of a full Hopf algebroid\cite{XH23} are a bar category. Section~\ref{secpair} then introduces the notions of a $*$-bialgebroid pair and $*$-Hopf algebroid pair. This requires a pair $(\cL,\cR)$ where $\CR$ is a right-bialgebroid over the {\em same} base as $\cL$, which have the additional properties of being $*$-related and reflexive, with result in Theorem~\ref{thmpairbar} that the left modules of $\CL$ and the right modules of $\CR$ can be identified and are bar categories. This is now sufficient, see Proposition~\ref{propES} to cover $\cL(P,H)$ in general where $H$ could now be noncommutative. Corollary~\ref{coractfull} also provide a  $*$-Hopf algebroid pair and in the process a full $*$-Hopf algebroid different from that obtained from the previous full Hopf $*$-algebroid structure (the two are not related in the same way, the pair notion being more general).

Section~\ref{sec:dif} then covers the $*$-structures for Ghobadi bialgebroids and Hopf algebroids. Here, Sections~\ref{sec:precalc} and Section~\ref{sec:preconn} provide the background respectively of noncommutative differentials and vector fields, and bimodule connections. Section~\ref{secBpiv} shows when $\Omega^1$ is a  $*$-bimodule that ${}_A\CM^I_A$ is a bar category of $\CI\CB(\Omega^1)$-modules. In fact, $\CI\CB(\Omega^1)$ forms a $*$-bialgebroid pair though we do not digress to give details. Building from this, Proposition~\ref{Prop. full star Hopf algebroid H omega} shows when $\Omega^1$ is pivotal that $\CI\CI\CB(\Omega^1)$ is a full $*$-Hopf algebroid. These results are  a  warm-up for the cases of interest in Section~\ref{secIL}, where $(\Omega^1,\extd)$ is a differential structure on $A$ and we show that $\CI\CL(\Omega^1)$ forms a $*$-bialgebroid pair, and in Section~\ref{sec:difpiv} where in the pivotal case $\CI\CI\CL(\Omega^1)$ forms a $*$-Hopf algebroid pair. Section~\ref{secK} shows how Ghobadi's $\CI\CL(\Omega^1)$ construction arises naturally from a more general construction on Proposition~\ref{propodot}. Finally, Section~\ref{sec:ex} provides examples when $\Omega^1$ is pivotal. In fact, this happens automatically when $\Omega^1$ is equipped with a generalised metric $\cg\in \Omega^1\tens_A\Omega^1$ in the sense of quantum Riemannian geometry \cite{BM}. It also happens automatically of $\Omega^1$ is biparallelisable, i.e. a free module on both sides with a joint basis. We show in Proposition~\ref{propGhoES} that in this case $\CI\CI\CL(\Omega^1)\cong \CL(P,H)$ where $P=A\#H$. This is a cocycle Hopf algebroid in the sense of \cite{HM22}, with trivial cocycle. This applies to some of the most well-studied quantum Riemannian geometries (we illustrate this for the integer line graph and the fuzzy sphere). The Ghobadi construction is, however, much more general. Notably, the case of graph calculi, which are not typically free, was studied in \cite{AryGho2}.

\section{Preliminaries}\label{sec:pre}

We consider bialgebroids with base an algebra $B$. We denote by  ${}_B\CM_B$ the monoidal category of $B$-bimodules, and set $B^e=B\tens B^{op}$. Let $b\in B$, we denote $\und{b}$ as an element of $B^{op}$. Namely, we view the underline as a map from $B$ to $B^{op}$ which is the identity map $\id_{B}$ on the vector space $B$. So we have $\und{a}\,\und{b}=\und{ba}$, for any $a, b\in B$. In some contexts, we will use $A$ for the base algebra and where both $A,B$ are used in the same construction, they will be related by $B=A^{op}$ (this is needed for full Hopf algebroids where one half has the opposite base).

\subsection{Left bialgebroids} Begin with an algebra $\mathcal{L}$ with algebra map $s_{L} :B\to \mathcal{L}$, antialgebra map $t_{L} :B\to \mathcal{L}$, and suppose that
all $s_{L}(a)$ commute with all $t_{L}(b)$ for $a,b\in B$.
Then $\mathcal{L}$ is a left $B^e$ module by
$(a\tens b).X=s_{L}(a)\, t_{L}(b)\, X$ for $X\in \mathcal{L}$. This also makes $\mathcal{L}$
into an $B$-bimodule by
\[ a.X.b=s_{L}(a)\, t_{L}(b)\, X.\]
 The above data can be characterised as making $\CL$ an algebra in the category ${}_{B^e}\CM_{B ^e}$, of which we use only the left $B^e$ action. The algebra map $\eta_{L}(a\tens b)=s_{L}(a)t_{L}(b)$ is the unit morphism of this algebra. We then define the Takeuchi product (summation implicit)
 \begin{align*}
\mathcal{L} \times_B \mathcal{L} = \big\{ X\tens Y\in \mathcal{L} \tens_B \mathcal{L} \, \big|\,
X\,t_{L}(a)\tens Y = X\tens Y\, s_{L}(a)\quad \forall a\in B\big\} \  .
\end{align*}
This forms an algebra with pairwise multiplication $(X\tens Y)(U\tens V)=XU\tens YV$. A left bialgebroid is such an $\CL$ equipped with a $B$-coring $\Delta_{L}:\CL\to \CL\tens_B\CL, \varepsilon_{L}:\CL\to B$ in the category ${}_B\CM_B$, where $\Delta_{L}$ has its image in $\mathcal{L} \times_B \mathcal{L}$ and is an algebra map. Also,  $\varepsilon$ satisfies $\varepsilon_{L}(XY)=\varepsilon_{L}(X s_{L}(\varepsilon_{L}(Y)))=\varepsilon_{L}(X t_{L}(\varepsilon_{L}(Y)))$ for any $X,Y\in \CL$. In the following, we will denote the image of the coproduct of left bialgebroids by sumless Sweedler notation $\Delta_{L}(X)=X\o\ot X\t$, for all $X\in \CL$.

\begin{definition}\label{def. left Hopf and anti-Hopf algebroids}

A left $B$-bialgebroid $\cL$ is a left Hopf algebroid (\cite{schau1}, Thm and Def 3.5.) if
\[\lambda: \cL\ot_{B^{op}}\cL\to \cL\ot_{B}\cL,\quad
    \lambda(X\ot_{B^{op}} Y)=X\o \ot_{B}X\t Y\]
is invertible, where the $B^{op}$-bimodule structure is
\[ a.X.b=  t_{L}(a) X  t_{L}(b) \]
 (so $Xt_{L}(a)\ot_{B^{op}}Y=X\ot_{B^{op}}t_{L}(a)Y$, for all $X, Y\in \cL$ and $b\in B$). A left $B$-bialgebroid $\cL$ is an anti-left Hopf algebroid if
\[\mu: \cL\ot^{B}\cL\to \cL\ot_{B}\cL,\quad
    \mu(X\ot^{B} Y)=Y\o X\ot_{B}Y\t\]
is invertible, where we write $\tens^{B}$ to avoid confusion as $\tens_B$ but with respect to a different $B$-bimodule structure $B$-bimodule structure
\[a.X.b=s_{L}(a)Xs_{L}(b).\]
More precisely, the balanced tensor product $\ot^B$ is defined by $s_{L}(a)X\ot^{B}Y=X\ot^{B}Ys_{L}(a)$, for all $X, Y\in \cL$ and $a\in B$.
\end{definition}

We adopt the shorthand
\begin{equation}\label{X+-} X_{+}\ot_{B^{op}}X_{-}:=\lambda^{-1}(X\ot_{B}1),\end{equation}
\begin{equation}\label{X[+][-]} X_{[-]}\ot^{B}X_{[+]}:=\mu^{-1}(1\ot_{B}X).\end{equation}
 It is easy to see that if  $\cL$ be a left $B$-bialgebroid then $\cL^{cop}$ is a left $B^{op}$-bialgebroid with the same underlying algebra structure as $\cL$ and
\begin{equation}\label{LLcop} s^{cop}_{L}:=t_{L}:B^{op}\to \cL,\quad t^{cop}_{L}:=s_{L}:B\to \cL,\quad \Delta^{cop}_{L}(X):=X\t\ot X\o \end{equation}
for any $X\in \cL$.

\subsection{Right bialgebroids} Begin with an algebra $\mathcal{R}$ with linear algebra map $s_{R} :A\to \mathcal{R}$,
linear antialgebra map $t_{R} :A\to \mathcal{R}$, and suppose that
all $s_{R}(a)$ commute with all $t_{R}(b)$ for $a,b\in A$.
Then $\mathcal{R}$ is a right $A^e=A\tens A^{op}$ module by
$X.(a\tens b)=X\, s_{R}(a)\, t_{R}(b)$ for $X\in \mathcal{R}$. This makes $\mathcal{R}$
into an $A$-bimodule by
\[ a.X.b=X\, t_{R}(a)\, s_{R}(b)\, \]
as part of $\CR$ an algebra in the category ${}_{A^e}\CM_{A^e}$. The unit of the algebra is $\eta_{R}(a\tens b)=s_{R}(a)t_{R}(b)$.  We define the Takeuchi crossed product (summation implicit)
 \begin{align*}
\mathcal{R} \times_A \mathcal{R} = \big\{ X\tens Y\in \mathcal{R} \tens_A \mathcal{R}  \, \big|\,
s_{R}(a)\, X\tens Y = X\tens t_{R}(a)\, Y\quad \forall a\in A\big\} \ .
\end{align*}
which is again an algebra and we define $\Delta_{R},\varepsilon_{R}$ similarly as before. We will denote the image of the coproduct of right bialgebroids by sumless Sweedler notation $\Delta_{R}(X)=X\ro\ot X\rt$ for all $X\in \CR$ (this is a slightly different from the left case to help distinguish them).

\begin{definition}\label{def. right Hopf and anti-Hopf algebroids}
A right $A$-bialgebroid  $\cR$  is a right Hopf algebroid if the canonical map
\[\hat{\lambda}: \cR\ot_{A^{op}}\cR\to \cR\ot_{A}\cR,\quad
    \hat{\lambda}(X\ot_{A^{op}} Y)=XY\ro\ot_{A} Y\rt\]
is invertible, where the $A^{op}$-bimodule structure is
\[a.X.b=t_{R}(a)Xt_{R}(b)\]
(so $Xt_{R}(a)\ot_{A^{op}}Y=X\ot_{A^{op}}t_{R}(a)Y$, for all $X, Y\in \cR$ and $a\in A$). A right $A$-bialgebroid $\cR$ is an anti-right Hopf algebroid if the canonical map
\[\hat{\mu}: \cR\ot^{A}\cR\to \cR\ot_{A} \cR,\quad
    \hat{\mu}(X\ot^{A} Y)=X\ro\ot_{A} YX\rt\]
is invertible, where we write $\ot^{A}$ to avoid confusion as $\ot_{A}$ but respect to a different $A$-bimodule structure
\[a.X.b=s_{R}(a)Xs_{R}(b).\]
More precisely, the balanced tensor product $\ot^A$ is defined by $s_{R}(a)X\ot^{A}Y=X\ot^{A}Ys_{R}(a)$, for all $X, Y\in \cR$ and $a\in A$.
\end{definition}
We adopt the shorthand
\begin{align}
    \label{X+-1} X\m\ot_{A^{op}} X\p :=\hat{\lambda}^{-1}(1\ot_{A} X),
\end{align}
\begin{align}\label{X[+][-]1} X\np\ot^{A}X\nm:=\hat{\mu}^{-1}(X\ot_{A} 1).
\end{align}

Similarly to the left case,  if $\cR$ is a right $A$-bialgebroid then $\cR^{cop}$ is a left $A^{op}$-right bialgebroid with the same underlying algebra structure as $\cR$ and
 \begin{equation}s^{cop}_{R}:=t_{R}:A^{op}\to \cR,\quad t^{cop}_{R}:=s_{R}:A\to \cR,\quad \Delta^{cop}_{R}(Y):=Y\rt\ot Y\ro\end{equation}
for any $Y\in \cR$. We similarly have

\begin{proposition}\label{prop. oppo bialgebroid}
 If $\cL$ is a left $B$-bialgebroid then $\cL^{op}$ is a right $B$-bialgebroid with the opposite algebra $\cL^{op}$ and source and target maps
  \[s_{L}^{op}:=t_{L}:B\to \cL^{op},\quad t_{L}^{op}:=s_{L}:B^{op}\to \cL^{op}\]
    for any $X\in \cL$. In view of (\ref{LLcop}),  $\cL^{bop}:=\cL^{cop,op}$ is then a right $B^{op}$-bialgebroid.
\end{proposition}

 Moreover, if $\cL$ is a left Hopf algebroid, we have the following Lemma
\begin{lemma}
   (i)  If $\cL$ is a left $B$-Hopf algebroid, then $\cL^{cop}$ is an anti-left Hopf algebroid and $\cL^{op}$ is an anti-right Hopf algebroid with respectively,
    \[X_{[+]^{cop}}\ot X_{[-]^{cop}}=X_{+}\ot X_{-},\quad X_{\hat{[+]}^{op}}\ot X_{\hat{[-]}^{op}}=X_{+}\ot X_{-}.\]
  (ii)   If $\cL$ is an anti-left $B$-Hopf algebroid then $\cL^{cop}$ is a left Hopf algebroid and $\cL^{op}$ is a right Hopf algebroid with respectively
    \[X_{+^{cop}}\ot X_{-^{cop}}=X_{[+]}\ot X_{[-]},\quad X_{\hat{+}^{op}}\ot X_{\hat{-}^{op}}=X_{[+]}\ot X_{[-]}\]
(iii)  If $\cR$ is a right $A$-Hopf algebroid  then $\cR^{cop}$ is an anti right Hopf algebroid with
    \[X_{\hat{[+]}^{cop}}\ot X_{\hat{[-]}^{cop}}=X_{\hat{+}}\ot X_{\hat{-}}\]
    (and similarly for $\CR^{op}$).

   \noindent (iv)  If $\cR$ is an anti-right $A$-Hopf algebroid then $\cR^{cop}$ is a right Hopf algebroid with
    \[X_{\hat{+}^{cop}}\ot X_{\hat{-}^{cop}}=X_{\hat{[+]}}\ot X_{\hat{[-]}}\]
    and similarly for $\CR^{op}$.
\end{lemma}
\begin{proof}
    We only show the first statement. Let $\cL$ be a left Hopf algebroid, then it is not hard to check the formulae is well defined over the balanced tensor product. We can also see
    \begin{align*}
        X_{[+]^{cop}}\t X_{[-]^{cop}}\ot X_{[+]^{cop}}\o=X_{+}\t X_{-}\ot X_{+}\o=1\ot X,
    \end{align*}
    and
    \begin{align*}
        X\o{}_{[-]^{cop}}X\t \ot X\o{}_{[+]^{cop}}=X\o{}_{-}X\t\ot X\o{}_{+}=1\ot X.
    \end{align*}
\end{proof}

\subsection{Full Hopf algebroids}

\begin{definition}
\label{def. full Hopf algebroid1}
    A left $B$-bialgebroid $\cL$ is a full Hopf algebroid, if there is an invertible anti-algebra map $S:\cL\to \cL$, such that
     \begin{enumerate}
       \item $S\circ t_{L}=s_{L}$,
        \item $(S^{-1}X\t)\o \ot_{B} (S^{-1}X\t)\t X\o  = S^{-1}(X) \ot_{B} 1_\cL$
        \item $S(X\o)\o X\t \ot_{B} S(X\o)\t  = 1_\cL \ot_{B} S(X)$.
    \end{enumerate}
\end{definition}

\begin{proposition}\label{def. full Hopf algebroid2}\cite{BS} Let $(\cL, \varepsilon_{L}, \Delta_{L}, s_{L}, t_{L}, m)$ be a left bialgebroid over $B$. Then $\cL$ is a full Hopf algebroid if and only if $\cR=\cL$ as an algebra forms  right bialgebroid $(\cR, \varepsilon_{R}, \Delta_{R}, s_{R}, t_{R}, m)$ over $A=B^{op}$, such that
   \begin{enumerate}
        \item $s_{L}(B)=t_{R}(A)$ and $t_{L}(B)=s_{R}(A)$ as subalgebras of $\cL=\cR$;
        \item $(\id\ot_{B}\Delta_{R})\circ\Delta_{L}=(\Delta_{L}\ot_{A}\id)\circ\Delta_{R}$ and $(\id\ot_{A}\Delta_{L})\circ\Delta_{R}=(\Delta_{R}\ot_{B}\id)\circ\Delta_{L}$;
        \item $\cL$ is a left-Hopf and anti-left Hopf algebroid.
    \end{enumerate}
    \end{proposition}
 \begin{proof} Details are in \cite{BS}, here we just recall the construction of the required maps.  Given a full Hopf algebroid  $(\cL, \varepsilon_{L}, \Delta_{L}, s_{L}, t_{L}, m)$ with $S$ as in definition \ref{def. full Hopf algebroid1}, we let
 \begin{align}\label{equ. left source target maps to right source and target maps}
  s_{R}(\und{a}):=t_{L}(a), \quad t_{R}(\und{a}):=S^{-1}\circ t_{L}(a),\quad \forall \und{a}\in B^{op}=A,
\end{align}
recall that we view the underline as a map from $B$ to $B^{op}$ which is the identity map  $\id_{B}$ on the vector space $B$.
\begin{align}\label{equ. left coproduct to right coproduct}
    \Delta_{R}(X):=S(S^{-1}(X)\t)\ot_{A} S(S^{-1}(X)\o)=S^{-1}(S(X)\t)\ot_{A} S^{-1}(S(X)\o);
\end{align}
\begin{align}\label{equ. left counit to right counit}
\varepsilon_{R}:=\varepsilon _{L}\circ S.
\end{align}
We can also see $\cL$ is a left Hopf algebroid and an anti-left Hopf algebroid with
\begin{align}   X_{+}\ot_{B^{op}}X_{-}=X\ro\ot_{B^{op}}S(X\rt),\quad X_{[-]}\ot^{B}X_{[+]}=S^{-1}(X\ro)\ot^{B}X\rt.
\end{align}
Conversely, given the left and right bialgebroid as above, we define $S$ and $S^{-1}$ by
\begin{align}\label{equ. antipode in terms of left Hopf structure}
    S(X)=s_{R}(\varepsilon_{R}(X_{+}))X_{-},\quad S^{-1}(X)=t_{R}(\varepsilon_{R}(X_{[+]}))X_{[-]}.
\end{align}
\end{proof}
We also have
\begin{align}
    S^{\pm}(X)\ro\ot_{A} S^{\pm}(X)\rt&=S^{\pm}(X\t)\ot_{A} S^{\pm}(X\o),\\
    S^{\pm}(X)\o\ot_{B}S^{\pm}(X)\t&=S^{\pm}(X\rt)\ot_{B}S^{\pm}(X\ro).
\end{align}

Moreover, $\cR$ is a right Hopf algebroid and anti-right Hopf algebroid with
\begin{align}\label{equ. right Hopf and anti-right Hopf}
    X\m\ot_{A}X\p=S(X\o)\ot_{A}X\t,\quad X\np\ot^{A}X\nm=X\o\ot^{A}S^{-1}(X\t).
\end{align}

 Given this equivalence, we will use $(\cL, \cR, S)$ or $\CH$ to denote the full Hopf algebroid with the left $B$-bialgebroid $\cL$ and right $A$-bialgebroid $\cR$ built on $\cR=\cL$ as an algebra as in Proposition~\ref{def. full Hopf algebroid2}, equipped with antipode $S$ and with $A=B^{op}$.

 \subsection{Modules and comodules}

 (1) If $\CL$ is a left Hopf algebroid over $B$, the category ${}_\CL\CM$ just means left modules of $\CL$ as an algebra. However, in the bialgebroid case there
is a forgetful functor $F:{}_{\mathcal{L}}\mathcal{M} \to {}_B\mathcal{M} {}_B$ given by pullback along $\eta$ as
${}_{\mathcal{L}}\mathcal{M} \to {}_{B^e}\CM$ and the identification of the latter with ${}_B\mathcal{M} {}_B$, which means
\[ a. m.b=s_{L}(a)\, t_{L}(b)\la m.\]
We make ${}_{\mathcal{L}}\mathcal{M} $ into a monoidal category by using
$\tens_B$ with respect to this $B$-bimodule structure, and the action of $\CL$ given by the coproduct, i.e.,
\[
x\la (m\tens n) = (x\o\la m)\tens_{B} (x\t\la n)\ .
\]

The  forgetful functor $F:{}_{\mathcal{L}}\mathcal{M} \to {}_B\mathcal{M} {}_B$ now becomes a strong monoidal functor. This is mostly obvious, except for the actions of $B$ on a tensor product. For this we note that
$\Delta_{L}:\mathcal{L} \to \mathcal{L} \times_B \mathcal{L}$ is an $B$-bimodule map, and from this we deduce
\[
\Delta_{L}(s_{L}(b))=s_{L}(b)\tens 1\ ,\quad \Delta_{L}(t_{L}(b))=1\tens t_{L}(b)\ .
\]
Then we have
\begin{align*}
&b. (m\tens n) =s_{L}(b)\la(m\tens n) =s_{L}(b)\o \la m\tens s_{L}(b)\t\la n=s_{L}(b)\la m\tens n =b.m\tens n\cr
& (m\tens n).b =t_{L}(b)\la(m\tens n) =t_{L}(b)\o \la m\tens t_{L}(b)\t\la n=m\tens t_{L}(b)\la n =m\tens n.b\ .
\end{align*}

(2)  Similarly for right $\CR$-modules.  The  forgetful functor
 $F:\mathcal{M}{}_{\mathcal{R}} \to {}_A\mathcal{M} {}_A$ via the pull-back along $\eta_R$ to an right $A^e$-module resultings in
\[ a. m.b= m\ra t_{R}(a)s_{R}(b).\]

(3) Following  \cite{HS24} we define a right $\cL$-comodule of a left $B$-bialgebroid $\CL$  to be a $B^{op}$-bimodule $\Gamma$, together with a $B^{op}$-bimodule map $\delta_{L}: \Gamma\to \Gamma\times_{B}\cL\subseteq\Gamma\ot_{B}\cL$ (where the balanced tensor product $\ot_{B}$ is such that $\underline{b}\rho\ot_{B}X=\rho\ot_{B}s_{L}(b)X$ for any $\underline{b}\in B^{op}$ since $\Gamma$ can be viewed as a $B$-bimodule with $b\rho b':=\underline{b'}\rho\underline{b}$), written $\delta_{L}(\rho)=\rho\z\ot_{B}\rho\o$ ($\delta_{L}$ is a $B^{op}$-bimodule map in the sense that $\delta_{L}(\underline{b}\rho\underline{b'})=\rho\z\ot t_{L}(b)\rho t_{L}(b')$), such that
    \[(\delta_{L}\ot_{B}\id_{\cL})\circ \delta_{L}=(\id_{\Gamma}\ot_{B}\Delta_{L})\circ \delta_{L},\qquad (\id_{\Gamma}\ot_{B}\varepsilon_{L})\circ \delta_{L}=\id_{\Gamma},\]
where
\begin{equation*} \Gamma\times_{B}\cL :=\{\ \sum_i \rho_i \ot_{B} X_i\in \Gamma\otimes_{B}\cL\ |\ \sum_i \rho_i\underline{b}  \ot_{B} X_i=
\rho_i  \ot_{B} X_i s_{L}(b),\quad \forall \underline{b}\in B^{op}\ \}.
\end{equation*}

\begin{definition}\label{def. right comodule of right bialgebroid}
    Let $\cR$ be a right bialgebroid over $A$, a right $\cR$-comodule is a $A$-bimodule $\Gamma$, together with a $A$-bimodule map $\delta_{R}: \Gamma\to \Gamma\times_{A}\cR\subseteq\Gamma\ot_{A}\cR$, written $\delta_{R}(\rho)=\rho\rz\ot_{A}\rho\ro$ ($\delta_{R}$ is a $A$-bimodule map in the sense that
$\delta_{R}(a\rho a')=\rho\rz \ot_{A} s_{R}(a)\rho\ro s_{R}(a')$), such that
    \[(\delta_{R}\ot_{A}\id_{\cR})\circ \delta_{R}=(\id_{\Gamma}\ot_{A}\Delta_{R})\circ \delta_{R},\qquad (\id_{\Gamma}\ot_{A}\varepsilon_{R})\circ \delta_{R}=\id_{\Gamma},\]
where
\begin{equation*} \Gamma\times_{A}\cR :=\{\ \sum_i \rho_i \ot_{A} X_i\in \Gamma\otimes_{A}\cR\ |\ \sum_i a\rho_i  \ot_{A} X_i=
\rho_i  \ot_{A} t_{R}(a)X_i,\quad \forall a\in A\ \}.
\end{equation*}
\end{definition}

In the following, we are going to introduce a more symmetric Hopf algebroid, namely, a bialgebroid with bijective antipode that is given in \cite{BS}

 Recall that\cite{Boehm}
\begin{definition}\label{def. comodule of full Hopf algebroid}
    Let $\CH=(\cL, \cR, S)$ be a full Hopf algebroid (with $\cL$ be a left $B$-bialgebroid), a right comodule of $\CH$ is a $A$-bimodule $\Gamma$, such that it is a right $\cL$-comodule with coaction $\delta_{L}$ and a right $\cR$-comodule with coaction $\delta_{R}$ satisfying
    \begin{enumerate}
        \item The underlying $B^{op}$-bimodule structure (associated to the right $\cL$-comodule structure) is the same with the underlying $A$-bimodule structure (associated to the right $\cR$-comodule structure) of $\Gamma$.
        \item $(\id\ot_{B}\Delta_{R})\circ\delta_{L}=(\delta_{L}\ot_{A}\id)\circ\delta_{R},\quad (\id\ot_{A}\Delta_{L})\circ\delta_{R}=(\delta_{R}\ot_{B}\id)\circ\delta_{L}.$
    \end{enumerate}
\end{definition}

\section{Full $*$-Hopf algebroids}\label{sec:fullstar}

We assume that the base $A$ is a $*$-algebra and look at associated antilinear structures on left and right bialgebroids.

\begin{definition} \label{pye}
Suppose that $(\mathcal{L},A, s_{L},t_{L},\Delta_{L},\varepsilon_{L})$ is a left $A$-bialgebroid and  that
$(\mathcal{R},A, s_{R},t_{R},\Delta_{R},\varepsilon_{R})$ is a right $A$-bialgebroid. We say that these are $*$-related if
 there is an invertible antilinear anti-algebra map $\circledast:\mathcal{L}\to \mathcal{R}$
so that
\begin{align*}
&s_{R}(a^*)=s_{L}(a)^\circledast\ ,\quad t_{R}(a^*)=t_{L}(a)^\circledast \ ,\cr
& \mathrm{flip}(\circledast \tens \circledast)\Delta_{L} (X)=\Delta_{R}(X^\circledast)\ ,\quad       \varepsilon_{R}(X^\circledast)= \varepsilon_{L}(X)^*\ .
\end{align*}
(Because of the flip here,  $\circledast$ is an anti-coalgebra map.) We say $\CL,\CR$ are anti-$*$-related if $\CL^{cop}$ and $\CR$ are $*$-related.
\end{definition}

To show that this makes sense we will check:

\begin{lemma}\label{lem. star exchange the balanced tensor product}If $\circledast$ is antilinear and obeys the conditions on the respective $s,t$ maps then $\mathrm{flip}(\circledast \tens \circledast)$ sends
$\mathcal{L} \times_A \mathcal{L}$ into $\mathcal{R} \times_A \mathcal{R}$. \end{lemma}
\proof   For $X,Z\in \mathcal{L}$ we have
\begin{align*}
\mathrm{flip}(\circledast \tens \circledast)(X.a\tens Z) &= \mathrm{flip}(\circledast \tens \circledast)(t_{L}(a)\,X\tens Z) = Z^\circledast  \tens  X^\circledast   \, t_{L}(a)^\circledast \cr
&= Z^\circledast  \tens  X^\circledast   \, t_{R}(a^*) = Z^\circledast  \tens  a^*.X^\circledast \ , \cr
\mathrm{flip}(\circledast \tens \circledast)(X\tens a.Z) &= \mathrm{flip}(\circledast \tens \circledast)(X\tens s_{L}(a)\,Z) = Z^\circledast \, s_{L}(a)^\circledast  \tens  X^\circledast   \cr
&= Z^\circledast \, s_{R}(a^*) \tens  X^\circledast    = Z^\circledast .a^* \tens  X^\circledast
\end{align*}
so we get a well defined map on $\mathcal{L} \tens_A \mathcal{L}$. Now we take the left
Takeuchi crossed product condition (summation implicit) and apply $\mathrm{flip}(\circledast \tens \circledast)$,
\begin{align*}
\mathrm{flip}(\circledast \tens \circledast)(X\,t_{L}(a)\tens Z) &= Z^\circledast \tens t_{L}(a)^\circledast\, X^\circledast
= Z^\circledast \tens t_{L}(a^*)\, X^\circledast\ ,\cr
\mathrm{flip}(\circledast \tens \circledast)(X\tens Z\,s_{L}(a)) &= s_{L}(a)^\circledast\, Z^\circledast \tens X^\circledast
= s_{L}(a^*)\, Z^\circledast \tens X^\circledast\ ,
\end{align*}
and equality of these is just the right Takeuchi crossed product condition.
 \endproof

By the above Lemma, we can define a full $*$-Hopf algebroid:

\begin{definition}\label{def. star Hopf algebroids}
    A full Hopf algebroid $(\cL, \cR, S)$ is a full $\star$-Hopf algebroid (with $\cL$ be a left $B$-bialgebroid), if $B,\cL$ are $\star$-algebras and $\cL$ and
    \[ s_{L}(b)^{\star}=t_{R}(\underline{b^{*}}),\quad t_{L}(b)^{\star}=s_{R}(\underline{b^{*}}),\quad \und{\varepsilon_{L}(X^{\star})}=\varepsilon_{R}(X)^{*}\]
    \[X^{\star}\ro\ot X^{\star}\rt=X\o{}^{\star}\ot X\t{}^{\star}\]
    for any $X\in \cL$ and $b\in B$. \end{definition}

\begin{remark}\label{fullcase} This is a special case of $(\cL,\cR)$ are anti-$\star$-related (or  $(\CL^{cop},\CR)$ $*$-related). Indeed, by (\ref{LLcop}), if $(\CL, \CR)$ is a left-right pair as for a full Hopf algebroid in Proposition~\ref{def. full Hopf algebroid2} then $(\CL,\CR)$ anti-star related amounts to an invertible antilinear map $\circledast: \CL\to \CL$ making $\CL$ a $*$-algebra and the conditions as stated, where we simply denote $\circledast$ as $*$ in view of this.
\end{remark}

\begin{lemma}\label{lem. full star Hopf algebroid}
    Let $(\cL, \cR, S)$ be a full $\star$-Hopf algebroid, then the canonical maps obey
    \[X^{\star}{}_{+}\ot X^{\star}{}_{-}=X\np{}^{\star}\ot X\nm{}^{\star},\quad X^{\star}{}_{[+]}\ot X^{\star}{}_{[-]}=X\p{}^{\star}\ot X\m{}^{\star},\]
    for any $X\in\cL$.
\end{lemma}
\begin{proof}
    It is not hard to see that the formulae are well defined over the balanced tensor products. We can see on the one hand,
    \begin{align*}
        X^{\star}{}_{+}\o\ot X^{\star}{}_{+}\t X^{\star}{}_{-}=&X\np{}^{\star}{}\o\ot X\np{}^{\star}{}\t X\nm{}^{\star}=X\np{}\ro{}^{\star}\ot X\np{}\rt{}^{\star} X\nm{}^{\star}\\
        =&X\np{}\ro{}^{\star}\ot (X\nm X\np{}\rt){}^{\star}
        =X^{\star}\ot 1.
    \end{align*}
    On the other hand,
    \begin{align*}
        X^{\star}{}\o{}_{+}\ot X^{\star}{}\o{}_{-}X^{\star}{}\t=&X\ro{}^{\star}{}_{+}\ot X\ro{}^{\star}{}_{-}X\rt{}^{\star}=X\ro{}\np{}^{\star}\ot X\ro{}\nm{}^{\star}X\rt{}^{\star}\\
        =&X\ro{}\np{}^{\star}\ot (X\rt X\ro{}\nm){}^{\star}=X^{\star}\ot 1.
    \end{align*}
\end{proof}

The following corollary is analogous to a well-known fact for ordinary $\star$-Hopf algebras
\begin{corollary}\label{coro. antipode with star}
    If $(\cL, \cR, S)$ be a full $\star$-Hopf algebroid then $S\circ \star\circ S\circ \star=\id$.
\end{corollary}
\begin{proof}
    By (\ref{equ. antipode in terms of left Hopf structure}), we have
    \begin{align*}
        S(X^{\star})=&s_{R}\circ\varepsilon_{R}(X^{\star}{}_{+})X^{\star}{}_{-}=s_{R}\circ\varepsilon_{R}(X\np{}^{\star})X\nm{}^{\star}=(t_{L}\circ\varepsilon_{L}(X\np)){}^{\star}X\nm{}^{\star}\\
        =&(X\nm t_{L}\circ\varepsilon_{L}(X\np))^{\star}=(S^{-1}(X\t)t_{L}\circ\varepsilon_{L}(X\o))^{\star}=(S^{-1}(s_{L}\circ\varepsilon_{L}(X\o) X\t))^{\star}\\
        =&S^{-1}(X)^{\star}.
    \end{align*}
    where the second step uses Lemma \ref{lem. full star Hopf algebroid}, the 5th step uses (\ref{equ. right Hopf and anti-right Hopf}).
\end{proof}

\subsection{Full $*$-Hopf algebroid structure of action bialgebroids}\label{sec:fullaction}

Let $H$ be a Hopf algebra with invertible antipode and $A$ a braided-commutative algebra in the right $H$-crossed module (or Drinfeld-Yetter-module) category ${}\CM^H_H$. The latter means a right action $\ra$ and right coaction $a\mapsto a\z\tens a\o$ such that
\begin{align}\label{equ. Drinfeld-Yetter equation}
    (a\triangleleft h\t)\z\tens h\o (a\triangleleft h\t)\o=a\z\ra h\o\tens a\o h\t,
\end{align}
for all $a\in B$ and $h\in H$. This is equivalent to
\[ (a\ra h)\z\tens (a\ra h)\o=a\z\ra h\t\tens (S h\o) a\o h\th.\]
Braided-commutativity of $A$ and an equivalent version of it are
\[ b\z(a\ra b\o)=ab,\quad  (a\ra S^{-1}b\o) b\z=ba\]
In this situation, there is a left Hopf algebroid $\CL=B\# H^{op}$ with algebra, $s$, $t$ and coalgebra\cite{Lu,HM23}:
\[ (a\tens h)(b\tens g)=a(b\ra h\o)\tens g h\t,\quad s_{L}(a)=a\tens 1,\quad t_{L}(a)=a\z\tens a\o,\]
\[ \varepsilon_{L}(a\tens h)=a\varepsilon(h),\quad \Delta_{L}(a\tens h)=(a\tens h\o)\tens (1\tens h\t).\]
In fact these are known to also be Hopf algebroids in a strong sense. The smash product views the right action of $H$ as a left action of $H^{op}$.

Next, suppose that $A$ is a $*$-algebra and $H$ a Hopf $*$-algebra. We recall that an action and coaction on $A$ are unitary if
\[ a^*{}\z\tens a^*{}\o=a\z{}^*\tens a\o{}^*,\quad (a\ra h)^*= a^*\ra S^{-1}(h^*)\]
respectively. These are standard notions\cite{Ma:book} (but note the different convention in \cite{BM} where the inverse was omitted). It can be shown\cite{BM09} that these relations are compatible with the crossed module structure and we say that the $A$ is unitary as an algebra in $\CM^H_H$. (In the case of $H$ finite-dimensional, this is equivalent to a certain Hopf $*$-algebra structure on the quantum double $D(H)$ as in \cite{Ma:book} and requiring that $A$ is unitary with respect to its action.)

\begin{proposition}\label{propact}
   If $B$ is a braided-commutative algebra in $\CM^H_H$ and $\CL=B\ot H^{op}$ as above, then $\cL$ is a full Hopf algebroid, where the antipode and its inverse are
    \begin{align*}
        S(a\ot h)=&a\z\ra S^{-2}(a\o) S^{-1}(h\t)\ot S^{-2}(a\t) S^{-1}(h\o),\\
        S^{-1}(a\ot h)=&a\z\ra S(h\t)\ot a\o S(h\o).
    \end{align*}
\end{proposition}
\begin{proof}
    To check $S$ is invertible, we have on the one hand
    \begin{align*}
        S\circ S^{-1}(a\ot h)=&S(a\z\ra S(h\t)\ot a\o S(h\o))\\
        =&(a\z\ra S(h\t))\z\ra S^{-2}((a\z\ra S(h\t))\o)S^{-1}(a\o\t S(h\o)\t)\\
        &\ot  S^{-2}((a\z\ra S(h\t))\t)S^{-1}(a\o\o S(h\o)\o)\\
        =&a\z\ra S(h\fiv)h\si S^{-2}(a\o) S^{-1}(h\fo) h\o S^{-1}(a\fo)\\
        &\ot h\sev S^{-2}(a\t) S^{-1}(h\th)h\t S^{-1}(a\th)\\
        =&a\ot h,
    \end{align*}
    on the other hand,
    \begin{align*}
        S^{-1}\circ S(a\ot h)=&S^{-1}(a\z\ra S^{-2}(a\o) S^{-1}(h\t)\ot S^{-2}(a\t) S^{-1}(h\o))\\
        =&S^{-1}(a\z\ra S^{-2}(a\o) S^{-1}(h\t))\z\ra S(S^{-2}(a\t)\t S^{-1}(h\o)\t)\\
        &\ot S^{-1}(a\z\ra S^{-2}(a\o) S^{-1}(h\t))\o  S(S^{-2}(a\t)\o S^{-1}(h\o)\o)\\
        =&a\z\ra S^{-2}(a\th)S^{-1}(h\fo)h\o S^{-1}(a\si)\\
        &\ot h\fiv S^{-1}(a\t) a\o S^{-2}(a\fo) S^{-1}(h\th)h\t S^{-1}(a\fiv)\\
        =&a\ot h.
    \end{align*}
    To check $S$ is an anti-algebra map, it is sufficient to check $S^{-1}$ is an anti-algebra map. We have on the one hand
    \begin{align*}
     S^{-1}((a\ot h)(b\ot g))=&S^{-1}(a(b\ra h\o)\ot gh\t)\\
     =&(a(b\ra h\o))\z\ra S(g\t h\th)\ot (a(b\ra h\o))\o S(g\o h\t)\\
     =&(a\z(b\z\ra h\t))\ra (S(h\fiv)S(g\t))\ot a\o S(h\o)b\o h\th S(h\fo) S(g\o)\\
     =&(a\z(b\z\ra h\t))\ra (S(h\th)S(g\t))\ot a\o S(h\o)b\o  S(g\o)
    \end{align*}
    on the other hand,
    \begin{align*}
        S^{-1}(b\ot g)S^{-1}(a\ot h)=&(b\z\ra S(g\t)\ot b\o S(g\o))(a\z\ra S(h\t)\ot a\o S(h\o))\\
        =&(b\z\ra S(g\th))(a\z\ra S(h\t)b\o S(g\t))\ot a\o S(h\o)b\t S(g\o)\\
        =&(b\z(a\z\ra S(h\t)b\o))\ra S(g\t)\ot a\o S(h\o)b\t S(g\o)\\
        =&((a\z\ra S(h\t))b\z)\ra S(g\t)\ot a\o S(h\o)b\o S(g\o)\\
        =&(a\z(b\z\ra h\t))\ra (S(h\th)S(g\t))\ot a\o S(h\o)b\o  S(g\o).
    \end{align*}
    We can also see
    \begin{align*}
        S\circ t_{L}(a)=&S(a\z\ot a\o)=a\z\ra S^{-2}(a\o) S^{-1}(a\fo)\ot S^{-2}(a\t) S^{-1}(a\th)\\
        =&a\ot 1=s_{L}(a).
    \end{align*}
    And
    \begin{align*}
        S(a\ot& h\o)\o (1\ot h\t)\ot_{B} S(a\ot h\o)\t\\
        =&(a\z\ra S^{-2}(a\o)S^{-1}(h\th)\ot S^{-2}(a\t)S^{-1}(h\t))(1\ot h\fo)\\
        &\ot_{B} (1\ot S^{-2}(a\th)S^{-1}(h\o))\\
        =& (a\z\ra S^{-2}(a\o)S^{-1}(h\th)\ot h\fo S^{-2}(a\t) S^{-1}(h\t))\ot_{B} (1\ot S^{-2}(a\th)S^{-1}(h\o)) \\
        =& (a\z\z\ra S^{-2}(a\o)\t S^{-1}(h\th)\ot h\fo S(S^{-2}(a\o)\o)a\z\o S^{-2}(a\o)\th  S^{-1}(h\t))\\
        &\ot_{B} (1\ot S^{-2}(a\t)S^{-1}(h\o)) \\
        =& (a\z\z\ra S^{-2}(a\o)\t S^{-1}(h\th)\ot h\fo S(S^{-2}(a\o)\o)a\z\o S^{-2}(a\o)\th  S^{-1}(h\t))\\
        &\ot_{B} (1\ot S^{-2}(a\t)S^{-1}(h\o)) \\
        =& (a\z\ra S^{-2}(a\o)S^{-1}(h\t))\z\ot (a\z\ra S^{-2}(a\o)S^{-1}(h\t))\o\\
        &\ot_{B}1\ot S^{-2}(a\t)S^{-1}(h\o)\\
        =&t_{L}(a\z\ra S^{-2}(a\o)S^{-1}(h\t))\ot_{B}1\ot S^{-2}(a\t)S^{-1}(h\o)\\
        =&1\ot 1\ot_{B}a\z\ra S^{-2}(a\o)S^{-1}(h\t)\ot S^{-2}(a\t)S^{-1}(h\o)\\
        =&1\ot 1\ot_{B}S(a\ot h).
    \end{align*}
    And
    \begin{align*}
        S^{-1}(1\ot& h\t)\o\ot_{B}S^{-1}(1\ot h\t)\t (a\ot h\o)\\
        =&1\ot S(h\t)\o\ot_{B}(1\ot S(h\t)\t)(a\ot h\o)\\
        =&1\ot S(h\fo)\ot_{B} a\ra S(h\th)\ot h\o S(h\t)\\
        =&1\ot S(h\t)\ot_{B} a\ra S(h\o)\ot 1\\
        =&1\ot S(h\t)\ot_{B} s_{L}(a\ra S(h\o))\\
        =&t_{L}(a\ra S(h\o))(1\ot S(h\t))\ot_{B} 1\ot 1\\
        =&(a\ra S(h\o))\z \ot S(h\t)(a\ra S(h\o))\o\ot_{B} 1\ot 1\\
        =&a\z\ra S(h\t)\ot S(h\fo)S^{2}(h\th)a\o S(h\o)\ot_{B} 1\ot 1\\
        =&S^{-1}(a\ot h)\ot_{B} 1\ot 1.
    \end{align*}
\end{proof}

By (\ref{equ. left source target maps to right source and target maps}), (\ref{equ. left coproduct to right coproduct}) and (\ref{equ. left counit to right counit}), $B\#H^{op}$ also has a right bialgebroid structure:
\begin{align}
    s_{R}(\und{a})=t_{L}(a)=a\z\ot a\o,\quad t_{R}(\und{a})=S^{-1}\circ t_{L}(a)=a\z\ra S(a\o)\ot 1, \forall\und{a}\in B^{op},
\end{align}
\begin{align}
    \Delta_{R}(a\ot h)=S(S^{-1}(a\ot h)\t)\ot_{A} S(S^{-1}(a\ot h)\o)=1\ot h\o S^{-1}(a\o)\ot_{A} a\z\ot h\t,
\end{align}
\begin{align}
    \varepsilon_{R}(a\ot h)=\und{a\z\ra S^{-2}(a\o)S^{-1}(h)}.
\end{align}

\begin{theorem}\label{thmact}
    Let $B$ be a braided-commutative algebra in $\CM^H_H$ and $\CL=B\#H^{op}$ as above. If $B$ is a $*$-algebra and $H$ a Hopf $*$-algebra with unitary action and coaction, then $\cL$ is a  full $\star$-Hopf algebroid with
    \begin{align*}
        (b\ot h)^{\star}=b^{*}\z\ra S(b^{*}\o)h^{*}\o\ot h^{*}\t
    \end{align*}
\end{theorem}
\begin{proof}
To see $\star$ is involutive, we have
\begin{align*}
   (b\ot h)^{\star\star}=&(b^{*}\z\ra S(b^{*}\o)h^{*}\o\ot h^{*}\t)^{\star}\\
   =&(b^{*}\z\ra S(b^{*}\o)h^{*}\o)^{*}\z\ra S((b^{*}\z\ra S(b^{*}\o)h^{*}\o)^{*}\o)h^{*}\t{}^{*}\o\ot h^{*}\t{}^{*}\t\\
   =&(b\z\ra S^{-1}(h\o S^{-1}(b\o)))\z\ra S((b\z\ra S^{-1}(h\o S^{-1}(b\o)))\o)h\t\ot h\th\\
   =&b\z\z\ra S^{-1}(h\o S^{-1}(b\o))\t\\
   &S\Big(S\big(S^{-1}(h\o S^{-1}(b\o))\o\big)b\z\o S^{-1}(h\o S^{-1}(b\o))\th \Big)h\t\ot h\th\\
   =&b\z\ra S^{-2}(b\th)S^{-1}(h\t)h\o S^{-1}(b\fo)S(b\o)b\t S(h\th)h\fo\ot h\fiv\\
   =&b\ot h.
\end{align*}
To see $\star$ is an anti-algebra map, we have on the one hand
\begin{align*}
    ((a\ot h)(b\ot g))^{\star}=&(a(b\ra h\o)\ot gh\t)^{\star}\\
    =&(a(b\ra h\o))^{*}\z\ra S((a(b\ra h\o))^{*}\o)(gh\t)^{*}\o\ot (gh\t)^{*}\t\\
    =&(b^{*}\ra S^{-1}(h^{*}\z)a^{*})\z\ra S((b^{*}\ra S^{-1}(h^{*}\z)a^{*})\o)h^{*}\t g^{*}\o\ot h^{*}\th g^{*}\t\\
    =&(b^{*}\z\ra S^{-1}(h^{*}\t)a^{*}\z)\\
    &\ra S(a^{*}\o)h^{*}\o S(b^{*}\o)S(h^{*}\th)h^{*}\fo g^{*}\o\ot h^{*}\fiv g^{*}\t\\
    =&(b^{*}\z\ra S^{-1}(h^{*}\t)a^{*}\z)\ra S(a^{*}\o)h^{*}\o S(b^{*}\o) g^{*}\o\ot h^{*}\th g^{*}\t,
\end{align*}
on the other hand,
\begin{align*}
    (b&\ot g)^{\star}(a\ot h)^{\star}\\
    =&(b^{*}\z\ra S(b^{*}\o)g^{*}\o\ot g^{*}\t)(a^{*}\z\ra S(a^{*}\o)h^{*}\o\ot h^{*}\t)\\
    =&(b^{*}\z\ra S(b^{*}\o)g^{*}\o)(a^{*}\z\ra S(a^{*}\o)h^{*}\o g^{*}\t)\ot h^{*}\t g^{*}\th\\
    =&\big((b^{*}\z\ra S(b^{*}\o))(a^{*}\z\ra S(a^{*}\o)h^{*}\o)\big) \ra g^{*}\o\ot h^{*}\t g^{*}\t\\
    =&\big(b^{*}\z(a^{*}\z\ra S(a^{*}\o)h^{*}\o b^{*}\o)\big) \ra S(b^{*}\t) g^{*}\o\ot h^{*}\t g^{*}\t\\
    =&\big((a^{*}\z\ra S(a^{*}\o)h^{*}\o)b^{*}\z\big) \ra S(b^{*}\o) g^{*}\o\ot h^{*}\t g^{*}\t\\
    =&\big((a^{*}\z\ra S(a^{*}\o))(b^{*}\z\ra S^{-1}(h^{*}\t))\big) \ra h^{*}\o S(b^{*}\o) g^{*}\o\ot h^{*}\th g^{*}\t\\
     =&\big(a^{*}\z (b^{*}\z\ra S^{-1}(h^{*}\t)a^{*}\o)\big) \ra S(a^{*}\t) h^{*}\o S(b^{*}\o) g^{*}\o\ot h^{*}\th g^{*}\t\\
    =&(b^{*}\z\ra S^{-1}(h^{*}\t)a^{*}\z)\ra S(a^{*}\o)h^{*}\o S(b^{*}\o) g^{*}\o\ot h^{*}\th g^{*}\t.
\end{align*}
We can also check
\begin{align*}
    s_{L}(b)^{\star}=b^{*}\z\ra S(b^{*}\o)\ot 1=t_{R}(\und{b^{*}}),
\end{align*}
and
\begin{align*}
    t_{L}(b)^{\star}=b^{*}\z\ra S(b^{*}\o)b^{*}\t\ot b^{*}\th=b^{*}\z\ot b^{*}\o=s_{R}(\und{b^{*}})
\end{align*}
and
\begin{align*}
    \varepsilon_{R}(a\ot h)^{\star}=&\und{(a\z\ra S^{-2}(a\o)S^{-1}(h))^{\star}}=\und{a\z{}^{*}\ra S^{-1}((S^{-2}(a\o)S^{-1}(h))^{*})}\\
    =&\und{a\z{}^{*}\ra \big(S^{-1}\circ\ast\circ S^{-2}(a\o) S^{-1}\circ\ast\circ S^{-1}(h) \big)}\\
    =&\und{a\z{}^{*}\ra S^{-1}(a\o{})^{*} h^{*}}=\und{a\z{}^{*}\ra S(a\o{}^{*}) h^{*}}=\und{\varepsilon_{L}((a\ot h)^{\star})}.
\end{align*}
Moreover, we have on the one hand,
\begin{align*}
    (b\ot h)\o{}^{\star}\ot_{A} (b\ot h)\t{}^{\star}=&(b\ot h\o){}^{\star}\ot_{A} (1\ot h\t){}^{\star}\\
    =&b^{*}\z\ra S(b^{*}\o)h^{*}\o\ot h^{*}\t\ot_{A} 1\ot h^{*}\th,
\end{align*}
on the other hand,
\begin{align*}
    (b\ot& h)^{\star}\ro\ot_{A}(b\ot h)^{\star}\rt\\
    =&(b^{*}\z\ra S(b^{*}\o)h^{*}\o\ot h^{*}\t)\ro\ot_{A}(b^{*}\z\ra S(b^{*}\o)h^{*}\o\ot h^{*}\t)\rt\\
    =&1\ot h^{*}\t\o S^{-1}((b^{*}\z\ra S(b^{*}\o)h^{*}\o)\o)\ot_{A} (b^{*}\z\ra S(b^{*}\o)h^{*}\o)\z\ot h^{*}\t\t\\
    =&1\ot h^{*}\t\o S^{-1}\big(S(S(b^{*}\o)\o h^{*}\o\o)b^{*}\z\o S(b^{*}\o)\th h^{*}\o\th\big)\\
    &\ot_{A} b\z\z\ra S(b^{*}\o)\t h^{*}\o\t\ot h^{*}\t\t\\
    =&1\ot S(b^{*}\t)h^{*}\o\ot_{A} b^{*}\z\ra S(b^{*}\o)h^{*}\t\ot h^{*}\th\\
    =&1\ot S(b^{*}\o)h^{*}\o\ot_{A} (1\ot h^{*}\t)t_{R}(b^{*}\z)\\
    =&(1\ot S(b^{*}\o)h^{*}\o)s_{R}(b^{*}\z)\ot_{A} (1\ot h^{*}\t)\\
    =&b^{*}\z\ra S(b^{*}\th)h^{*}\o\ot b^{*}\o  S(b^{*}\t)h^{*}\t\ot_{A}1\ot h^{*}\th\\
    =&b^{*}\z\ra S(b^{*}\o)h^{*}\o\ot h^{*}\t\ot_{A}1\ot h^{*}\th.
\end{align*}
\end{proof}

\begin{example}\rm ($*$-structure for pair bialgebroid). A canonical example \cite{HM23} is what can be called the `pair Hopf algebroid' where $A=H$ itself right crossed module structure
\[   h\ra g:=(S g\o)h g\t,\quad h\z\tens h\o:=h\o\tens h\t\]
with resulting cross product $H\# H^{op}$ isomorphic to $H\tens H$ as an algebra. Here $H\tens H\to H\# H^{op}$ is $h\tens g\mapsto h g\o\tens Sg\t$. Using this isomorophism, the coproduct and other structures transfer back to $H\tens H$.  The new part is to suppose that $H$ is a Hopf $*$-algebra. Then the crossed-module structure is unitary:
\[ (h\ra g)^*=g^*\o h^* S^{-1}g^*\t=h^*\ra S^{-1}(g^*),\quad (h^*)\z\tens (h^*)\o=\Delta( h^*)=(*\tens *)\Delta h= (h\z)^*\tens (h\o)^*.\]
Hence the theory above applies.
\end{example}

\begin{example}\rm ($*$-structure for Weyl bialgebroids) If $H$ is finite-dimensional (or more generally dually paired with a Hopf algebra in the role of $H^*$, we have
\[ a\ra h= \<a\o,h\>a\t,\quad   \<a\z, b\>a\o= b\o a S b\t\]
for all $b\in H^*$. He cross product in this case is a version of the Weyl algebra of any Hopf algebra (also sometimes called the Heisenberg double). The new part is to check in the Hopf $*$-algebra case that the crossed-module structure is unitary:
\[ (a\ra h)^*=\overline{\<a\o,h\>}(a\t)^*= \<S^{-1}(a\o^*),h^*\>a\t^*= \<a\o^*,S^{-1}(h^*)\>a\t^*=a^*\ra S^{-1}(h^*)\]
using the duality pairing conventions $\<a,h^*\>=\overline{\<S^{-1}(a^*),h\>}$ in \cite{Ma:book} (which differs by $S^{-2}$ from the convention in \cite{BM}). For the coaction, applying $*$,
\begin{align*}
    a^*\z\<b^*,S^{-1}(a^*\o)\>=&S^{-1}(b^*\t)a^* b^*\o=(a\z \<b, a\o\>)^*\\
    =& (a\z)^*\<S^{-1} (b^*),(a\o)^*\>=(a\z)^*\<b^*,S^{-1} ((a\o)^*)\>
\end{align*}
for all $b$, from which we deduce that the coaction is a $*$-algebra map (this is a standard argument for the adjunction of a unitary action to a unitary coaction).
\end{example}

\subsection{Full $*$-Hopf algebroid for quantum principal bundles with classical fibre}\label{sec:fullES}

\begin{definition} \label{def:hg}
Let $H$ be a Hopf algebra, $P$ a $H$-comodule algebra with coaction $\delta_R$ and $B:= P^{coH}=\big\{b\in P ~|~ \delta_R (b) = b \ot 1_H \big\} \subseteq
P$ the coinvariant subalgebra.
The extension $B\subseteq P$ is called a \textup{$H$--Galois} if the
\textit{canonical map}
\begin{align*}
\chi := (m \ot \id) \circ (\id \ot _B \delta_R ) :
P \ot _B P \longrightarrow P \ot H ,
\quad q \ot_B p  &\mapsto q p\z \ot p\o
\end{align*}
is an isomorphism.
\end{definition}

Since the canonical map $\chi$ is left $P$-linear, its inverse is
determined by the restriction $\tau:=\chi^{-1}_{|_{1_P \ot H}}$, named \textit{translation map},
\begin{eqnarray*}
\tau =\chi^{-1}_{|_{1_P \ot H}} :  H\to P\ot _B P ~ ,
\quad h \mapsto \tau(h) = \tuno{h} \ot_B \tdue{h} \, .
\end{eqnarray*}
A Hopf-Galois extension is said to be  faithfully flat if $\tens_BP$ is faithfully flat, i.e. $P$ is flat as a left $B$-module. Then it is known, e.g. \cite[Lemma 34.4]{BW} that it obeys
\begin{align}\label{equ. translation map 1}
  \tuno{h} \ot_B \tdue{h}\z \ot \tdue{h}\o &= \,\tuno{h\o} \ot_B \tdue{h\o} \ot
h\t,\\
\label{equ. translation map 2}
~~ \tuno{h\t}  \ot_B \tdue{h\t} \ot S h\o&= \tuno{h}\z \ot_B {\tdue{h}}  \ot \tuno{h}\o,\\
\label{equ. translation map 3}
\tuno{h}\tdue{h}\z\ot \tdue{h}\o &= 1_{P} \ot h,\\
\label{equ. translation map 4}
    p\z\tuno{p\o}\ot_{B}\tdue{p\o} &= 1_{P} \ot_{B}p,\\
\label{equ. translation map 5}    \tuno{h\o}\ot_{B}\tdue{h\o}\tuno{h\t}\ot_{B}\tdue{h\t} &=\tuno{h}\ot_{B}1\ot_{B}\tdue{h}.
\end{align}
for all $h\in H$ and $p\in P$.

\begin{lemma}\label{lem. *-Hopf Galois extension}
    Let $B\subseteq P$ be a $H$-Galois extension. If $P$ is a $*$-algebra and $H$ is a Hopf $*$-algebra such that the coaction is unitary  then
    \[h^{*}\tuno{}\ot_{B}h^{*}\tdue{}=S^{-1}(h)\tdue{}{}^{*}\ot_{B}S(h)\tuno{}{}^{*},\]
    for any $h\in H$.
\end{lemma}

\begin{proof}
    We can check that
    \begin{align*}
        \chi(S^{-1}(h)\tdue{}{}^{*}\ot_{B}S^{-1}(h)\tuno{}{}^{*})=&S^{-1}(h)\tdue{}{}^{*}S^{-1}(h)\tuno{}{}^{*}\z\ot S^{-1}(h)\tuno{}{}^{*}\o\\
        =&S^{-1}(h)\tdue{}{}^{*}S^{-1}(h)\tuno{}\z{}^{*}\ot S^{-1}(h)\tuno{}\o{}^{*}\\
        =&(S^{-1}(h)\tuno{}\z{}S^{-1}(h)\tdue{})^{*}\ot S^{-1}(h)\tuno{}\o{}^{*}\\
        =&(S^{-1}(h)\t\tuno{}S^{-1}(h)\t\tdue{})^{*}\ot S(S^{-1}(h)\o){}^{*}\\
        =&1\ot h^{*}.
    \end{align*}
\end{proof}

By \cite{HS24}, we can also write $h^{*}\tuno{}\ot_{B}h^{*}\tdue{}=h^{[2]*}\ot_{B}h^{[1]*}$, where $h^{[1]}\ot_{B}h^{[2]}$ is the image of anti-right translation map. Recall that \cite{schau4},
\begin{definition}\label{def:ec}
Let $B=P^{co H}\subseteq P$ be a faithfully flat Hopf Galois extension. The \textup{Ehresmann-Schauenburg Hopf algebroid} is
\begin{equation*}
\cL(P, H)=(P\ot P)^{co(H)} := \{p\ot  q\in P\ot  P \, \quad|\quad p\z\ot  q\z \ot  p\o q\o=p\ot  q\ot  1_H \}. \label{ec2}
\end{equation*}
 with $B$-bimodule inherited from $P$ and $B$-coring coproduct and counit
 \[
\Delta_{L}(p\otimes q)=p\z\otimes p\o\tuno{}\ot_{B}p\o\tdue{}\otimes q,\quad  \varepsilon_{L}(p\otimes q)=pq.
\]
Moreover, $\cL(P, H)$ is a $B^e$-ring with the product and unit
\[ (p\otimes q)(r\otimes u)=pr\otimes uq,\quad \eta_{L}(b\tens c)=b\tens c
\]
for all $p\otimes q$, $r\otimes u \in \cL(P, H)$ and $b\tens c\in B^e$. Here $s_{L}(b)=b\otimes 1$ and $t_{L}(b)=1\otimes b$. \end{definition}
For any $p\ot q\in \cL(P, H)$, we also have $p\z\ot S(p\o)\ot q=p\ot q\o\ot q\z$. In general, $\cL(P, H)$ is not a full Hopf algebroid, but only a left Hopf algebroid\cite{HM22}. In addition, by \cite{DLZ24}, if $H$ is commutative, $\cL(P, H)$ is  a full Hopf algebroid with antipode given by
$S(p\ot q)=q\ot p$ for any $p\ot q\in \cL(P,H)$.

\begin{proposition}\label{propEScom}
    Let $B\subseteq P$ be a faithfully flat $H$-Galois extension, if $P$ is a $*$-algebra and $H$ is a commutative Hopf $*$-algebra such that the coaction is unitary, then
   $\cL(P, H)$ is a full $*$-Hopf algebroid. More precisely,
   \[(p\ot q)^{\star}=p^{*}\ot q^{*},\]
   for any $p\ot q\in \cL(P, H)$.
\end{proposition}
\begin{proof}
 Let $p\ot q\in \cL(P, H)$, then $p^*\ot q^*\in \cL(P, H)$
 since $q\ot p\in \cL(P, H)$ and the coaction is unitary. By (\ref{equ. left source target maps to right source and target maps}), (\ref{equ. left coproduct to right coproduct}) and (\ref{equ. left counit to right counit}), we can see
 \begin{align*}
     s_{R}(b)=1\ot b,\quad t_{R}(b)=b\ot 1, \quad \varepsilon_{R}(p\ot q)=qp,
 \end{align*}
 and
 \begin{align*}
     \Delta_{R}(p\ot q)=&S(S^{-1}(p\ot q)\t)\ot_{B^{op}} S(S^{-1}(p\ot q)\o)\\
     =&S(q\o\tdue{}\ot p)\ot_{B^{op}} S(q\z\ot q\o\tuno{})\\
     =&p\ot q\o\tdue{}\ot_{B^{op}} q\o\tuno{}\ot q\z.
 \end{align*}
 Clearly, $s_{L}(b)^{\star}=t_{R}(b^{*})$, $t_{L}(b)^{\star}=s_{R}(b^{*})$ and $\varepsilon_{L}(p^*\ot q^*)=p^{*}q^{*}=\varepsilon_{R}(p\ot q)^{*}$.
 We can also see
 \begin{align*}
     \Delta_{R}((p\ot q)^{\star})=&p^{*}\ot q^{*}\o\tdue{}\ot_{B^{op}} q^{*}\o\tuno{}\ot q^{*}\z\\
     =&p^{*}\ot S(q)\o\tuno{}{}^{*}\ot_{B^{op}} S(q)\o\tdue{}{}^{*}\ot q\z{}^{*}\\
     =&p\z{}^{*}\ot p\o\tuno{}{}^{*}\ot_{B^{op}} p\o\tdue{}{}^{*}\ot q{}^{*}\\
     =&(\star\ot \star)\circ \Delta_{L}(p\ot q).
 \end{align*}
 where the 2nd step uses Lemma \ref{lem. *-Hopf Galois extension}, the 3rd step uses the fact that $H$ is commutative.
\end{proof}

\subsubsection{$\theta$-deformation of $SU(2)$-fibration}\label{secthetasph}

In this subsection, we first present an example of Hopf Galois extension with classical fibre given in \cite{LS04}. More precisely, the total space is   the algebra $\CA(S_{\theta}^{7})$ of polynomial functions on the noncommutative sphere $S_{\theta}^{7}$, which is generated by elements
$z_a, z_a^*$, $a=1,\dots,4$, subject to relations
\begin{align}
    \label{s7t}
z_a z_b = \lambda_{a b} \, z_b z_a, \quad  z_a z_b^* = \lambda_{b a} \, z_b^* z_a,
\quad z_a^*z_b^* = \lambda_{a b} \, z_b^* z_a^* ,
\end{align}
and with the spherical relation $\sum_a z_a^* z_a=1$, where $\lambda_{a b} = e^{2 \pi i \theta_{ab}}$ and $(\theta_{ab})$ a real antisymmetric matrix given by
\begin{align*}
\theta_{ab}=\frac{\theta}{2}\begin{pmatrix} 0 & 0 & -1 & 1 \\
0 & 0 & 1 & -1 \\
1 & -1 & 0 & 0 \\
-1 & 1 & 0 & 0  \end{pmatrix}.
\end{align*}
The Hopf algebra is $\mathcal{A}(SU(2))$ is an unital complex $*$-algebra generated by $\omega_{1}, \overline{\omega}_{1}, \omega_{2}, \overline{\omega}_{2}$ subject to the relation $\omega_{1}\overline{\omega}_{1}+\omega_{2}\overline{\omega}_{2}=1$. The coproduct, counit and antipode is given by:
\begin{align*}
\Delta : \begin{pmatrix} \omega_{1} &-\overline{\omega}_{2}\\\omega_{2} & \overline{\omega}_{1} \end{pmatrix}\mapsto  \begin{pmatrix} \omega_{1} &-\overline{\omega}_{2}\\\omega_{2} & \overline{\omega}_{1} \end{pmatrix} \otimes  \begin{pmatrix} \omega_{1} &-\overline{\omega}_{2}\\\omega_{2} & \overline{\omega}_{1} \end{pmatrix},
\end{align*}
with counit $\varepsilon(\omega_{1})=\varepsilon(\overline{\omega}_{1})=1$, $\varepsilon(\omega_{2})=\varepsilon(\overline{\omega}_{2})=0$ and antipode $S(\omega_{1})=\overline{\omega}_{1}$, $S(\omega_{2})=-\omega_{2}$.  If we denote  $\CA(S_{\theta}^{7})$ and $\CA(SU(2))$ by matrix-valued function by
\begin{align*}
    \Psi  =
\begin{pmatrix}
z_1 & - z^*_2 \\
z_2 & z^*_1 \\
z_3 & -z^*_4 \\
z_4& z^*_3
\end{pmatrix},\qquad
\omega=\begin{pmatrix}
    \omega_{1} &-\overline{\omega}_{2}\\\omega_{2} & \overline{\omega}_{1}
\end{pmatrix},
\end{align*}
then the coaction can be written as $\delta(\Psi)=\Psi\ot\omega$. This means in components $\delta(\Psi_{ik})=\Psi_{ij}\ot\omega_{jk}$.

The algebra $\CA(S_{\theta}^{4})$ generated by   $\zeta_1=  z_1 z^{*}_3 + z^*_2 z_4$, $\zeta_2 =  z_2 z^*_3 - z^*_1 z_4$ and $\zeta_0 = z_1 z^*_1 + z^*_2 z_2 = 1 - z_3 z^*_3 - z^*_4 z_4$
  is the subalgebra of coinvariant. Moreover, $\CA(S_{\theta}^{4})\subseteq\CA(S_{\theta}^{7})$ is a $\CA(SU(2))$-Galois extension. The corresponding translation map $\tau$ is given by\cite{HLL22}
\begin{align*}
    \tau(\Psi)=\Psi\ot_{\CA(S_{\theta}^{4})}\Psi^{\dagger},\quad \mathrm{or}\quad\tau(\Psi_{ik})=\Psi_{ij}\ot_{\CA(S_{\theta}^{4})}\Psi^{\dagger}_{jk}.
\end{align*}
We can define two projections by
\begin{align}
    p = \Psi \cdot \Psi^\dagger =
\begin{pmatrix}
\zeta_0 & 0 & \zeta_1 & - \bar{\mu} \zeta_2^* \\
0 & \zeta_0 & \zeta_2  & \mu \zeta_1^* \\
\zeta_1^*& \zeta_2^* & 1-\zeta_0 & 0\\
-\mu \zeta_2 & \bar{\mu}   \zeta_1 & 0 & 1-\zeta_0
\end{pmatrix},\qquad q =  \Psi \cdot_{op} \Psi^\dagger =
\begin{pmatrix}
\zeta_0 & 0 & \bar{\mu} \zeta_1 & - \zeta_2^* \\
0 & \zeta_0 &  \mu \zeta_2  & \zeta_1^* \\
 \mu \zeta_1^*& \bar{\mu} \zeta_2^* & 1- \zeta_0 & 0\\
- \zeta_2 & \zeta_1 & 0 & 1- \zeta_0
\end{pmatrix}.
\end{align}

Let $P=\CA(S_{\theta}^{7})$, $B=\CA(S_{\theta}^{4})$ and $H=\CA(SU(2))$, $\cL(P,H)$ is generated
by the entries of the tensor products $p \ot 1$, $1\ot {q}$ and $V:=\Psi\ot \Psi^{\dagger}$ by \cite{HLL22}. As a coring, the coproduct is given by
\begin{align*}   \Delta(V)=V\ot_{B}V=\Psi\ot \Psi^{\dagger}\ot_{B}\Psi\ot \Psi^{\dagger}.
\end{align*}
The antipode is the flip. More precisely, $S(V_{ik})=S(\Psi_{ij}\ot \Psi^{\dagger}_{jk})=\Psi^{\dagger}_{jk}\ot \Psi_{ij}$. Moreover, the $\star$-structure is given by
\begin{align*}
    V_{ik}^{\star}=\Psi_{ij}^{*}\ot\Psi^{T}_{jk}.
\end{align*}

\subsection{Bar-category of modules of a full $*$-Hopf algebroid}\label{sec:bar}

The central notion behind our approach to $*$-structures is that of a bar category\cite{BM09}:

\begin{definition}
A strict monoidal category $(\CC, \ot, 1_{\CC})$ is called  a bar category if there is a monoidal functor $bar:\CC\to \CC^{op}$, denoted $V\mapsto \overline{V}$ on objects and $\phi\mapsto\overline{\phi}$ on morphisms with  natural equivalences $\Upsilon_{V,W}:\overline{V\ot W}\to \overline{W}\ot\overline{V}$ and $\star: 1_{\CC}\mapsto \overline{1_{\CC}}$, such that
\begin{enumerate}
    \item[(a)] $(\id\ot\Upsilon_{U,V})\circ\Upsilon_{U\ot V, W}=(\Upsilon_{V,W}\ot \id)\circ \Upsilon_{U,V\ot W}:\overline{U\ot V\ot W}\to \overline{W}\ot \overline{V}\ot \overline{U}$.
    \item[(b)] natural equivalence $bb$ between the identity and the $bar\circ bar$ functors on $\CC$, such that $\overline{\star}\circ \star=\id_{1_{\CC}}:1_{\CC}\to \overline{\overline{1_{\CC}}}$ and $\overline{bb_{V}}=bb_{\overline{V}}:\overline{V}\to \overline{\overline{\overline{V}}}$.
\end{enumerate}
Given $(\CC, \ot, 1_{\CC})$ and $(\CD, \ot, 1_{\CD})$ two bar categories, a bar functor $F: \CC\to \CD$  is a monoidal functor  together with a natural equivalence from $bar\circ F$ to $F\circ bar$, i.e. $fb_{Y}:\overline{F(Y)}\to F(\overline{Y})$, such that
\begin{enumerate}
    \item For the unit, $F(\overline{1})=\overline{F(1)}$, and for the $\star:1\to \overline{1}$ morphism, $F(\star)=\star$;
    \item $F(bb_{Y})=fb_{\bar{Y}}\circ\overline{fb_{Y}}\circ bb_{F(Y)}$.
    \item The following diagram commutes:
     \[
\begin{tikzcd}
  &\overline{F(X\ot Y)} \arrow[d, "fb_{X\ot Y}"] \arrow[r, "\overline{F^{-1}_{X, Y}}"] & \overline{F(X)\ot F(Y)}\arrow[r, "\Upsilon"] &\overline{F(Y)}\ot \overline{F(X)}\arrow[d, "fb_{Y}\ot fb_{X}"]&\\
   & F(\overline{X\ot Y})   \arrow[r, "F(\Upsilon)"] & F(\overline{Y}\ot\overline{X})\arrow[r, "F^{-1}_{\overline{Y},\overline{X}}"]&F(\overline{Y})\ot F(\overline{X}) .&
\end{tikzcd}
\]
\end{enumerate}

\end{definition}
 We also recall that a bar category $(\CC, \ot, 1_{\CC})$ is braided with antireal braiding if it is also a braided category and the braiding $\sigma$ makes the following diagram commute: \[
\begin{tikzcd}
  &\overline{X}\ot \overline{Y} \arrow[d, "\Upsilon^{-1}"] \arrow[r, "\sigma"] & \overline{Y}\ot\overline{X} \arrow[d, "\Upsilon^{-1}"] &\\
   & \overline{Y\ot X}   \arrow[r, "\overline{\sigma^{-1}}"] & \overline{X\ot Y}. &
\end{tikzcd}
\]

It is known\cite{BM09} that if  $A$ be a $*$-algebra then the category of $A$-bimodules is a bar category, which we recall as follows. Let $V$ be a $A$-bimodule. We
take its conjugate $\overline{V}$ to be the same as $V$ as a set, writing its element as $\overline{v}\in \overline{V}$ for any $v\in V$ with $\overline{v}+\overline{u}=\overline{v+u}$. The $A$-bimodule structure on $\overline{V}$ is given by $a\overline{v}a'=\overline{a'^{*}va^{*}}$. A linear map $T:V\to W$ gives another linear map $\overline{T}:\overline{V}\to \overline{W}$ by $\overline{T}(\overline{v})=\overline{T(v)}$. The map $\Upsilon$ is
\[ \Upsilon_{V\ot_{A} W}(\overline{V\ot_{A} W})=\overline{W}\ot_{A}\overline{V},\quad \Upsilon_{V\ot_{A} W}(\overline{v\ot_{A} w})=\overline{w}\ot_{A}\overline{v}\]
 which factors through the balanced tensor product, and $bb_{V}:V\to \overline{\overline{V}}$ is given by $bb_{V}(v)=\overline{\overline{v}}$.
It is also known\cite{BM09} that if $H$ is a Hopf $*$-algebra then its category of  $H$-comodules is a bar category. Similarly, we have

\begin{proposition}\label{propfullbar} Let $\CH=(\cL, \cR, S)$ be a full $*$-Hopf algebroid (with $\cL$ a left $B$-bialgebroid).

  (1) The category of right $\CH$-comodules $\mathcal{M}^{\CH}$ is a bar category with $\star$, $\Upsilon$, $bb$ as for the underlying $A$-bimodule structure of $\Gamma\in \CM^{\CH}$. The conjugate coactions on $\bar\Gamma$ are
\[\delta_{R}(\overline{\rho})=\overline{\rho\z}\ot_{A}\rho\o{}^{\star},\quad \delta_{L}(\overline{\rho})=\overline{\rho\rz}\ot_{B}\rho\ro{}^{\star},\]
for all $\rho\in \Gamma\in \CM^\CH$.

(2) The category of right $\CR$-modules $\mathcal{M}_{\CR}$ is a bar category with $\star$, $\Upsilon$, $bb$ as for the underlying $A$-bimodule  and the conjugate action
    \[\overline{\rho}\ra X=\overline{\rho\ra S^{-1}(X^{\star})},\]
    for any $\rho\in \Gamma\in \CM_{\CR}$.

    (3) Similarly,  category of left $\CL$-modules ${}_\CL\mathcal{M}$ is a bar category with $\star$, $\Upsilon$, $bb$ as for the underlying $B=A^{op}$-bimodule  and the conjugate action
    \[X\la \overline{\rho} =\overline{S(X^{\star})\la \rho},\]
    for any $\rho\in \Gamma\in {}_\CL\CM$.

\end{proposition}
\begin{proof} (1) By the same reason as Lemma \ref{lem. star exchange the balanced tensor product}, the coactions are well defined. Moreover, for any $a,a'\in B^{op}=A$
    \begin{align*}
        \delta_{R}(a\overline{\rho}a')=&\delta_{R}(\overline{a'^{*}\rho a^*})=\overline{\rho\z}\ot_{A}(s_{R}(a'^{*})\rho\o s_{R}(a^*))^\star\\
        =&\overline{\rho\z}\ot_{A}t_{L}(\und{a})\rho\o{}^\star t_{L}(\und{a'})=\overline{\rho\z}\ot_{A}s_{R}(a)\rho\o{}^\star s_{R}(a').
    \end{align*}
    Similarly, $\delta_{L}$ is $A$-bilinear as well. To see $(\id\ot_{B}\Delta_{R})\circ\delta_{L}=(\delta_{L}\ot_{A}\id)\circ\delta_{R}$, we have
    \begin{align*}
        (\id\ot_{B}\Delta_{R})\circ\delta_{L}(\overline{\rho})=&\overline{\rho\rz}\ot \rho\ro{}^{\star}\ro\ot \rho\ro{}^{\star}\rt=\overline{\rho\rz}\ot \rho\ro\o{}^{\star}\ot \rho\ro\t{}^{\star}\\
        =&\overline{\rho\z\rz}\ot \rho\z\ro{}^{\star}\ot \rho\o{}^{\star}=(\delta_{L}\ot_{A}\id)\circ\delta_{R}(\overline{\rho}).
    \end{align*}
    Similarly, $(\id\ot_{A}\Delta_{L})\circ\delta_{R}=(\delta_{R}\ot_{B}\id)\circ\delta_{L}$. So $\overline{\Gamma}\in \CM^{\CH}$. It is a direct computation to check $\Upsilon_{\Gamma,\Lambda}$ and $bb_{\Gamma}$ is right $\cR$ and $\cL$-colinear for any $\Gamma, \Lambda\in \CM^{\CH}$ (which is also the same with the Hopf algebra case).

    (2)  First, we can see for any $a, a'\in B^{op}=A$
    \begin{align*}
        \overline{\rho}\ra s_{R}(a')t_{R}(a)=&\overline{\rho\ra S^{-1}(s_{R}(a')^{\star}t_{R}(a)^{\star})}=\overline{\rho\ra S^{-1}(t_{L}(\und{a'^*})s_{L}(\und{a^*}))}=\overline{\rho\ra t_{R}(a'^*)t_{L}(a^*)}\\
        =&\overline{a'^{*}\rho a^{*}}=a\overline{\rho}a'.
    \end{align*}
    where the 3rd step uses (\ref{equ. left source target maps to right source and target maps}). Also we can see
    \begin{align*}
        bb(\rho)\ra X=\overline{\overline{\rho}}\ra X=\overline{\overline{\rho}\ra S^{-1}(X^{\star})}=\overline{\overline{\rho\ra S^{-1}(S^{-1}(X^{\star})^{\star})}}=\overline{\overline{\rho\ra X}},
    \end{align*}
    where the last step uses Corollary \ref{coro. antipode with star}. Moreover,
    \begin{align*}
        \Upsilon((\overline{\rho\ot_{A}\eta})\ra X)=&\Upsilon(\overline{\rho\ra S^{-1}(X^{\star})\ro\ot_{A}\eta\ra S^{-1}(X^{\star})\rt })\\
        =&\Upsilon(\overline{\rho\ra S^{-1}(X^{\star}\t)\ot_{A}\eta\ra S^{-1}(X^{\star}\o) })\\
        =&\Upsilon(\overline{\rho\ra S^{-1}(X\rt{}^{\star})\ot_{A}\eta\ra S^{-1}(X\ro{}^{\star}) })\\
        =&\overline{\eta\ra S^{-1}(X\ro{}^{\star})}\ot_{A}\overline{\rho\ra S^{-1}(X\rt{}^{\star})}\\
        =&\Upsilon(\overline{\rho\ot_{A}\eta})\ra X.
    \end{align*}

    (3) This is similar to (2).
\end{proof}

These results justify the notion of a full $*$-Hopf algebroid and described its representation theory. One can prove many other parallel results to the Hopf case and the next section is relevant to the study of differential structures of Hopf algebroids to be studied elsewhere.

\subsection{Bar category of crossed or Drinfeld-Yetter modules}

\begin{definition}\label{def. right-right YD}\label{sec:cross}
    Let $\cR$ be a right bialgebroid over $A$, a right-right Yetter-Drinfeld module of $\cR$ is a right $\cR$-comodule  and a right $\cR$-module $\Lambda$, such that
 \[(\rho\ra X\rt)\rz\ot_{A}X\ro(\rho\ra X\rt)\ro=\rho\rz \ra X\ro \ot_{A}\rho\ro X\rt\]
 for all  $X\in \cR, \rho\in \Lambda$.  We denote the category of `right crossed' or right-right Yetter-Drinfeld modules of $\cR$ by $\mathcal{YD}_{\cR}^{\cR}$.
\end{definition}

Recall that\cite{XH23}
\begin{definition}
    Let $\CH=(\cL, \cR, S)$ be a full Hopf algebroid, a full right-right Yetter-Drinfeld module of $\CH$ is a right-right Yetter-Drinfeld module of $\cR$ as well as a right comodule of $\CH$, such that the underlying right $\cR$-comodule structure (of the right-right Yetter-Drinfeld module structure of $\cR$) is the same as the underlying right $\cR$-comodule structure (of the right comodule structure of $\CH$). We denoted the category of full right-right Yetter-Drinfeld modules of $\cR$ by $\mathcal{YD}_{\cR}^{\cH}$.
    \end{definition}

It is given in \cite{XH23} that, if $(\cL, \cR, S)$ is a full Hopf algebroid and $\Lambda\in \mathcal{YD}^{\cR}_{\cR}$, then the Yetter-Drinfeld condition is equivalent to
\begin{align}\label{equ. YD condition}
    (\eta\ra X)\rz\ot_{A}(\eta\ra X)\ro=\eta\rz\ra X\t\ro\ot_{A}S(X\o)\eta\ro X\t\rt
\end{align}
    for any $\eta\in \Lambda$. Moreover, we can show

\begin{proposition}\label{prop. YD module}
    Let $\CH=(\cL, \cR, S)$ be a full Hopf algebroid and $\Lambda\in \mathcal{YD}^{\cH}_{\cR}$, then we have
    \[(\eta\ra X)\z\ot_{B}(\eta\ra X)\o=\eta\z\ra X\rt\o\ot_{A}S(X\ro)\eta\o X\rt\t,\]
    for any $\eta\in\Lambda$ and $X\in \cR$.
\end{proposition}

\begin{proof}
First, the formulae on the right hand side is well defined, here we only check it factors through the balanced product $\eta\z\ot_{B}\eta\o$. Indeed,
\begin{align*}
    (a\eta\z)\ra X\rt\o\ot_{A}S(X\ro)\eta\o X\rt\t=&\eta\z\ra(t_{R}(a) X\rt\o)\ot_{A}S(X\ro)\eta\o X\rt\t\\
    =&\eta\z\ra(t_{R}(a) X\o\rt)\ot_{A}S(X\o\ro)\eta\o X\t\\
    =&\eta\z\ra( X\o\rt)\ot_{A}S(s_{R}(a)X\o\ro)\eta\o X\t\\
    =&\eta\z\ra( X\o\rt)\ot_{A}S(X\o\ro)s_{L}(a)\eta\o X\t.
\end{align*}
Now, we can check
    \begin{align*}
        (\eta&\ra X)\z\ot_{B}(\eta\ra X)\o\\
        =&(\eta\ra X)\z\rz \varepsilon_{R}((\eta\ra X)\z\ro)\ot_{B}(\eta\ra X)\o\\
        =&(\eta\ra X)\rz \varepsilon_{R}((\eta\ra X)\ro\o)\ot_{B}(\eta\ra X)\ro\t\\
        =&\eta\rz\ra X\t\ro  \varepsilon_{R}(S(X\o)\o\eta\ro\o X\t\rt\o)\ot_{B}S(X\o)\t\eta\ro\t X\t\rt\t\\
        =&\eta\z\rz\ra X\ro\t  \varepsilon_{R}(S(X\ro\o)\o\eta\z\ro X\rt\o)\ot_{B}S(X\ro\o)\t\eta\o X\rt\t\\
        =&\eta\z\rz\ra X\ro\t  \varepsilon_{R}(t_{R}\circ\varepsilon_{R}(S(X\ro\o)\o)\eta\z\ro X\rt\o)\ot_{B}S(X\ro\o)\t\eta\o X\rt\t\\
        =&(\varepsilon_{R}(S(X\ro\o)\o)\eta\z\rz)\ra X\ro\t  \varepsilon_{R}(\eta\z\ro X\rt\o)\ot_{B}S(X\ro\o)\t\eta\o X\rt\t\\
        =&\eta\z\rz \ra (t_{R}\circ\varepsilon_{R}(S(X\ro\o)\o)X\ro\t)  \varepsilon_{R}(\eta\z\ro X\rt\o)\ot_{B}S(X\ro\o)\t\eta\o X\rt\t\\
        =&\eta\z\rz \ra (t_{R}\circ\varepsilon_{R}(S(X\ro\o\rt))X\ro\t)  \varepsilon_{R}(\eta\z\ro X\rt\o)\ot_{B}S(X\ro\o\ro)\eta\o X\rt\t\\
        =&\eta\z\rz \ra (t_{R}\circ\varepsilon_{R}(S(X\ro\rt\o))X\ro\rt\t)  \varepsilon_{R}(\eta\z\ro X\rt\o)\ot_{B}S(X\ro\ro)\eta\o X\rt\t\\
        =&\eta\z\rz \ra (t_{R}\circ\varepsilon_{R}(X\ro\rt\m)X\ro\rt\p)  \varepsilon_{R}(\eta\z\ro X\rt\o)\ot_{B}S(X\ro\ro)\eta\o X\rt\t\\
        =&\eta\z\rz \ra (X\ro\rt)  \varepsilon_{R}(\eta\z\ro X\rt\o)\ot_{B}S(X\ro\ro)\eta\o X\rt\t\\
        =&\eta\z\rz \ra (X\o\ro\rt)  \varepsilon_{R}(t_{R}\circ\varepsilon_{R}(\eta\z\ro) X\o\rt)\ot_{B}S(X\o\ro\ro)\eta\o X\t\\
        =&\eta\z\rz \ra (s_{R}\circ\varepsilon_{R}(\eta\z\ro)X\o\ro\rt)  \varepsilon_{R}( X\o\rt)\ot_{B}S(X\o\ro\ro)\eta\o X\t\\
        =&(\eta\z\rz \varepsilon_{R}(\eta\z\ro))\ra(X\o\ro\rt)  \varepsilon_{R}( X\o\rt)\ot_{B}S(X\o\ro\ro)\eta\o X\t\\
        =&\eta\z\ra(X\o\ro\rt s_{L}\circ \varepsilon_{R}( X\o\rt))\ot_{B}S(X\o\ro\ro)\eta\o X\t\\
        =&\eta\z\ra X\rt\o\ot_{A}S(X\ro)\eta\o X\rt\t,
    \end{align*}
where the 10th step uses (\ref{equ. right Hopf and anti-right Hopf}), the 11st step uses the fact that $t_{R}\circ\varepsilon_{R}(X\m)X\p=X$ which is given in \cite{XH23}.
\end{proof}

 It is given in \cite{XH23} that, $\mathcal{YD}_{\cR}^{\CH}$ is a braided monoidal category with the braiding and its inverse given by
    \begin{align*}
        \sigma(\rho\ot_{A}\eta)=\eta\rz\ot_{A}\rho\ra \eta\ro,\quad \sigma^{-1}(\eta\ot_{A}\rho)=\rho\ra S^{-1}(\eta\o)\ot_{A}\eta\z,
    \end{align*}
    for any $\eta\in \Lambda\in \mathcal{YD}_{\cR}^{\CH}$ and $\rho\in \Gamma\in \mathcal{YD}_{\cR}^{\CH}$. Moreover, we can show

\begin{theorem}\label{thm. bar category of YD module}
      Let $\CH=(\cL, \cR, S)$ be a full $*$-Hopf algebroid, then the category $\mathcal{YD}_{\cR}^{\CH}$ is an anti-real bar category with coaction and action as in Proposition \ref{propfullbar}.
\end{theorem}

\begin{proof}
    Let $\Gamma\in \mathcal{YD}_{\CR}^{\CH}$, then we first need to show $\overline{\Gamma}\in \mathcal{YD}_{\CR}^{\CH}$,\begin{align*}
        (\overline{\rho}\ra X)\rz\ot_{A}(\overline{\rho}\ra X)\ro=&\overline{(\rho\ra S^{-1}(X^{\star}))\z}\ot_{A}(\rho\ra S^{-1}(X^{\star}))\o{}^{\star}\\
        =&\overline{\rho\z\ra S^{-1}(X^{\star})\o\rt}\ot_{A}(S(S^{-1}(X^{\star})\o\ro)\rho\o S^{-1}(X^{\star})\t)^{\star}\\
        =&\overline{\rho\z\ra S^{-1}(X^{\star}\rt\o)}\ot_{A}(X^{\star}\rt\t\rho\o S^{-1}(X^{\star}\ro))^{\star}\\
        =&\overline{\rho\z\ra S^{-1}(X\t\ro{}^{\star})}\ot_{A}(X\t\rt{}^{\star}\rho\o S^{-1}(X\o{}^{\star}))^{\star}\\
        =&\overline{\rho\z\ra S^{-1}(X\t\ro{}^{\star})}\ot_{A}S^{-1}(X\o{}^{\star})^{\star}\rho\o{}^{\star}X\t\rt\\
        =&\overline{\rho\z\ra S^{-1}(X\t\ro{}^{\star})}\ot_{A}S(X\o)\rho\o{}^{\star}X\t\rt\\
        =&\overline{\rho\z}\ra X\t\ro\ot_{A}S(X\o)\rho\o{}^{\star}X\t\rt\\
    \end{align*}
    where the 2nd step uses Proposition \ref{prop. YD module}.  To check $\mathcal{YD}_{\cR}^{\CH}$ is a anti-real bar category,   for any $\eta\in \Lambda\in \mathcal{YD}_{\cR}^{\CH}$ and $\rho\in \Gamma\in \mathcal{YD}_{\cR}^{\CH}$ we have
    \begin{align*}
        \Upsilon^{-1}\circ \sigma(\overline{\rho}\ot\overline{\eta})=&\Upsilon^{-1}(\overline{\eta}\rz\ot_{A}\overline{\rho}\ra\overline{\eta}\ro)
        =\Upsilon^{-1}(\overline{\eta\z}\ot_{A}\overline{\rho}\ra \eta\o{}^{\star})
        =\overline{\rho\ra S^{-1}(\eta\o)\ot_{A}\eta\z}\\
        =&\overline{\sigma^{-1}(\eta\ot_{A}\rho)}
        =\overline{\sigma^{-1}}\circ\Upsilon^{-1}(\overline{\rho}\ot_{A}\overline{\eta}).
    \end{align*}
    \end{proof}

\section{More general formulation of left-right $*$-bialgebroid pairs}\label{secpair}

Here we introduce a more general notion of $*$-structures for bialgebroids and Hopf algebroids which is still, however, sufficient to generate a bar category of modules.

\subsection{Left-right pairs of bialgebroids}

\begin{definition} \label{oru}
 Let $\CL$, $\CR$ be left and right bialgebroids over $A$.
The category ${}_\mathcal{L}\mathcal{J}{}_\mathcal{R}$ of \em{joint $\mathcal{L}$-$\mathcal{R}$ modules} consists of left $\mathcal{L}$-modules which are also right
$\mathcal{R}$-modules such that the  $\mathcal{L}$ and $\mathcal{R}$ forgetful functors agree (giving
$F:{}_\mathcal{L}\mathcal{J}{}_\mathcal{R} \to {}_A\mathcal{M}{}_A $). This condition is equivalent to
the compatibility conditions
\[
s_{L}(a) \la m = m\ra t_{R}(a), \quad
t_{L}(a) \la m = m\ra s_{R}(a)\ .
\]
\end{definition}

Note that this definition does not give an $\CL$-$\CR$-bimodule in general, as the left and right actions will not commute. We now suppose that the left and right actions are related.

\begin{proposition}\label{prop. star related module is bar category}
    Let $(\cL, \cR)$ be $*$-related as in Definition~\ref{pye}. The category ${}_{\cL}J_{\cR}$ is a bar category, with
    \[\overline{m}\ra Y=\overline{Y^{\circledast^{-1}}\la m},\quad X\la \overline{m}=\overline{m\ra X^{\circledast}},\]
    for any $X\in\cL$, $Y\in \cR$ and $m\in M\in {}_{\cL}J_{\cR}$.
\end{proposition}
\begin{proof}
    First, we can see
    \begin{align*}
        \overline{m}\ra t_{R}(a)=\overline{m}\ra t_{L}(a^*)^{\circledast}=\overline{t_{L}(a^*)\la m}=\overline{ma^*}=a\overline{m},
    \end{align*}
    and
    \begin{align*}
        s_{L}(a)\la \overline{m}=s_{R}(a^*)^{\circledast^{-1}}\la \overline{m}=\overline{m\ra s_{R}(a^*)}=\overline{ma^*}=a\overline{m}.
    \end{align*}
    So we have $s_{L}(a)\la \overline{m}=\overline{m}\ra t_{R}(a)$. Similarly, $t_{L}(a)\la \overline{m}=\overline{m}\ra s_{R}(a)$. Second, we can check
    \begin{align*}
        \overline{m}\ra (YZ)=\overline{(YZ)^{\circledast^{-1}}\la m}=\overline{Z^{\circledast^{-1}}\la(Y^{\circledast^{-1}}\la m)}=\overline{(Y^{\circledast^{-1}}\la m)}\ra Z=(\overline{m}\ra Y)\ra Z,
    \end{align*}
    for any $Y, Z\in \cR$. Similarly, $(XW)\la \overline{m}=X\la(W\la \overline{m})$, for any $X, W\in \cL$. To see this is a bar category,
    \begin{align*}
        \Upsilon(\overline{m\ot_{A}n}\ra Y)=&\Upsilon(\overline{Y^{\circledast^{-1}}\la (m\ot_{A}n)})=\Upsilon(\overline{Y^{\circledast^{-1}}\o\la m\ot_{A}Y^{\circledast^{-1}}\t\la n})\\
        =&\Upsilon(\overline{Y\rt{}^{\circledast^{-1}}\la m\ot_{A}Y\ro{}^{\circledast^{-1}}\la n})=\overline{Y\ro{}^{\circledast^{-1}}\la n}\ot_{A}\overline{Y\rt{}^{\circledast^{-1}}\la m}\\
        =&\overline{n}\ra Y\ro\ot_{A}\overline{m}\ra Y\rt=\Upsilon(\overline{m\ot_{A}n})\ra Y.
    \end{align*}
    Similarly, $\Upsilon(X\la\overline{(m\ot_{A}n)})=X\la\Upsilon(\overline{m\ot_{A}n})$. Moreover,
    \begin{align*}
        X\la bb(m)=X\la\overline{\overline{m}}=\overline{\overline{m}\ra X^{\circledast}}=\overline{\overline{X\la m}}.
    \end{align*}
    Similarly, $bb(m)\ra Y=\overline{\overline{m\ra Y}}$.
\end{proof}

Let $\cL$ be a left $A^{op}$-bialgebroid and $\cR$ be a right $A$-bialgebroid. Define $\CJ^{(\cL,\cR)}$ the the category of right $\cL$-comodule and right $\cR$-comodule, such that the underlying $A$-bimodules are the same.
Similarly to Proposition \ref{propfullbar} for full $*$-Hopf algebroids, we have:

\begin{proposition}
    Let $\cL$ be a left $A^{op}$-bialgebroid and $\cR$ be a right $A$-bialgebroid, such that $(\cL^{cop}, \cR)$ be $*$-related. $\CJ^{(\cL,\cR)}$ is a bar category with
    \[\delta_{R}(\overline{\rho})=\overline{\rho\z}\ot_{A}\rho\o{}^{\circledast},\quad \delta_{L}(\overline{\rho})=\overline{\rho\rz}\ot_{A}\rho\ro{}^{\circledast^{-1}},\]
for all $\rho\in \Gamma\in \CJ^{(\CL,\cR)}$.
\end{proposition}
\begin{proof}We can see that for any $a, a'\in A$
    \begin{align*}
        \delta_{R}(a\overline{\rho}a')=&\delta_{R}(\overline{a'^{*}\rho a^*})=\overline{\rho\z}\ot_{A}(t_{L}(\und{a'^{*}})\rho\o t_{L}(\und{a^*}))^\circledast
=\overline{\rho\z}\ot_{A}s_{R}(a)\rho\o{}^\circledast s_{R}(a').
    \end{align*}
    Similarly, $\delta_{L}(a\overline{\rho}a')=\overline{\rho\rz}\ot_{A}t_{L}(\und{a})\rho\ro{}^{\circledast^{-1}} t_{L}(\und{a'})$. It is straightforward computation to check $(\Upsilon\ot \id)\circ \delta_{R}=\delta_{R}\circ \Upsilon$, $(\Upsilon\ot \id)\circ \delta_{L}=\delta_{L}\circ \Upsilon$, and $\delta_{R}\circ bb=(bb\ot \id)\circ \delta_{R}$, $\delta_{L}\circ bb=(bb\ot \id)\circ \delta_{L}$.
\end{proof}

\begin{example} For the case of Example~\ref{fullcase} of a full Hopf algebroid, we can apply the above to $\CL^{cop},\CR$, then (a) the category $\CJ^{(\CL^{cop},\CR)}$ amounts  to Definitiion \ref{def. comodule of full Hopf algebroid}  without the 2nd conditon  (b) in the full $*$-Hopf algebroid, we have $\CL^{cop}$ and $\CR$ $*$-related and we have a bar category for the category (a).
\end{example}

\begin{proposition}\label{prop. reflexive module equivalence} \label{pux}
 Let $\CL,\CR$  be left and right bialgebroids over a base $A$. We say that the pair $(\CL,\CR)$ is {\em reflexive} if there is a map $\Phi:\mathcal{L} \to \mathcal{R}$  which is an invertible linear anitalgebra coalgebra map such that
\[ \Phi(1)=1,\quad \Phi\circ s_{L}=t_{R},\quad \Phi\circ t_{L}=s_{R}.\]
Then the equation
\[
X\la m = m\ra\Phi(X)
\]
gives a monoidal isomorphism  ${}_\mathcal{L}\mathcal{M}\cong\mathcal{M}{}_\mathcal{R}$ compatible with the respective forgetful functors to ${}_A\CM_A$.  Moveover,  ${}_\CL\CM, \CM_\CR\hookrightarrow {}_\CL\CJ_\CR$ are (isomorphic) full subcatgories, where the action of the other side is given by the displayed equation.
\end{proposition}
\noindent\textbf{Proof:}\quad Suppose that $M\in \mathcal{M}{}_\mathcal{R}$. We define a left action of $\mathcal{L}$ on $M$ by the displayed equation in the statement for $X\in \mathcal{L}$ and $m\in M$. Now
\[
X\la (Y\la m) = X\la (m\ra\Phi(Y))= m\ra\Phi(Y)\Phi(X)= m\ra\Phi(XY)=(XY)\la m
\]
and for $n\in N\in \mathcal{M}{}_\mathcal{R}$
\begin{align*}
X\la (n\tens m) &=  (n\tens m)  \ra \Phi(X) = n\ra \Phi(X)\ro \tens m\ra \Phi(X)\rt \cr
&= n\ra \Phi(X\o) \tens m\ra \Phi(X\t ) = X\o\la n\tens  X\t\la m\ .
\end{align*}
For the forgetful functors, we have
\[
s_{L}(a) \la m = m\ra \Phi(s_{L}(a) ) = m\ra t_{R}(a)
\]
so the two expressions for $a.m$ agree. On swapping $s$ and $t$ the two  expressions for $m.a$ agree.
As $\Phi$ is invertible we can define a right $\mathcal{R}$ action on a left $\mathcal{L}$ module $P$
by $p\ra z=\Phi^{-1}(z) \la p$, and the other way round follows. \hfill $\square$

\begin{definition} A $*$-bialgebroid pair $(\CL,\CR)$ over $A$ means both $*$-related by $\circledast$ and reflexive by $\Phi$, such that $\Phi^{-1}\circ \circledast=\circledast^{-1}\circ \Phi$.  A $*$-Hopf algebroid pair $(\CL,\CR)$ is the same when $\CL$ is a left and anti-left Hopf algebroid (or $\CR$ is a right and anti-right Hopf algebroid).  \end{definition}

A motivation behind this definition, aside from the fact that it results in a bar category as we shall see, is that in this case  $\theta:=\Phi^{-1}\circ \circledast:\CL\to \CL$ is then an antilinear algebra and anti-coalgebra map which plays the role of the composite $*\circ S$ in the case of an usual Hopf $*$-algebra. The two versions of the condition for a $*$-bialgebroid pair to be a $*$-Hopf algebroid pair are equivalent due to the reflexivity.

\begin{theorem}\label{thmpairbar}
    If $(\cL, \cR)$ is a $*$-bialgebroid pair then  ${}_\CL\CM$ and  $\CM_{\cR}$ are bar categories such their inclusion in ${}_\CL\CJ_\CR$ in Proposition \ref{prop. reflexive module equivalence} is as bar categories. Here if $m\in M\in \CM_\CR$,
    \[ \overline{m}\ra Y=\overline{m\ra \Phi(Y^{\circledast^{-1}})}.\]
    \end{theorem}
\begin{proof} As $\CL,\CR$ are $*$-related, we know that ${}_\CL\CJ_\CR$ is a bar category. As the pair is also reflexive and $\Phi$ is compatible with $*$, the bar category structure restricts to the image of  $\CM_\CR$, say (similarly for ${}_\CL\CM$). Thus, we consider $m\in M\in \CM_\CR$ viewed via $\Phi$ in ${}_\CL\CJ_\CR$ and check using the bar category structure there that
    \[\overline{m}\ra Y=\overline{Y^{\circledast^{-1}}\la m}=\overline{m\ra \Phi(Y^{\circledast^{-1}})}\]
 for the conjugate left action for all $Y\in \cR$ as stated.  There is also a conjugate left action as an object of ${}_{\CL}\CJ_\CR$ and we check that the result is still in the image of $\CM_\CR$,
    \[ X\la \overline{m}=\overline{m\ra X^{\circledast}}=\overline{\Phi^{-1}(X^{\circledast})\la m}=\overline{\Phi(X)^{\circledast^{-1}}\la m}=\overline{m}\ra \Phi(X).\]
for all $X\in\cL$.

One can also see directly that  $\CM{}_{\cR}$ is a bar category with the stated conjugate action. Indeed,
\begin{align*}
        \overline{m}\ra s_{R}(a')t_{R}(a)=&\overline{m\ra \Phi(s_{R}(a')^{\circledast^{-1}}t_{R}(a)^{\circledast^{-1}})}=\overline{m\ra \Phi(s_{L}(a'^*)t_{L}(a^*))}=\overline{\rho\ra t_{R}(a'^*)s_{R}(a^*)}\\
        =&\overline{a'^{*}m a^{*}}=a\overline{m}a'.
    \end{align*}
   Also we can see
    \begin{align*}
        bb(m)\ra Y=\overline{\overline{m}}\ra Y=\overline{\overline{m}\ra \Phi(Y^*)}=\overline{\overline{m\ra \Phi(\Phi(Y^{\circledast^{-1}})^{\circledast^{-1}})}}=\overline{\overline{m\ra Y}},
    \end{align*}
    and
    \begin{align*}
        \Upsilon((\overline{m\ot_{A}n})\ra Y)=&\Upsilon(\overline{m\ra \Phi(Y^{\circledast^{-1}})\ro\ot_{A}n\ra \Phi(Y^{\circledast^{-1}})\rt })\\
        =&\Upsilon(\overline{m\ra \Phi(Y^{\circledast^{-1}}\o)\ot_{A}n\ra \Phi(Y^{\circledast^{-1}}\t) })\\
        =&\Upsilon(\overline{m\ra \Phi(Y\rt{}^{\circledast^{-1}})\ot_{A}n\ra \Phi(Y\ro{}^{\circledast^{-1}}) })\\
        =&\overline{n\ra \Phi(Y\ro{}^{\circledast^{-1}})}\ot_{A}\overline{m\ra \Phi(Y\rt{}^{\circledast^{-1}})}\\
        =&\overline{n}\ra Y\ro\ot_{A}\overline{m}\ra Y\rt\\
        =&\Upsilon(\overline{m\ot_{A}n})\ra Y.
    \end{align*}
   Similarly, ${}_{\cL}\CM$ is also a bar category with
    \[X\la\overline{m}=\overline{\Phi^{-1}(X^\circledast)\la m},\]
    for any $x\in\cL$ and $m\in M\in {}_{\cL}\CM$. \end{proof}

\begin{example}\label{exfullpair}
   In the case of Remark~\ref{fullcase}, let $(\cL, \cR, S)$ be a full $*$-Hopf algebroid.  If we set $\Phi=S^{-1}$ and $\circledast=\star$ then $(\cL^{cop},\cR)$ is a $*$-Hopf algebroid pair. To prove this, note that
    $S^{-1}:\cL^{cop}\to \cR$ is an anitalgebra coalgebra map. We can also see, $S^{-1}\circ s_{L}^{cop}=S^{-1}\circ t_{L}=t_{R}$ and $S^{-1}\circ t_{L}^{cop}=S^{-1}\circ s_{L}=t_{L}=s_{R}$. Moreover, $\Phi$ is a coring map
    \begin{align*}
        S^{-1}(X\t)\ot S^{-1}(X\o)=S^{-1}(X)\ro\ot S^{-1}(X)\rt.
    \end{align*}
    As we know $\cL^{cop}$ and $\cR$ are $\star$-related and by Corollary~\ref{coro. antipode with star}, we have $S^{-1}\circ \star=\star\circ S$, so $(\cL^{cop}, \cR)$ is a $*$-Hopf algebroid pair.
\end{example}

\begin{proposition}\label{propES}
    Let $B\subseteq P$ be a faithfully flat $H$-Galois extension. If $P$ is a $*$-algebra and $H$ is a Hopf $*$-algebra with bijective antipode such that the coaction is unitary then $(\cL(P, H), \cL^{op}(P, H))$ is a $*$-Hopf algebroid pair, with
    \[(p\ot q)^{\circledast}=q^{*}\ot p^{*},\quad \Phi=\id,\]
    for any $p\ot q\in \cL(P, H)$.
\end{proposition}
\begin{proof}
    First, we can see $p\ot q\in \cL(P, H)$, $q^{*}\ot p^{*}\in \cL(P, H)$. Indeed,
    \begin{align*}
        \delta(q^{*}\ot p^{*})=&q\z{}^{*}\ot p\z{}^{*}\ot q\o{}^{*}p\o{}^{*}=q\z{}^{*}\ot p\z{}^{*}\ot (p\o q\o){}^{*}\\
        =&q^{*}\ot p^{*}\ot 1.
    \end{align*}
    We can see $\circledast$ is an anti-algebra map
    \begin{align*}
        ((p\ot q)\cdot_{\cL}(p'\ot q'))^{\circledast}=&(pp'\ot q'q)^\circledast=q^{*}q'^{*}\ot p'^{*}p^{*}=(q'^{*}\ot p'^{*})\cdot_{\cL^{op}}(q^{*}\ot p^{*})\\
        =&(p'\ot q')^{\circledast}\cdot_{\cL^{op}}(p\ot q)^{\circledast}.
    \end{align*}
   Moreover,
    \begin{align*}
        s_{L}(b)^\circledast=(b\ot 1)^\circledast=1\ot b^*=s_{L}^{op}(b^{*}),
    \end{align*}
    for any $b\in B$. Similarly, $t_{L}(b)^\circledast=t_{L}^{op}(b^*)$. And
    \begin{align*}
        \varepsilon((p\ot q)^\circledast)=q^* p^*=(pq)^*=(\varepsilon_{L}(p\ot q))^{*}=(\varepsilon_{L}^{op}(p\ot q))^{*}.
    \end{align*}
    Finally,
    \begin{align*}
        \Delta_{L}^{op}((p\ot q)^\circledast)=&\Delta_{L}((p\ot q)^\circledast)\\
        =&q^{*}\z\ot q^{*}\o\tuno{}\ot_{B}q^{*}\o\tdue{}\ot p^*\\
        =&q\z{}^{*}\ot S^{-1}(q\o)\tdue{}{}^{*}\ot_{B}S^{-1}(q\o)\tuno{}{}^{*}\ot p^*\\
        =&q^{*}\ot p\o\tdue{}{}^{*}\ot_{B}p\o\tuno{}{}^{*}\ot p\z{}^{*}\\
        =&(p\o\tdue{}\ot q)^\circledast\ot_{B}(p\z\ot p\o\tuno{})^\circledast\\
        =&\mathrm{flip}\circ (\circledast\ot\circledast)\circ\Delta_{L}(p\ot q).
    \end{align*}
    where the 3rd step uses Lemma \ref{lem. *-Hopf Galois extension}, the 4th step uses the fact that
    \[p\z\ot p\o\ot q=p\z\ot S^{-1}(p\t q\o)p\o\ot q\z=p\ot S^{-1}(q\o)\ot q\z.\]
    It is not hard to see $\Phi=\id$ is a reflexive structure by Proposition \ref{prop. oppo bialgebroid}. It is given by \cite{HM22} and \cite{HS25} that $\cL(P, H)$ is left and anti-left Hopf algebroid. Also, since $\id\circ \circledast=\circledast^{-1}\circ\id$, we have the result.
\end{proof}

To conclude this section we make some general remarks about $\CL,\CR$ pairs as above.

\begin{proposition}\label{propodot}  Given $\CL$, $\CR$ left and right bialgebroids over $A$, there is a left bialgebroid $\CL\odot \CR^{op}$ over $A$ such that ${}_{\CL\odot \CR^{op}}\CM={}_\CL\CJ_\CR$. We similarly define a right $A$-bialgebroid $\CL^{op}\und{\odot}\CR$ such that $\CM_{\CL^{op}\und{\odot}\CR}={}_\CL\CJ_\CR$. Moreover if   $(\cL, \cR)$ are $*$-related as in Definition~\ref{pye} then  there is a map
\[\circledast:\CL\odot\CR^{op} \to \CL^{op} \,\und{\odot} \, \CR\]
making these a $*$-bialgebroid pair with $\Phi$ the identity on the generators and given by the original $\circledast: \CL\to \CR$ and its inverse. Moreover, ${}_{\CL\odot \CR^{op}}\CM={}_\CL\CJ_\CR$ as bar category.
\end{proposition}
\proof (1) We write the product in $\CL$ as $\bullet$ and in $\cR$ as $\Rbullet$.
We define $\CL\odot\CR^{op}$ as a quotient of an amalgamated free product of $\CL$ and $\CR^{op}$ by the relations that $s_{\cL}(a)=t_{\cR}(a)$ and $t_{\cL}(a)=s_{\cR}(a)$ for any $a\in A$. To make this explicit, we replace the copy of $A^e$ in $\CR$ by writing $a\in A\subset \CR$ as $\und a\in A^{op}\subset \CL$
and $\und a\in A^{op}\subset \CR$ as $a\in A\subset \CL$.
We have the relations (using $\bullet$ for the product in $\CL\odot\CR^{op}$), for $Y_i\in\CR$,
$Y_1\bullet Y_2=Y_2\Rbullet Y_1$, the $\CL$ product is not affected.
The source and target is
\[s_{\CL\odot\CR^{op}}(a)=s_{\cL}(a)=a,\quad t_{\CL\odot\CR^{op}}(a)=t_{\cL}(a)=\und a.\]
The coproduct and counit on $\CL\odot \CR^{op}$ is the same as that on $\CL$ and $\CR$ on the respective generators, and on $X\bullet Y$ satisfies
\[\Delta_{\CL\odot\CR^{op}}(X\bullet Y)=X\o\bullet Y\o\ot_{A} X\t\bullet Y\t,\quad \varepsilon_{\CL\odot\CR^{op}}(X\bullet Y)=\varepsilon_{\CL}(X\,\varepsilon_{\CR^{op}}(Y)),\]
Clearly, $\cL\odot \cR^{op}$ is a $A^e$-ring and $A$-coring with the image of the coproduct belongs to the Takeuchi product. Moreover, the coproduct is an algebra map and the counit satisfies the axiom of a left bialgebroid. Hence, $\cL\odot \cR^{op}$ is a left bialgebroid over $A$.
The left module for $\CL\odot\CR^{op}$ corresponding to the joint module has $x\la m$ defined as usual for $x\in \CL$ and $y\la m=m\ra y$ for $y\in \CR$. Therefore, it is not hard to see ${}_{\CL\odot \CR^{op}}\CM={}_\CL\CJ_\CR$.

(2) We similarly define the right bialgebroid $\CL^{op}\und{\odot}\CR$ another quotient of an amalgamated free product of $\CL^{op}$ and $\CR$ by the relations that $s_{\cL}(a)=t_{\cR}(a)$ and $t_{\cL}(a)=s_{\cR}(a)$ for any $a\in A$. To make this explicit, we replace the copy of $A^e$ in $\CL$ by writing $a\in A\subset \CL$ as $\und a\in A^{op}\subset \CR$
and $\und a\in A^{op}\subset \CL$ as $a\in A\subset \CR$.
We have the relations (using $\Rbullet$ for the product in $\CL\odot\CR^{op}$), for $X_i\in\CL$,
$X_1\Rbullet X_2=X_2\bullet X_1$, the $\CR$ product is not affected.
The source and target is
\[s_{\CL^{op}\und{\odot}\CR}(a)=s_{\cR}(a)=a,\quad t_{\CL^{op}\und{\odot}\CR}(a)=t_{\cR}(a)=\und a.\]
The coproduct and counit on $\CL^{op}\und{\odot}\CR$ is the same as that on $\CL$ and $\CR$ on the respective generators, and on $X\bullet Y$ satisfies
\[\Delta_{\CL^{op}\und{\odot}\CR}(X\Rbullet Y)=X\o\Rbullet Y\o\ot_{A} X\t\Rbullet Y\t,\quad \varepsilon(X\Rbullet Y)=\varepsilon_{\CR}(\varepsilon_{\CL^{op}}(X)\,Y).\]

(3) Now we suppose that $(\CL,\CR)$ are $*$-related by $\circledast:\CL\to \CR$ and show that $(\CL\odot\CR^{op},\CL^{op}\und{\odot}\CR)$ is a $*$-bialgebroid pair. For the purposes of the proof, to avoid confusion, we denote the desired $\circledast$ on the product by a different symbol $ \divideontimes:\CL\odot\CR^{op}\to \CL^{op}\und{\odot}\CR$ and set $X^{ \divideontimes}=X^{\circledast}, Y^{ \divideontimes}=Y^{\circledast^{-1}}$ as stated, which obeys $(X\bullet Y)^{\divideontimes}=Y^{\circledast^{-1}}\Rbullet X^{\circledast}$. We set  $\Phi:\CL\odot\CR^{op}\to \CL^{op}\und{\odot}\CR$ to be the identity on the vector space $\CL\odot\CR^{op}$. Clearly, $\divideontimes$ an invertible antilinear anti-algebra map. Also,
\[s_{\CL^{op}\und{\odot}\CR}(a^*)=s_{\cR}(a^*)=s_{\cL}(a)^{\circledast}=s_{\CL\odot\CR^{op}}(a)^{\divideontimes}.\]
Similarly, $t_{\CL^{op}\und{\odot}\CR}(a^*)=t_{\CL\odot\CR^{op}}(a)^{\divideontimes}.$ Moreover,
\[\varepsilon_{\CL^{op}\und{\odot}\CR}(X^{\divideontimes})=\varepsilon_{\cR}(X^\circledast)=\varepsilon_{\cL}(X)^*= \varepsilon_{\CL\odot\CR^{op}}(X)^*.\]
Similarly, $\varepsilon_{\CL^{op}\und{\odot}\CR}(Y^{\divideontimes})=\varepsilon_{\cL}(Y^{\circledast^{-1}})=\varepsilon_{\cR}(Y)^*= \varepsilon_{\CL\odot\CR^{op}}(X)^*$. Also, it is not hard to see
$\mathrm{flip}(\divideontimes \tens \divideontimes)\Delta_\mathcal{\CL\odot\CR^{op}} =\Delta_\mathcal{\CL^{op}\und{\odot}\CR}\circ \divideontimes$. So $\divideontimes$ make the bialgebroids above $*$-related. We can also see the $\Phi$ being the identity on the vector space $\CL\odot\CR^{op}$ is a anti-algebra map and a coalgebra map. Moreover, $\Phi^{-1}\circ \divideontimes=\divideontimes^{-1}\circ \Phi$. As $(\cL, \cR)$ is $*$-related we know that  ${}_\CL\CJ_\CR$ is a bar category by Proposition \ref{prop. star related module is bar category} and ${}_{\CL\odot\CR^{op}}\CM$ is a bar categories by Proposition \ref{thmpairbar}.  We can see ${}_{\CL\odot \CR^{op}}\CM={}_\CL\CJ_\CR$ as bar categories. Indeed,
\begin{align*}
    X\la \overline{m}=\overline{\Phi^{-1}(X^{\divideontimes})\la m}=\overline{X^\circledast\la m}=\overline{ m\ra X^\circledast},
\end{align*}
for any $m\in M\in {}_{\CL\odot \CR^{op}}\CM$, where the last step use correspondence between ${}_{\CL\odot \CR^{op}}\CM$ and ${}_\CL\CJ_\CR$. Similarly,
\begin{align*}
     Y\la \overline{m}=\overline{\Phi^{-1}(Y^{\divideontimes})\la m}=\overline{Y^{\circledast^{-1}}\la m}=\overline{ m\ra Y^{\circledast^{-1}}}
\end{align*}
as required. \endproof

We mention that in the $*$-related case we can also define a $*$-algebra structure on $\CL\odot\CR^{op}$ by $X^\circledast$ and $Y^{\circledast^{-1}}$ which is an antilinear, antialgebra and coalgebra map.

\subsection{$*$-Hopf algebroid pair for action bialgebroids}\label{sec:pairaction}

By a similar construction to Section~\ref{sec:fullaction}\cite{Lu}, one has from the same data of a braided-commutative algebra $B\in\CM^H_H$ a right $B$-Hopf algebroid  $\cR=H\#B$ with
\[ (h\tens a)(g\tens b)=h g\o\tens (a\ra g\t) b,\quad s_{R}(a)=1\tens a,\quad t_{R}(b)=S^{-1}b\o\tens b\z\]
\[ \varepsilon_{R}(h\tens a)=\varepsilon(h)a,\quad \Delta_{R}(h\tens a)=h\o\tens 1\tens h\t\tens a.\]
We can also view the right coaction as a left action $a\mapsto S^{-1}a\o\tens a\z$. We already saw in Theorem \ref{thmact}  that $\cL=B\#H^{op}$ is part of a full Hopf algebroid.

\begin{theorem}\label{thmactpair}Let $B$ be a braided-commutative algebra in $\CM^H_H$ in the $*$-algebra and Hopf $*$-algebra case with unitary (co)actions. Then $(B\# H^{op},H\#B)$ as above is a $*$-Hopf algebroid pair with
\[ (a\tens h)^\circledast=S^{-1}(h^*)\tens a^*,\quad  \Phi(a\tens h)=h S^{-1}a\o\tens a\z,\quad \Phi^{-1}(h\tens a)=a\z\tens h a\o.  \]
\end{theorem}
\proof For $\circledast$,
\begin{align*}
((a\tens h)(b\tens g))^\circledast&= S^{-1}((g h\t)^*)\tens (a (b\ra h\o))^*=S^{-1}(g^*)S^{-1}(h\t^*)\tens (b^*\ra S^{-1}(h\o^*))a^*\\
&=S^{-1}(g^*)S^{-1}(h^*)\o\tens (b^*\ra S^{-1}(h^*)\t)a^*=(S^{-1}(g^*)\tens b^*)(S^{-1}(h^*)\tens a^*)\\
&=(b\tens g)^\circledast (a\tens h)^\circledast\\
\Delta_{R}((a\tens h)^\circledast)&=\Delta_{R}(S^{-1}(h^*)\tens a^*)=S^{-1}(h^*)\o\tens 1\tens S^{-1}(h^*)\t\tens a^*\\
&=S^{-1}(h^*\t)\tens 1\tens S^{-1}(h^*\o)\tens a^*={\rm flip}(\circledast\tens\circledast)(a\tens h\o\tens 1\tens h\t)\\
&={\rm flip}(\circledast\tens\circledast)\Delta_{L}(a\tens h).
\end{align*}
We also have
\[ s_{L}(a)^\circledast=(a\tens 1)^\circledast=1\tens a^*=s_{R}(a^*),\quad t_{L}(a)^\circledast=(a\z\tens a\o)^\circledast=S^{-1}(a^*\o)\tens a^*\z=t_{R}(a^*).\]
For $\Phi$,
\begin{align*}
\Phi(&(a\tens h)(b\tens g))=\Phi(a(b\ra h\o)\tens g h\t)=g h\t S^{-1}(a(b\ra h\o)\o)\tens (a (b\ra h\o))\z\\
&= g S^{-1}((b\ra h\o)\o Sh\t)S^{-1}a\o\tens a\z (b\ra h\o)\z= g S^{-1}(Sh\o b\o)S^{-1}a\o\tens a\z (b\z\ra h\t)\\
&= g (S^{-1}b\o)h\o S^{-1}a\o\t\tens (b\z\ra(h\t S^{-1}a\o\o))a\z=(g S^{-1} b\o\tens b\z)(h\t S^{-1}a\o\tens a\z)\\
&=\Phi(b\tens g)\Phi(a\tens h),
\end{align*}
where the 3rd equality uses the crossed module condition and the 5th uses braided commutativity. Next,
\begin{align*}\Delta_{R}\Phi(a\tens h)&=\Delta_{R}(h S^{-1}a\o\tens a\z)=(h S^{-1}a\o)\o\tens 1\tens (h S^{-1}a\o)\t\tens a\z\\
&=h\o S^{-1}a\o\tens 1\tens h\t S^{-1}a\z\o\tens a\z\z= h\o S^{-1}a\o\tens 1\tens (h\t \tens 1)t_{R}(a\z)\\
&=(h\o S^{-1}a\o\tens 1)s_{R}(a\z)\tens h\t\tens 1=h\o S^{-1}a\o\tens a\z\tens h\t\tens 1\\
&=(\Phi\tens\Phi)(a\tens h\o\tens 1\tens h\t)=(\Phi\tens\Phi)\Delta_{L}(a\tens h)
\end{align*}
where the 5th equality holds because the image of $\Delta_{R}$ being considered is in $\CR\tens_{B}\CR$. We also have
\begin{align*} \Phi(s_L(a))&=\Phi(a\tens 1)=S^{-1}a\o\tens a\z=t_R(a),\\
 \Phi( t_L(a))&=\Phi(a\z\tens a\o)=a\o S^{-1}a\z\o\tens a\z\z=a\o\t S^{-1} a\o\o\tens a\z=1\tens a=s_{R}(a)
 \end{align*}
It is clear that $\circledast$ is invertible, with inverse given by $(h\ot a)^{\circledast^{-1}}=a^*\ot S^{-1}(h^*)$. We can check that $\Phi$ is invertible. Indeed, Also,
\[\Phi\circ \Phi^{-1}(h\ot a)=\Phi(a\z\ot h a\o)=h a\t S^{-1}a\o\ot a\z=h\ot a.\]
Also,
\[\Phi^{-1}\circ \Phi(a\ot h)=\Phi^{-1}(h S^{-1}a\o\ot a\z)=a\z\ot h S^{-1}(a\t)a\o=a\ot h.\]
Finally, on the one hand
\begin{align*}
    \Phi^{-1}\circ \circledast(a\ot h)=\Phi^{-1}(S^{-1}(h^{*})\ot a^*)=a\z{}^{*}\ot S^{-1}(h^*)a\o{}^{*},
\end{align*}
on the other hand
\begin{align*}
    \circledast^{-1}\circ\Phi(a\ot h)=(hS^{-1}(a\o)\ot a\z)^{\circledast^{-1}}=a\z{}^*\ot S^{-1}((hS^{-1}(a\o))^{*})=a\z{}^{*}\ot S^{-1}(h^*)a\o{}^{*}.
\end{align*}
Finally, since  $\cL$ is a full Hopf algebroid, it is in particular a left and anti-left Hopf algebroid. Hence, $(\cL, \CR)$ is a $*$-Hopf algebroid pair.
\endproof

\begin{corollary}\label{coractfull}
    If $B$ is a braided-commutative algebra in $\CM^H_H$ and $\CR=H\#B$ as above, $\CR^{op}$ is a full $*$-Hopf algebroid with antipode and the star structure
        \[S(h\ot a)=S^{-1}(a\o)S^{-1}(h\t)\ot a\z\ra S^{-1}(h\o),\quad S^{-1}(h\ot a)=S(h\t a\t)\ot a\z\ra S(h\o a\o),\]
    \[(h\ot a)^{\star}=h^*\o\ot a^*\ra h^*\t,\]
    for all $h\ot a\in \CR^{op}$.
\end{corollary}
\proof
Since $\Phi:\cL\to \cR^{op}$ is an isomorphism of left bialgebroids, and $\cL$ is a full $\star$-Hopf algebroid, then the corresponding structure of $\cR$ can be inherited from $\cL$. Namely,
\[s_{R}:=\Phi\circ s_{L}^{-1}\circ\Phi^{-1},\quad \star_{\cR}=\Phi\circ \star_{\cL}\circ\Phi^{-1}.\]
\endproof

Note both the full $*$-Hopf algebroid in Section~\ref{sec:fullaction} and the one in the corollary imply $*$-Hopf algebroid pairs by Example~\ref{exfullpair}, but  $*$-Hopf algebroid pairs are a more general concept and the one constructed above is in general different from both of these.

\begin{proposition} For $B$ a braided-commutative algebra in $\CM^H_H$ and $\CL,\CR$ as above.
\newline (i) There is a monoidal isomorphism
\[  {}_\CL\CM\cong {}_B(\CM_H)_B\]
compatible with the induced forgetful functor to ${}_B\CM_B$. This sends a left $\CL$-module to
\[ a.m.b=(a\tens 1)(b\z\tens b\o)\la m=(ab\z\tens b\o)\la m,\quad  m \blacktriangleleft h=(1\tens h)\la m.\]
 (ii) There is a monoidal isomorphism
\[ \CM_\CR\cong {}_B(\CM_H)_B\]
compatible with the induced forgetful functor to ${}_B\CM_B$. This sends a right $\CR$-module to
\[ a.m.b=m\ra (1\tens b)(S^{-1}a\o\tens a\z) =m\ra (S^{-1}a_{1}\tens a\z b),\quad m  \blacktriangleleft h= m\ra (h\tens 1).\]
(iii) The condition to be a joint $\CL$ and $\CR$ module in ${}_\CL\CJ_\CR$  is, for all $m$ in the module and $a,b\in B$,
\[
(ab\z\tens b\o)\la m = m\ra (S^{-1}a_{1}\tens a\z b)\ .
\]
\end{proposition}
\proof (i) It is immediate that the first formula displayed in (i) gives an $B$-bimodule structure, and $\blacktriangleleft$ gives a right $H$-module. It remains to check that
\[
(a.m.b)  \blacktriangleleft h = (a  \ra h ).(m \blacktriangleleft h ).(b \ra h )\ .
\]
(ii) follows the same pattern as (i), and (iii) follows from the two expressions for the $A$-bimodule structure.
\endproof

Part (iii) of the Proposition  also coincides with the $\Phi$-reflexive structure between $\cL$ and $\cR$. Indeed, by Theorem \ref{thmactpair},
    \begin{align*}
        \Phi(ab\z\ot b\o)=b\t S^{-1}(a\o b\o)\ot a\z b\z= S^{-1}(a\o)\ot a\z b.
    \end{align*}

\section{$*$-structure for Ghobadi bialgroids}\label{sec:dif}

The goal of this section is to show when $A$ is a $*$ algebra and $(\Omega,\extd)$ is a $*$-calculus that Ghobadis bialgebroid
is a $*$-bialgebroid pair. The philosophy,  roughly speaking, is that the bialgebroid pair $(\CL,\CR)$ over base $A$ is reconstructed from monoidal category of ${}_A\CE_A$ of bimodules with bimodule connection and the forgetful functor to ${}_A\CM_A$, and our new result is that all of this works at the bar category level in the $*$-algebra and $*$-calculus setting.

\subsection{Preliminaries on differentials and vector fields.}\label{sec:precalc}

Suppose that $A$ is a unital algebra.
For $a\in A$ and $A^{op}$ the algebra with the opposite product, recall that we write elements $\underline{a}\in A^{op}$ where the overline is simply a reminder of being in the opposite algebra. Then we have $\underline{a_1}\cdot \underline{a_2}=\underline{a_2 a_1}$.

A first order differential calculus (FODC) over an algebra $A$ is a $A$-bimodule $\Omega^{1}$ together with a linear map $\extd: A\to \Omega^{1}$, such that
\begin{enumerate}
    \item $\extd(ab)=(\extd a)b+a(\extd b)$ for any $a, b\in A$;
    \item  $\Omega^{1}=\textup{span}\{a\extd b | a,b\in A\}.$
\end{enumerate}
A $*$-FODC over a $*$-algebra $A$ is a FODC $\Omega^1$ over $A$ together with an antilinear map $*:\Omega^1\to \Omega^1$, such that $\extd(a^*)=(\extd a)^*$ and $(a\omega a')^*=a'^* \omega^* a^*$ for any $a, a'\in A$ and $\omega\in \Omega^1$.

For an algebra $A$ we define $A^e=A\tens A^{op}$ with the product simply the product in each factor, i.e.
\[
(a_1\tens \underline{a_2})(a_3\tens \underline{a_4})=a_1a_3\tens \underline{a_4a_2}\ .
\]
A left  $A^e$-module $K$ is essentially the same as an $A$-bimodule, using
\begin{equation}\label{Aeconv}
(a_1\underline{a_2}).k=a_1k\, a_2\ .
\end{equation}
For $E\in {}_A\mathcal{M}{}_A$ and $F\in {}_{A^{op}}\mathcal{M}{}_{A^{op}}$ we define the $A^e$-bimodule
$E\tens F$ in a similar fashion by taking actions separately on the first and second tensor factor.
We shall often to write $a_1\underline{a_2}\in A^e$ rather than $a_1\tens \underline{a_2}\in A^e$ to avoid confusion with other tensor products, and also $(e,f)\in E\tens F$ rather than $e \tens f\in E\tens F$.

Suppose that $\Omega^1$ is right fgp $A$-module and we define the right vector fields $\cX^R=\mathrm{Hom}_A(\Omega^1,A)$ and the coevaluation $\coev^{L}(1)=\omega_i\tens_{A} x_i$ (summation implicit). If we denote $\ev^{L}(x,\omega):=x(\omega)$, and define a $A$-bimodule structure on $\cX^{R}$ such that $\ev^{L}(axb, \omega)= a \ev^{L}(x, b\omega)$ then
\[\omega=\omega_{i}\ev^{L}(x_{i}, \omega),\quad x=\ev^{L}(x, \omega_{i})x_{i}.\]
for any $a, b\in A$ and $X\in \cX^{R}$, $\omega\in \Omega^{1}$. Moreover, the coevaltation commutes with $A$, namely, $a\omega_{i}\ot x_{i}=\omega_{i}\ot x_{i}a$. Similarly, suppose $\Omega^{1}$ is a left fgp A-module and define $\cX^{L}:={}_{A}\Hom(\Omega, A)$, there is $\ev^{R}:\Omega^{1}\ot_{A} \cX^{L}\to A$ and $\coev^{R}(1)=y_{j}\ot_{A}\eta_{j}$, such that $\ev^{R}(a\omega, byc)=a\ev^{R}(\omega b, y)c$ for any $a, b, c\in A$ and $y\in \cX^{L}$, $\omega\in \Omega^{1}$. Moreover, we have
\[\omega=\ev^{R}(\omega, y_{j})\eta_{j},\quad y=y_{j}\ev^{R}(\eta_{j}, y),\quad a y_{j}\ot \eta_{j}=y_{j}\ot \eta_{j}a.\]

\subsection{Ghobadi's  bialgebroid $\cI\CB(\Omega^1)$} \label{secBpiv} Given a left and right fgp $A$-bimodule $\Omega^1$, we have that $\cX^R\tens\Omega^{1}$  and $\Omega^{1}\ot \cX^L$ are $A^e$-bimodules with the bimodule structure given by
\begin{align} \label{ufyt}
a_1\underline{a_2}(x,\omega)a_3\underline{a_4}=(a_1 x a_3  ,  a_4 \omega a_2),\quad a_1\underline{a_2}(\omega,y)a_3\underline{a_4}=(a_1 x a_3  ,  a_4 \omega a_2).
\end{align}
 Let $\CI\CB(\Omega^1)$ be the quotient of the  free product of tensor algebras
$T_{A^e} (\cX^R\tens\Omega^{1})$ and $T_{A^e} (\Omega^{1}\ot \cX^L)$ by the relations
\begin{align}\label{equ. relations 1}
    (\omega_{i},y)(x_{i}, \omega)=\underline{\ev^{R}(\omega, y)},\quad (x,\eta_{j})(w, y_{j})=\ev^{L}(x, w).
\end{align}
It is shown in \cite{AryGho1} that $\CI\CB(\Omega^1)$ is a left $A$-bialgebroid with the coproduct and counit given by
\begin{align}\label{equ. coproduct 1}
    \Delta_{L}(\omega,y)=(\omega,y_{j})\ot_{A}(\eta_{j}, y),\quad \Delta_{L}(x,\omega)=(x,\omega_{i})\ot_{A}(x_{i}, \omega),
\end{align}
and
\begin{align}\label{equ. counit 1}
    \varepsilon_{L}(\omega,y)=\ev^{R}(\omega, y),\quad\varepsilon_{L}(x,\omega)=\ev^{L}(x,\omega).
\end{align}

Define ${}_{A}\mathcal{IM}_A^{\Omega^1}$ be the category of objects  $(M, \sigma_{M})$, where $M$ is an $A$-bimodule, $\sigma_{M}:M\ot_{A}\Omega^1\to \Omega^1\ot_A M$ is an invertible $A$-bimodule map. A morphism from $(M, \sigma_{M})$ to  $(N, \sigma_{N})$ consists of an $A$-bimodule map $f:M\to N$, such that $(\id\ot f)\sigma_{M}=\sigma_{N}(f\ot \id):M\ot \Omega^1\to \Omega^1\ot N$. It is clear that ${}_{A}\mathcal{IM}_{A}^{\Omega^1}$ is a monoidal category with
\[\sigma_{M\ot N}=(\sigma_{M}\ot\id)\circ(\id\ot \sigma_{N}):M\ot N\ot \Omega^1\to \Omega^1\ot M\ot N,\]
for any $(M, \sigma_{M})$ and $(N, \sigma_{N})\in {}_{A}\mathcal{IM}^{\Omega^1}_A$.

\begin{proposition}\label{prop. relation between invertible bimodule and left module}\cite{AryGho1}
There is an equivalence of monoidal categories ${}_{\CI\CB(\Omega^1)}\mathcal{M} \cong   {}_A \mathcal{IM}_{A}^{\Omega^1}$. The map
$\sigma_{M}$ and the left $\CI\CB(\Omega^1)$ action on $M$ can be given in terms of each other by
\begin{align*}
(x,\omega)\la m=(\ev^L\tens \id)(x\tens\sigma_{M}(m\tens\omega))\ ,\quad (\omega,y)\la m=(\id\ot \ev^R)(\sigma_{M}^{-1}(\omega,m)\ot y).
\end{align*}
Conversely,
\begin{align*}
\sigma_{M}(m\tens\omega)=\omega_i\tens (x_i, \omega)\la m\ ,\quad \sigma_{M}^{-1}(\omega\ot m)=(\omega, y_{j})\la m\ot \eta_{j}.
\end{align*}
for any $m\in M$, $\omega\in\Omega^1$ and $x\in\cX^R$.
\end{proposition}

Note that if we did not require $\sigma_M$ to be invertible, i.e. the category ${}_A\CM^{\Omega^1}_A$, then we would use $\CB(\Omega^1)$ with just the $\cX^R\tens\Omega^{1})$ generators.

Now let $A$ be a $*$-algebra and  $\Omega^1$ a $*$-bimodule in the sense of a $*$-object in the bar category of $A$-bimodules, i.e. a $*$-object in the bar category of $A$-bimodules. The latter means that there is an antilinear involution $*:\Omega^1\to \Omega^1$ such that $(a\omega)^*=\omega^* a^*$  (as for a $*$-calculus). There is an antilinear map $\circledast:\cX^R\to \cX^L$ given by
\begin{align}\label{equ. star translate right vector field to left vector field}
    \ev^R(\omega, x^\circledast)=\ev^{L}(x,\omega^*)^*.
\end{align}
To check this is well defined, we can see
    \begin{align*}
        \ev^R(a\omega,
  x^\circledast )=\ev^{L}(x,(a\omega)^*)^*=\ev^{L}(x,\omega^* a^*)^*=a \ev^{L}(x,\omega^* )^*=a\ev^R(\omega, x^\circledast).
  \end{align*}
Similarly, $\circledast^{-1}:\cX^{L}\to \cX^R$ can be given by
\[\ev^L(y^{\circledast^{-1}},\omega)=\ev^R(\omega^*,y)^*.\]
Also, we can see
\begin{align}\label{equ. coev in terms of star}
    \omega_{i}^*\ot x_{i}^\circledast=\coev^R(1),\quad y_{j}^{\circledast^{-1}}\ot \eta_{j}^*=\coev^{L}(1).
\end{align}
Indeed,
\begin{align*}
    \ev^R(\omega, x_{i}^\circledast)\omega_{i}^*=\ev^L(x_{i},\omega^*)^* \omega_{i}^*=(\omega_{i}\ev^L(x_i, \omega^*))^*=\omega.
\end{align*}
Moreover, we have $(axb)^\circledast=b^* x^\circledast a^*$. Indeed,
\begin{align*}
    \ev^R(\omega, b^* x^\circledast a^*)=&\ev^R(\omega b^*, x^\circledast)a^*=\ev^L(x, b\omega^*)^* a*=(a\ev^L(x, b\omega^*))^*= (\ev^L(axb, \omega^*))^*\\
    =&\ev^R((axb)^\circledast, \omega).
\end{align*}

\begin{theorem}\label{thm. first bar category equ}
    If $A$ is a $*$-algebra $\Omega^{1}$ is an fgp $*$-bimodule  then

   (1) ${}_{\CI\CB(\Omega^1)}\mathcal{M}$ is a bar category with $\overline{M}$ an object by
      \[(\omega, y)\la \overline{m}=\overline{(y^{\circledast^{-1}}, \omega^*)\la m},\quad (x,\omega)\la \overline{m}=\overline{(\omega^*,x^{\circledast})\la m},\quad a\und{b}\la \overline{m}=\overline{b^* m a^*}, \]
    for any $m\in M\in {}_{\CI\CB(\Omega^1)}\mathcal{M}$.

    (2)   ${}_A \cI\mathcal{M}_{A}^{\Omega^1}$ is a bar category, with $(\overline{M},\sigma_{\overline{M}})$ defined by
    \[\sigma_{\overline{M}}(\overline{m}\ot \omega)=\omega_{i}\ot \overline{(\id\ot \ev^R)(\sigma^{-1}_{M}(\omega^{*}, m)\ot x_{i}^\circledast)},\]
    with inverse
    \[\sigma^{-1}_{\overline{M}}(\omega\ot\overline{m})=\overline{(\ev^{L}\ot\id)(y_{j}^{\circledast^{-1}}\ot \sigma_{M}(m,\omega^*))}\ot \eta_{j}\]
 for all $m\in M\in {}_A \cI\mathcal{M}_{A}^{\Omega^1}$.

 (3) The bar categories in (1) and (2) are isomorphic.
\end{theorem}
\begin{proof}
    (1) We show ${}_{\CI\CB(\Omega^1)}\mathcal{M}$ is a bar category. We can see $(s_{L}(a)t_{L}(b))\la \overline{m}=(a,b)\la \overline{m}=\overline{(b^*,a^*)\la m}=\overline{b^* m a^*}=a\overline{m}b$.
    We also have
    \begin{align*}
        (\omega,y)\la \overline{\overline{m}}=\overline{(y^{\circledast^{-1}},\omega^*)\la \overline{m}}=\overline{\overline{(\omega,y)\la m}}.
    \end{align*}
    Similarly, $(x,\omega)\la \overline{\overline{m}}=\overline{\overline{(x,\omega)\la m}}$. And
    \begin{align*}
        \Upsilon((x,\omega)\la \overline{m\ot n})=&\Upsilon(\overline{(\omega^*, x^\circledast)\o\la m\ot (\omega^*, x^\circledast)\t\la n})
        =\Upsilon(\overline{ (\omega^*, y_{j})\la m\ot (\rho^j, x^\circledast)\la n})\\
        =&\Upsilon(\overline{ (\omega^*, x_{i}^\circledast)\la m\ot (\omega_{i}^{*}, x^\circledast)\la n})=\overline{(\omega_{i}^{*}, x^\circledast)\la n}\ot\overline{ (\omega^*, x_{i}^\circledast)\la m}\\
        =&(x,\omega_{i})\la \overline{n}\ot (x_{i},\omega)\la\overline{m}=(x,\omega)\la(\Upsilon(\overline{m\ot n})).
    \end{align*}
    Similarly, $\Upsilon((\omega,y)\la \overline{m\ot n})=(\omega,y)\la(\Upsilon(\overline{m\ot n}))$.

    (2) Next, we check ${}_A \cI\mathcal{M}_{A}^{\Omega^1}$ is a bar category. We can see
    \begin{align*}
        \sigma_{\overline{M}}(\overline{m}\ot \omega)=\omega_{i}\ot \overline{(\id\ot \ev^R)(\sigma^{-1}_{M}(\omega^{*}, m)\ot x_{i}^\circledast)}=\omega_{i}\ot \overline{(\omega^*, x_{i}^\circledast)\la m}.
    \end{align*}
    And
    \begin{align*}
        \sigma^{-1}_{\overline{M}}(\omega\ot\overline{m})=\overline{(\ev^{L}\ot\id)(y_{j}^{\circledast^{-1}}\ot \sigma_{M}(m,\omega^*))}\ot \eta_{j}=\overline{(y_{j}^{\circledast^{-1}},\omega^*)\la m}\ot\eta_{j}.
    \end{align*}
    We can see $\sigma_{\overline{M}}$ factors through the balanced product between $\overline{m}$ and $\omega$. Indeed,
    \begin{align*}
        \sigma_{\overline{M}}(\overline{m}a\ot \omega)=&\omega_{i}\ot \overline{(\omega^*, x_{i}^\circledast)\la (a^* m)}=\omega_{i}\ot \overline{(\omega^* a^*, x_{i}^\circledast)\la  m}=\omega_{i}\ot \overline{((a\omega)^*, x_{i}^\circledast)\la  m}\\
        =&\sigma_{\overline{M}}(\overline{m}\ot a\omega).
    \end{align*}
    Also,
    \begin{align*}
        \sigma_{\overline{M}}(a\overline{m}\ot \omega b)=&\omega_{i}\ot \overline{(b^* \omega^*, x_{i}^\circledast)\la (ma^*)}=\omega_{i}\ot (\overline{( \omega^*,  a^* x_{i}^\circledast)\la m})b=\omega_{i}\ot (\overline{( \omega^*,  (x_{i}a)^\circledast)\la m})b\\
        =&a\omega_{i}\ot (\overline{( \omega^*,  x_{i}^\circledast)\la m})b= a\sigma_{\overline{M}}(\overline{m}\ot \omega )b.
    \end{align*}
    So $\sigma_{\overline{M}}$ is $A$-bilinear. Similarly, we can show $\sigma^{-1}_{\overline{M}}$ is well defined. We can see $\sigma_{\overline{M}}$ is invertible
    \begin{align*}
        \sigma_{\overline{M}}\circ \sigma_{\overline{M}}^{-1}(\omega\ot \overline{m})=&\omega_{i}\ot \overline{((\eta_{j}^*, x_{i}^\circledast)(y_{j}^{\circledast^{-1}},\omega^*))\la m}=\omega_{i}\ot \overline{\und{\ev^R(\omega^*, x_{i}^\circledast)}\la m}\\
        =&\omega_{i}\ot \overline{\und{\ev^L(x_{i}, \omega)^*}\la m}=\omega_{i}\ot \overline{m \ev^L(x_{i}, \omega)^*}=\omega_{i}\ev^L(x_{i}, \omega)\ot \overline{m}\\
        =&\omega\ot \overline{m}.
    \end{align*}
    And
    \begin{align*}
        \sigma^{-1}_{\overline{M}}\circ \sigma_{\overline{M}}(\overline{m}\ot \omega)=&\overline{(y_{j}^{\circledast^{-1}},\omega_{i}^*)(\omega^*, x_{i}^\circledast)\la m}\ot\eta_{j}=\overline{\ev^{L}(y_{j}^{\circledast^{-1}}, \omega^*)\la m}\ot\eta_{j}\\
        =&\overline{m}\ot \ev^{R}(\omega, y_{j})\eta_{j}=\overline{m}\ot \omega.
    \end{align*}
    Assume $f:(M,\sigma_{M})\to (N, \sigma_{N})$ is a morphism in the category of ${}_{A}\mathcal{IM}_A{}^{\Omega^1}$, we have
    \begin{align*}
        (\id\ot \overline{f})\circ \sigma_{\overline{M}}(\overline{m}\ot \omega)=&\omega_{i}\ot \overline{f\circ (\id\ot \ev^R)(\sigma^{-1}_{M}(\omega^{*}, m)\ot x_{i}^\circledast)}\\
        =&\omega_{i}\ot \overline{(\id\ot \ev^R)(\sigma^{-1}_{M}(\omega^{*}, f(m))\ot x_{i}^\circledast)}\\  =&\sigma_{\overline{N}}\circ (\overline{f}\ot\id)(\overline{m}\ot \omega).
    \end{align*}
    We can also see
    \begin{align*}
        \sigma_{\overline{\overline{M}}}(\overline{\overline{m}}\ot \omega)=\omega_{i}\ot \overline{(\omega^*, x_{i}^\circledast)\la \overline{m}}=\omega_{i}\ot \overline{\overline{(x_{i},\omega)\la m}}.
    \end{align*}
    And on the one hand
    \begin{align*}
       (\id&\ot\Upsilon_{\overline{M\ot N}})\circ\sigma_{\overline{M\ot N}}(\overline{m\ot n}\ot\omega)\\
       =&(\id\ot\Upsilon_{\overline{M\ot N}})(\omega_{i}\ot \overline{(\omega^*, x_{i}^\circledast)\la (m\ot n)})\\
       =&(\id\ot\Upsilon_{\overline{M\ot N}})(\omega_{i}\ot \overline{(\omega^*, x_{i}^\circledast)\o\la m\ot (\omega^*, x_{i}^\circledast)\t\la n})\\
       =&\omega_{i}\ot\overline{(\omega^*, x_{i}^\circledast)\t\la n}\ot \overline{(\omega^*, x_{i}^\circledast)\o\la m}\\
       =&\omega_{i}\ot\overline{(\eta_{j}, x_{i}^\circledast)\la n}\ot \overline{(\omega^*, y_{j})\la m}
    \end{align*}
 We already know that ${}_{\CI\CB(\Omega^1)}\mathcal{M} \cong   {}_A \mathcal{IM}_{A}^{\Omega^1}$ are equivanlent as monoidal category according to
        \[\sigma_{M\ot N}(m\ot n\ot \omega)=\omega_{i}\ot (x_{i},\omega)\la(m\ot n).\]
    So on the other hand, we have
    \begin{align*}
        \sigma_{\overline{N}\ot\overline{M}}&\circ (\Upsilon_{\overline{M\ot N}}\ot\id)(\overline{m\ot n}\ot\omega)       =\sigma_{\overline{N}\ot\overline{M}}(\overline{n}\ot\overline{m}\ot\omega)
        =\omega_{i}\ot (x_{i},\omega)\la(\overline{n}\ot\overline{m})\\
        =&\omega_{i}\ot (x_{i},\omega)\o\la\overline{n}\ot (x_{i},\omega)\t\la\overline{m}
        =\omega_{i}\ot (x_{i},\omega_{j})\la\overline{n}\ot (x_{j},\omega)\la\overline{m}\\
        =&\omega_{i}\ot \overline{(\omega_{j}^{*}, x_{i}^{\circledast})\la n}\ot \overline{(\omega^*, x_{j}^{\circledast})\la m}.
    \end{align*}

   (3)  Let $F:{}_{\CI\CB(\Omega^1)}\mathcal{M}\to   {}_A \mathcal{IM}_{A}^{\Omega^1}$ be the functor. We can see $fb_{M}:\overline{F(M)}\to F(\overline{M})$ is identity, since
    \begin{align*}
        \sigma_{\overline{F(M)}}(\overline{m}\ot \omega)=\omega_{i}\ot \overline{(\omega^*,x_{i}^{*})\la m}=\omega_{i}\ot (x_{i},\omega)\la \overline{m}=\sigma_{F(\overline{M})}(\overline{m}\ot \omega).
    \end{align*}
    So it is not hard to see $F$ is an equivalence bar functor.
\end{proof}

This bar category structure comes from $\CI\CB(\Omega^1)$ being part of a $*$-bialgebroid pair. We omit details of the other half here since it follows the same line  as for $\CI\CL(\Omega^1)$ covered later.

\subsection{Full $*$-Hopf algebroid $\CI\CI\CB(\Omega^1)$.}\label{sec:pivotal}
In this subsection, given $\Omega^1$ left and right fgp, we let $\cX$ be an $A$-bimodule equipped with both $\ev^L,\coev^L$ and $\ev^R,\coev^R$ structures making it both left and right dual to $\Omega^1$, so   $\cX=\cX^R=\cX^L$ simultaneously. By \cite{AryGho1}, we can define
 $\CI\CI\CB(\Omega^1)$ is the quotient of $\CI\CB(\Omega^1)$ by the relations
\begin{align}\label{equ. relations 2}
    (y_{j},\omega)(\eta_{j}, x)=\underline{\ev^{L}(x, \omega)},\quad (\omega,x_{i})(x, \omega_{i})=\ev^{R}(\omega, x).
\end{align}
$\CI\CI\CB(\Omega^1)$ is a full Hopf algebroid with antipode given by
\[S(x,\omega)=(\omega, x),\quad S(\omega, x)=(x,\omega).\]
The left bialgebroid structure can be inherited from $\CI\CB(\Omega^1)$. By Proposition \ref{def. full Hopf algebroid2}, its corresponding right bialgebroid structure is given by
\begin{align}\label{equ. source and target maps of right bialgebroid}
    s_{R}(\und a)=\und a,\quad t_{R}(\und a)=a
\end{align}
and the coproduct
\begin{align}\label{equ. coproduct of right bialgebroid}
   \Delta_{R}(\omega, y)=(\omega, x_i)\ot_{A^{op}} (w_{i}, y),\quad \Delta_{R}(x, \omega)=(x,\eta_{j})\ot_{A^{op}} (y_{j}, \omega).
\end{align}
And the counit
\begin{align}\label{equ. counit of right bialgebroid}
    \varepsilon_{R}(\omega, y)=\und{\ev^L(y,\omega)},\quad\varepsilon_{R}(x,\omega)=\und{\ev^R(\omega,x)}.
\end{align}
Indeed, for the coproduct of the right bialgebroid, by (\ref{equ. left coproduct to right coproduct}), we have
\begin{align*}
    \Delta_{R}(\omega, y)=&S(S^{-1}(\omega, y)\t)\ot S(S^{-1}(\omega, y)\o)=S(x_{i}, \omega)\ot S(y, \omega_{i}) \\
    =&(\omega, x_i)\ot (w_{i}, y),
\end{align*}
similar for the rest.

This recaps the construction of \cite{AryGho1} in our notations. Now we recall that for a $*$-algebra $A$ and a fgp $*$-bimodule $\cX$, we have a map $\circledast: \cX\to \cX$ satisfying (\ref{equ. star translate right vector field to left vector field}) and $(\ref{equ. coev in terms of star})$.

\begin{proposition}\label{Prop. full star Hopf algebroid H omega} If $\Omega^1$ is an fgp $*$-bimodule and $\cX$ a left and right dual as above.  If $\circledast=\circledast^{-1}$ then $\CI\CI\CB(\Omega^1)$ is a full $\star$-Hopf algebroid with
    \begin{align}
       s_{L}(a)^\star=a^*,\quad t_{L}(a)^\star=\underline{a^*},\quad(\omega, y)^\star=(\omega^*, y^{\circledast}),\quad (x, \omega)^\star=(x^\circledast, \omega^*).
    \end{align}
\end{proposition}
\begin{proof}
    Clearly, $s_{L}(a)^\star=a^*=t_{R}(\underline{a^*})$ and $t_{L}(a)^\star=\underline{a^*}=s_{R}(\underline{a^*})$. Also, $\varepsilon_{L}((\omega, y)^\star)=\varepsilon_{L}(\omega^*, y^{\circledast})=\ev^R(\omega^*, y^{\circledast})=\ev^L(y, \omega)^*=\und{\varepsilon_{R}(\omega, y)^*}$. Similarly, $\varepsilon_{L}((x, \omega)^\star)=\und{\varepsilon_{R}(x, \omega)^*}$. We can see on the one hand,
    \begin{align*}
        (x,\omega)^\star\ro\ot (x,\omega)^\star\rt=(x^\circledast, \omega^*)\ro\ot (x^\circledast, \omega^*)\rt=(x^*, \eta_{j})\ot (y_{j}, \omega^*).
    \end{align*}
    On the other hand,
    \begin{align*}
        (x,\omega)\o{}^\star\ot (x,\omega)\t{}^\star=(x,\omega_{i}){}^\star\ot (x_{i},\omega){}^\star=(x^\circledast, \omega_{i}^{*})\ot (x_{i}^\circledast, \omega^*).
    \end{align*}
    To see $\star$ is an anti-algebra map that preserve the relations (\ref{equ. relations 1}) and (\ref{equ. relations 2}), we have
    \begin{align*}
        ((\omega_{i}, y)(x_{i}, \omega))^\star=&\underline{\ev^R(\omega, y)}^\star=\underline{\ev^R(\omega, y)^*}=\underline{\ev^L(y^\circledast, \omega^*)}=(x_{i}^\circledast, \omega^*)(\omega_{i}^*, y^\circledast)\\
        =&(x_{i}, \omega)^\star(\omega_{i}, y)^\star.
    \end{align*}
    And
    \begin{align*}
        ((x,\eta_{j})(w, y_{j}))^\star=&\ev^{L}(x, \omega)^\star=\ev^{L}(x, \omega)^*=\ev^{R}(\omega^*, x^\circledast)=(\omega^*, y_{j}^\circledast)(x^\circledast, \eta_{j}^*)\\
        =&(w, y_{j})^\star (x,\eta_{j})^\star.
    \end{align*}
    The rest is similar.
\end{proof}

Let ${}_{A}\CI\CI\CM_{A}^{\Omega^1}$ be the submonoidal category of ${}_{A}\mathcal{IM}_{A}^{\Omega^1}$ consisting of $M\in {}_{A}\mathcal{IM}_{A}^{\Omega^1}$ such that the $A$-bimodule map
\begin{align}
    \tau_{M}:\cX\ot_A M\to M\ot_A \cX,\quad \tau_{M}(x\ot m)=(x, \omega_{i})\la m\ot x_{i}
\end{align}
is invertible with inverse given by
\begin{align}
    \tau_{M}^R(m\ot x)=y_{j}\ot (\eta_{j}, x)\la m.
\end{align}
A morphism $f:(M, \sigma_{M}, \tau_{M})\to (N, \sigma_{N}, \tau_{N})$ satisfies in additional $(f\ot\id)\circ \tau_{M}=\tau_{N}\circ (\id\ot f)$.
It is given by \cite{AryGho1} that ${}_{A}\mathcal{IIM}_{A}^{\Omega^1}\cong {}_{\CI\CI\CB(\Omega^1)}\mathcal{M}$ as monoidal category.

\begin{theorem}\label{thm. biinvertible module} Under the assumptions of Proposition~\ref{Prop. full star Hopf algebroid H omega}, the isomorphism ${}_{A}\CI\CI\CM_{A}^{\Omega^1}\cong {}_{\CI\CI\CB(\Omega^1)}\mathcal{M}$ is as bar monoidal categories, as  bar subcategories of ${}_{A}\mathcal{IM}_{A}^{\Omega^1}$ and $ {}_{\cI\CB(\Omega^1)}\mathcal{M}$ respectively.
\end{theorem}
\begin{proof}
    Since $\CI\CI\CB(\Omega^1)$ is a full $*$-Hopf algebroid, ${}_{\CI\CI\CB(\Omega^1)}\mathcal{M}$ is a bar category by Proposition \ref{propfullbar}. More precisely, given $m\in M\in {}_{\CI\CI\CB(\Omega^1)}\mathcal{M}$,
    \[(x,\omega)\la \overline{m}=\overline{S((x,\omega)^\star)\la m}=\overline{(\omega^*, x^\circledast)\la m},\]
    and similarly,
    \[(\omega, y)\la \overline{m}=\overline{(y^{\circledast^{-1}}, \omega^*)\la m}.\]
    It is not hard to see that this defines an $\CI\CI\CB(\Omega^1)$-module structure on $\overline{M}$ descending from the left $\CI\CB(\Omega^1)$-module structure given in Theorem \ref{thm. first bar category equ}. As a result, $ {}_{\CI\CI\CB(\Omega^1)}\mathcal{M}$ is a bar subcategory of $ {}_{\CI\CB(\Omega^1)}\mathcal{M}$. To see ${}_{A}\mathcal{IIM}_{A}^{\Omega^1}$ is a bar category with the bar structure inherited from ${}_{A}\mathcal{IM}_{A}^{\Omega^1}$, we first observe that $\tau_{\overline{M}}$ given by $\tau_{\overline{M}}(x\ot \overline{m})=(x, \omega_{i})\la \overline{m}\ot x_{i}$ is well defined. We can also  check $\tau_{\overline{M}}$ is invertible with the given fomular above. Indeed,
    \begin{align*}
        \tau_{\overline{M}}\circ \tau_{\overline{M}}^{R}(\overline{m}\ot x)=&(y_{j}, \omega_{i}) (\eta_{j},x)\la \overline{m}\ot x_{i}=  \ev^L(x,\omega_{i})\la \overline{m}\ot x_{i}=\overline{m}\ot \ev^L(x,\omega_{i})x_{i}
        =\overline{m}\ot x.
    \end{align*}
    Similarly, $\tau_{\overline{M}}^{R}\circ \tau_{\overline{M}}(x\ot \overline{m})=x\ot \overline{m}$. Let $f:(M, \sigma_{M}, \tau_{M})\to (N, \sigma_{N}, \tau_{N})$ be a morphism. We can see
    \begin{align*}
        (\overline{f}\ot\id)\circ \tau_{\overline{M}}(x\ot \overline{m})=&(\overline{f}\ot\id)((x, \omega_{i})\la \overline{m}\ot x_{i})=\overline{f((\omega_{i}^*, x^\circledast)\la m)}\ot x_{i}=\overline{f((\omega_{i}^*, x^\circledast)\la m)}\ot x_{i}^{**}\\
        =&\overline{f((\eta_{j}, x^\circledast)\la m)}\ot y_{j}^{*}=\overline{(\eta_{j}, x^\circledast)\la f(m)}\ot y_{j}^{*}=\tau_{N}\circ (\id\ot \overline{f})(x\ot\overline{m}).
    \end{align*}
    Moreover,
    \begin{align*}
        \tau_{\overline{\overline{M}}}(x\ot\overline{\overline{m}})=\overline{\overline{(x, \omega_{i})\la m}}\ot x_{i}.
    \end{align*}
     And on the one hand
    \begin{align*}
       (\Upsilon_{\overline{M\ot N}}&\ot\id)\circ\tau_{\overline{M\ot N}}(x\ot \overline{m\ot n})\\
       =&(\Upsilon_{\overline{M\ot N}}\ot\id)(\overline{(\omega_{i}^*, x^\circledast)\la (m\ot n)}\ot x_{i})\\
       =&(\Upsilon_{\overline{M\ot N}}\ot\id)(\overline{(\omega_{i}^*, x^\circledast)\o\la m\ot (\omega_{i}^*, x^\circledast)\t\la n}\ot x_{i})\\
       =&\overline{(\omega_{i}^*, x^\circledast)\t\la n}\ot \overline{(\omega_{i}^*, x^\circledast)\o\la m}\ot x_{i}\\
       =&\overline{(\eta_{j}, x^\circledast)\la n}\ot \overline{(\omega_{i}^*, y_{j})\la m}\ot x_i,
    \end{align*}
    on the other hand,
    \begin{align*}
        \tau_{\overline{N}\ot\overline{M}}&\circ (\id\ot\Upsilon_{\overline{M\ot N}})(x\ot \overline{m\ot n})       =\tau_{\overline{N}\ot\overline{M}}(x\ot\overline{n}\ot\overline{m})
        =(x,\omega_{i})\la(\overline{n}\ot\overline{m})\ot x_i\\
        =& (x,\omega_{i})\o\la\overline{n}\ot (x,\omega_{i})\t\la\overline{m}\ot x_i
        = (x,\omega_{j})\la\overline{n}\ot (x_{j},\omega_i)\la\overline{m}\ot x_i\\
        =& \overline{(\omega_{j}^{*}, x^{\circledast})\la n}\ot \overline{(\omega_{i}^*, x_{j}^{\circledast})\la m}\ot x_i.
    \end{align*}
    Let $F:{}_{\CI\CI\CB(\Omega^1)}\mathcal{M}\to   {}_{A}\mathcal{IM}_{A}^{\Omega^1}$ be the functor. We can see $fb_{M}:\overline{F(M)}\to F(\overline{M})$ is identity, since
    \begin{align*}
        \tau_{\overline{F(M)}}(x\ot\overline{m})=\overline{(\omega_{i}^*,x^{\circledast})\la m}\ot x_i=(x_i, \omega)\la\overline{m}\ot x_i=\tau_{F(\overline{M})}(x\ot\overline{m}).
    \end{align*}
    So it is not hard to see $F$ is an equivalence bar functor.
\end{proof}

This upgrades the construction of $\CI\CI\CB(\Omega^1)$ and its representation category in \cite{AryGho1} to the full $*$-Hopf algebroid case. This is warm up for the next section.

\subsection{Preliminary background on bimodule connections}\label{sec:preconn}

 In this section we let $(\Omega^1,\extd)$ be  differential calculus on $A$. We first recall that a left bimodule connection $(E, \nabla_E,\sigma_E)$ consist of an $A$-bimodule $E$, an left connection $\nabla_{E}$ in the sense of a linear map obeying
\[   \nabla_{E}(a e)=a\nabla_{E}(e)+\extd a \ot e\]
  for any $a\in A$ and $e\in E$, and a $A$-bimodule map (`generalised braiding') $\sigma_{E}:E\ot_{A}\Omega^1\to \Omega^1\ot_{A} E$ such that
  \begin{align*}
 \nabla_{E}(e a)=\nabla_{E}(e)\, a+ \sigma_{E}(e\ot \extd a),
  \end{align*}
  for any $a\in A$ and $e\in E$. Such objects form a monoidal category which we will denote ${}_A\CE_A$. A morphism $f:(E, \nabla_{E},\sigma_{E})\to (F, \nabla_{F},\sigma_{F})$ is a $A$-bimodule map $f: E\to F$, such that
  \[\sigma_{F}\circ(f\ot\id)=(\id\ot f)\circ \sigma_{E},\quad (\id\ot f)\circ\nabla_{E}=\nabla_{F}\circ f.\]
  The monoidal structure is
  \[\nabla_{E\ot F}=\nabla_{E}\ot\id+(\sigma_{E}\ot\id)\circ(\id\ot \nabla_{F}),\quad \sigma_{E\ot F}=(\sigma_{E}\ot\id)\circ(\id\ot\sigma_{F}),\]
for any $(E, \nabla_{E},\sigma_{E})$, $(F, \nabla_{F},\sigma_{F})\in {}_{A}\CE_{A}$
  We denote by ${}_{A}\cI\CE_{A}$  the monoidal subcategory with invertible braiding. Similarly, a  right bimodule connection $(F, \nabla^R_F,\sigma^R_F)$ consists of an $A$-bimodule $F$ and intertwining $\sigma_F^R:\Omega^1\ot F\to F\ot \Omega^1$, and right connection $\nabla_{F}^R:F\to F\ot\Omega^1$, such that
  \[ \nabla_{F}^R(fa)=\nabla_{F}^R(f)a+f\ot \extd a,\quad \nabla_{F}^R(af)=a\nabla_{F}^R(f)+\sigma_F(\extd a\ot f).\] We denote the category of right $A$-bimodule connection by ${}_{A}\CE_{A}^R$ and its subcategory ${}_{A}\cI\CE^{R}_{A}$ with invertible braiding, which is isomorphic with  ${}_{A}\cI\CE_{A}$ as monoidal category. More precisely, for  any  $(E,\sigma_{E}, \nabla_{E})\in{}_{A}\cI\CE_{A}$, we can give a right bimodule connection by
\begin{align}\label{LRconnE}
    \sigma_E^R:=(\sigma_E)^{-1},\quad \nabla^R_E=(\sigma_E)^{-1}\circ \nabla_E.
\end{align}
This gives a monoidal isomorphism ${}_{A}\cI\CE_{A}\cong {}_{A}\cI\CE_{A}^R$. It is not hard to see that ${}_{A}\cI\CE_{A}$ is a submonoidal category of ${}_{A}\cI\CM_{A}^{\Omega^1}$. Also, if $A$ is a $*$-algebra and $(\Omega^1,\extd)$ a $*$-calculus then ${}_{A}\cI\CE_{A}$  is, moreover, a bar category\cite{BM} with
\begin{align}\label{equ. star compatible}
   \nabla_{\overline{E}}(\overline{e})=\dagger(\overline{(\sigma_{E})^{-1}\circ \nabla_{E}(e)}),\quad \sigma_{\overline{E}}(\overline{e}\ot \omega)=\dagger( \overline{(\sigma_{E})^{-1}(\omega^*\ot e)}),
\end{align}
where $\dagger:\overline{E\ot_A \Omega^1}\to \Omega^1\ot_A \overline{E}$ is defined by $\dagger(\overline{e\ot \omega})=\omega^*\ot \overline{e}$.
More precisely, if we denote $e^\alpha\ot e^\beta:=(\sigma_{E})^{-1}\circ \nabla_{E}(e)$ then $\nabla_{\overline{E}}(\overline{e})=e^{\beta^{*}}\ot \overline{e^\alpha}$. In this case, there is a forgetful functor ${}_{A}\cI\CE_{A}$ to ${}_{A}\cI\CM_{A}^{\Omega^1}$ in the previous section.

\begin{definition} For $\Omega^1$ f.g.p. and pivotal i.e. with a single $\cX=\cX^R=\cX^L$, we say that a left bimodule connection is {\em bi-invertible} if $\sigma_E$ is invertible with inverse $\sigma^R_E$ as before {\em and} its `transpose'  $\tau_{E}$ defined by
\[\tau_E=(\ev^L\ot\id_{E\ot \cX})\circ(\id_{\cX}\ot \sigma_{E}\ot \id_{\cX})\circ (\id_{\cX\ot E}\ot \coev^L):\cX\ot_A E\to E\ot_A \cX\]
is invertible with inverse
\[\tau_E^R=(\id_{\cX\ot E}\ot \ev^R)\circ(\id_{\cX}\ot \sigma_{E}^R\ot \id_{\cX})\circ (\coev^R\ot\id_{E\ot \cX}).\]
A morphism $f:M\to N$ between biinvertible connections needs to satisfies in addition $\tau_F\circ(\id\ot f)=(f\ot\id)\circ \tau_E$.
\end{definition}

It is not hard to see ${}_A\cI\cI\CE_A$  has a forgetful functor to ${}_{A}\mathcal{IIM}_{A}^{\Omega^1}$ in Section~\ref{sec:pivotal}. Also, it is shown by \cite{AryGho1} that this is a monoidal subcategory  of ${}_A\cI\CE_A$, which we denote ${}_A\cI\cI\CE_A$.

\begin{lemma} In the $*$-algebra case with a $*$-calculus,  ${}_A\cI\cI\CE_A\hookrightarrow {}_A\cI\CE_A$ as bar categories compatibly with the forgetful functors to${}_{A}\mathcal{IIM}_{A}^{\Omega^1}\hookrightarrow{}_{A}\mathcal{IM}_{A}^{\Omega^1}$ respectively.
\end{lemma}
\proof The proof is similar to the proof of Theorem \ref{thm. biinvertible module}, so we omit details. \endproof

\subsection{$*$-bialgebroid pair $\cI\cL(\Omega^1), \cI\cR(\Omega^1)$}\label{secIL}

Now, following  \cite{AryGho1}, assume that $\Omega^1$ is right fgp.
Define $\CL(\Omega^1)$ to be the algebra generated by $a\in A$, $\underline a\in A^{op}$, $x\in \cX^R$ and $(x, \xi)\in \cX^R\tens \Omega^1$ subject to the relations
\begin{align}\label{equ. relations 3}
&a\bullet (x,\xi) = (ax,\xi) \ ,\quad (x, \xi)\bullet a = (xa, \xi) \cr
&\underline a \bullet(x, \xi) = (x, \xi a) \ ,\quad (x, \xi)\bullet \underline a = (x, a\xi) \cr
& a\bullet x = a\,x\ ,\quad x\bullet a = x\, a + \ev^L(x, \extd a)\ ,\quad
x\bullet \underline a =  \underline a \bullet x + (x, \extd a), \quad a\underline{a}=\underline{a}a,
\end{align}
where we use $\bullet$ to denote the product of $\CL(\Omega^1)$ (in order to distinguish the bimodule structure of $\cX^R$).
By \cite{AryGho1}, $\CL(\Omega^1)$ is a left $A$-bialgebroid with the coring  structure
\begin{align}\label{equ. coring structure 1}
&\Delta_{L}(x) = x\tens 1 + (x,  \omega_{i})\tens x_i\ ,\quad \varepsilon_{L}(x)=0\ ,\cr
&\Delta_{L}(a) =a\tens 1\ ,\quad \Delta_{L}(\underline a) =1\tens\underline a\ ,\quad
\varepsilon_{L}(a)=a\ ,\quad \varepsilon_{L}(\underline a)=a\ ,\cr
&\Delta_{L}(x, \xi) = (x, \omega_{i})\tens (x_i, \xi)\ ,\quad
\varepsilon_{L}(x, \omega) =\ev^L(x,  \omega) \ .
\end{align}

\begin{proposition}\label{prop. relation between bimodule connection and left module}\cite{AryGho1}
   There exist an isomorphism of monoidal category ${}_{\CL(\Omega^1)}\CM\cong {}_{A}\CE_{A}$. More precisely, for $e\in E\in {}_{A}\CE_{A}$,
   \begin{align*}
&a \la e = ae\ ,\quad \underline a \la e = ea\ ,\quad x\la e = (\ev^L\tens\id)(x\tens\nabla_{E} e) \ ,\cr
&(x\tens\xi)\la e=(\ev^L\tens\id)(x\tens\sigma_{E}(e\tens\xi))\ .
\end{align*}
Conversely, for $e\in E\in {}_{\CL(\Omega^1)}\CM$,
\begin{align*}
    \sigma_{E}(e\ot \omega)=\omega_i\ot (x_{i}, \omega)\la e,\quad \nabla_{E}(e)=\omega_i\ot x_{i}\la e.
\end{align*}
\end{proposition}

Similarly, if $\Omega^1$ is left fgp, we  define $\CR(\Omega^1)$ to be an algebra generated by $a\in A$, $\underline a\in A^{op}$, $y\in \cX^L$ and $(\xi, y)\in  \Omega^1\tens_{A}\cX^L$ subject to the relations
\begin{align}\label{equ. relations 4}
&a\Rbullet (\xi, y) = (\xi, a\,y)  \ ,\quad (\xi, y)\Rbullet  a = (\xi, y\,a) \cr
&\underline a \Rbullet(\xi, y)  = (\xi \, a, y)  \ ,\quad (\xi, y)\Rbullet  \underline a = (a\,\xi, y) \ ,\cr
& y\Rbullet a = y\,a\ ,\quad a\Rbullet y = a\, y + \ev^R(\extd a, y)\ ,\quad
 \underline a \Rbullet  y = y \Rbullet \underline a + (\extd a , y),\quad a\underline{a}=\underline{a}a.
\end{align}
where we use $\Rbullet$ to denote the product of $\CR(\Omega^1)$. $\CR(\Omega^1)$ have the right $A$-bialgebroid structure
\begin{align}\label{equ. coring structure 2}
&\Delta_{R}(y) = 1\tens y + y_j\tens (\eta_{j}, y)\ ,\quad \varepsilon_{R}(y)=0\ ,\cr
&\Delta_{R}(a) =1\tens a \ ,\quad \Delta_{R}(\underline a) =\underline a\tens 1\ ,\quad
\varepsilon_{R}(a)=a\ ,\quad \varepsilon_{R}(\underline a)=a\ ,\cr
&\Delta_{R}(\xi, y) = (\xi\tens y_j)\tens (\eta_{j}\tens y)\ ,\quad
\varepsilon_{R}(y,\omega) =\ev^R(y,\omega) \ .
\end{align}
By a parallel construction, the category ${}_A\CE_A^R$ is isomorphic with $\CM_{\CR(\Omega^1)}$. More precisely,
if $F$ is an $A$-bimodule with a right bimodule connection $(\nabla^R_F,\sigma^R_F)$ then $F$ is a right
$\CR(\Omega^1)$-module, by
\begin{align*}
& f\ra a = fa\ ,\quad f\ra  \underline a  = af\ ,\quad f\ra y = (\id\tens\ev^R)(\nabla^R_F f\tens y) \ ,\cr
&f\ra (\xi, y)=(\id\tens\ev^R)(\sigma^R_F(\xi\tens f) \tens y)\ .
\end{align*}
Conversely, if $F$ is a right $\CR(\Omega^1)$-module,
\begin{align}\label{equ. right bimodule connection with right module}
    \sigma_F^R(\omega\ot f)=f\ra (\omega,y_j)\ot \eta_{j},\quad \nabla_F^R(f)=f\ra y_j\ot \eta_{j}.
\end{align}

This completes our summary of \cite{AryGho1}. Now let $A$ be a $*$-algebra and $(\Omega^1,\extd)$ an fgp $*$-differential calculus. As before, now recall that there is a map $\circledast:\cX^R\to \cX^L$ satisfying (\ref{equ. star translate right vector field to left vector field}) and $(\ref{equ. coev in terms of star})$.

\begin{lemma}\label{lem. star related LX and RX} If $A$ is a $*$-algebra and $(\Omega^1,\extd)$ an fgp $*$-differential calculus $\Omega^1$ then $\CL(\Omega^1)$ and $\CR(\Omega^1)$ are $*$-related, with
    \[a^\circledast=a^*,\quad \underline{a}^\circledast=\underline{a^*},\quad (\underline{a}\bullet x)^\circledast=x^\circledast\Rbullet\underline{a^*},\quad
(x,\xi)^\circledast = (\xi^*, x^\circledast )\]
\end{lemma}
\begin{proof}
    First, we have
    \begin{align*}
        s_{L}(a)^\circledast=a^\circledast=a^*=s_{R}(a^*).
    \end{align*}
    Similarly, $t_{L}(a)^\circledast=t_{R}(a^*)$. Also
    \begin{align*}
        \varepsilon_{R}((x,\omega)^\circledast)=\ev^R(\omega^*, x^\circledast)=\ev^L(x,\omega)^*=\varepsilon_L(x,\omega)^*.
    \end{align*}
    To see $\circledast$ is an anti-algebra map, we have
    \begin{align*}
        (s_{L}(a)(x,\omega))^\circledast=(ax, \omega)^\circledast=(\omega^*, x^\circledast a^*)=(x,\omega)^\circledast s_{R}(a^*)=(x,\omega)^\circledast s_{L}(a)^\circledast.
    \end{align*}
    Similar for the rest. Also,
    \begin{align*}
        (a\bullet x)^\circledast=(ax)^\circledast=x^\circledast a^*=x^\circledast \Rbullet a^*=x^\circledast s_{R}(a^*)=x^\circledast s_{L}(a)^\circledast,
    \end{align*}
    and
    \begin{align*}
        (x\bullet a)^\circledast=(xa)^\circledast+\ev^L(x, \extd a)^*=a^* x^\circledast+\ev^R(\extd a^*, x^\circledast)=a^*\Rbullet x^\circledast,
    \end{align*}
    and
    \begin{align*}
        (x\bullet\underline{a})^\circledast=(\underline{a}\bullet x)^\circledast+(da^*, x^\circledast)=x^\circledast\Rbullet \underline{a^*}+(da^*, x^\circledast)=\underline{a}^\circledast\Rbullet x^\circledast.
    \end{align*}
    For the coproduct,
    \begin{align*}
        \textup{flip}\circ (\circledast\ot\circledast)\Delta_{L}(x)=1\ot x^\circledast+x_{i}^\circledast\ot (\omega_{i}^*, x^\circledast)=\Delta_{R}(x^\circledast).
    \end{align*}
    And
    \begin{align*}
        \textup{flip}\circ (\circledast\ot\circledast)\Delta_{L}(x, \xi)=(\xi^*, x_i^\circledast)\ot (\omega_{i}^*, x^\circledast)=\Delta_{R}((x,\xi)^\circledast).
    \end{align*}
\end{proof}

Following \cite{AryGho1}, but in our notation, we now define $\cI\cL(\Omega^1)$ to be  free product of tensor algebras
$T_{A^e} (\cX^R\tens\Omega^{1})$,  $T_{A^e} (\Omega^{1}\ot \cX^L)$ and $T(\cX^R)$ by the relations (\ref{equ. relations 1}) (now written with $\bullet$)  and (\ref{equ. relations 3}). $\cI\cL(\Omega^1)$ is a left bialgebroid with $A^e$-ring structure given by (\ref{ufyt}),
and coring structure given by (\ref{equ. coproduct 1}), (\ref{equ. counit 1}) and (\ref{equ. coring structure 1}).
Similarly, we can define $\cI\cR(\Omega^1)$ to be free product of tensors products $T_{A^e} (\cX^R\tens\Omega^{1})$,  $T_{A^e} (\Omega^{1}\ot \cX^L)$ and $T(\cX^L)$ by the relations (\ref{equ. relations 4}) and
\begin{align}\label{equ. relations 5}
    (x_i, \omega)\Rbullet(\omega_i, y)=\ev^R(\omega, y),\quad (\omega, y_j)\Rbullet(x, \eta_{j})=\underline{\ev^L(x, \omega)}.
\end{align}
We can see $\cI\cR(\Omega^1)$ is a right bialgebroid with $A^e$-ring structure
\begin{align}
    s_R(a)=a,\quad t_R(a)=\und{a},
\end{align}
and $A$-coring given by (\ref{equ. coring structure 2}) and
\begin{align}
    \Delta_R(x,\omega)=(x,\omega_i)\ot (x_i,\omega),\quad \varepsilon_R(x, \omega)=\ev^L(x, \omega).
\end{align}

By Propositions \ref{prop. relation between invertible bimodule and left module} and \ref{prop. relation between bimodule connection and left module}, as in \cite{AryGho1}, we  have isomorphisms of monoidal categories
\begin{equation} \label{prop. invertible bimodule connection}
     {}_{\cI\CL(\Omega^1)}\CM\cong {}_{A}\cI\CE_{A},\quad {}\CM_{\cI\CR(\Omega^1)}\cong {}_{A}\cI\CE_{A}^R,
\end{equation}
and similarly for the right side.

\begin{theorem}\label{thm star bialgebroid pair IL and IR}
    If $A$ is a $*$-algebra and $(\Omega^1, \extd)$ is an fgp $*$-differential calculus then
$(\cI\cL(\Omega^1), \cI\cR(\Omega^1))$ is a $*$-bialgebroid pair, with $\circledast$ given by Lemma \ref{lem. star related LX and RX} and additionally
\[(\omega, y)^\circledast=(y^{\circledast^{-1}}, \omega^*).\]
And $\Phi:\cI\cL(\Omega^1)\to \cI\cR(\Omega^1)$ given by
\[\Phi(x, \omega)=(x,\omega),\quad \Phi(\omega, y)=(\omega, y),\quad \Phi((\omega_i, y)\bullet x_i)=y,\quad \Phi(x)=y_j\Rbullet(x, \eta_j),\quad \Phi(a\und{b})=\und{a}b.\]
\end{theorem}
\begin{proof}
    By the proof of Lemma \ref{lem. star related LX and RX}, we only need to show
    \[((\omega_i, y)\bullet(x_i, \omega))^\circledast=(x_i, \omega)^\circledast \Rbullet (\omega_i, y)^\circledast,\quad ((x,\eta_{j})\bullet(\omega,y_j))^\circledast=(\omega, y_j)^\circledast \Rbullet(x, \eta_{j})^\circledast.\]
    On the one hand,
    \begin{align*}
       ((\omega_i, y)\bullet(x_i, \omega))^\circledast=\underline{\ev^R(\omega, y)}^\circledast=\underline{\ev^R(\omega, y)^*}.
    \end{align*}
    On the other hand,
    \begin{align*}
        (x_i, \omega)^\circledast\Rbullet (\omega_i, y)^\circledast=(\omega^*, x_i^\circledast)\Rbullet(y^{\circledast^{-1}}, \omega_i^*)=\underline{\ev^L(y^{\circledast^{-1}}, \omega^*)},
    \end{align*}
    so they are equal. The rest is similar. We will also see
    \begin{align*}
        \textup{flip}\circ (\circledast\ot \circledast)\circ \Delta_L(\xi, y)=(y^{\circledast^{-1}}, \eta_{j}^*)\ot (y_j^{\circledast^{-1}}, \xi^*)=\Delta_R(y^{\circledast^{-1}}, \xi^*)=\Delta_R((\xi, y)^\circledast).
    \end{align*}
    To see the bialgebroid pair is $\Phi$-reflexive. We can first observe that $\Phi((\omega_i, y)\bullet(x_i, \omega))=\Phi((x_i, \omega))\Rbullet\Phi((\omega_i, y))$ and $\Phi((x,\eta_{j})\bullet(\omega,y_j))=\Phi(\omega, y_j)\Rbullet\Phi(x, \eta_{j})$. Also,
    \begin{align*}
       \Phi(x_i)\Rbullet \Phi(\omega_i, y)=y_j\Rbullet (x_i, \eta_j)\Rbullet (\omega_i, y)=y_j\Rbullet \ev^R(\eta_j, y)=y=\Phi((\omega_i, y)\bullet x_i).
    \end{align*}
    To see $\Phi$ is compatible with the coproduct, we have
    \begin{align*}
        (\Phi\ot& \Phi)\circ \Delta_L((\omega_{i}, y)\bullet x_i)=(\Phi\ot \Phi)(\Delta_L(\omega_{i}, y)\bullet\Delta_L(x_i))\\
        =&(\Phi\ot \Phi)\big(((\omega_{i}, y_j)\ot (\eta_{j}, y))\bullet((x_i\ot 1)+(x_i, \omega_{k})\ot x_k)\big)\\
        =&(\Phi\ot \Phi)\big((\omega_{i}, y_j)\bullet x_i \ot (\eta_{j}, y) +(\omega_{i}, y_j)\bullet(x_i, \omega_{k})\ot (\eta_{j}, y)\bullet x_k\big)\\
        =&(\Phi\ot \Phi)\big((\omega_{i}, y_j)\bullet x_i \ot (\eta_{j}, y) +t_L(\ev^R(\omega_{k}, y_j))\ot (\eta_{j}, y)\bullet x_k\big)\\
        =&(\Phi\ot \Phi)\big((\omega_{i}, y_j)\bullet x_i \ot (\eta_{j}, y) +1 \ot (\ev^R(\omega_{k}, y_j)\eta_{j}, y)\bullet x_k\big)\\
        =&(\Phi\ot \Phi)\big((\omega_{i}, y_j)\bullet x_i \ot (\eta_{j}, y) +1 \ot (\omega_{k}, y)\bullet x_k\big)\\
        =&y_j\ot (\eta_{j}, y)+1\ot y\\
        =&\Delta_R\circ \Phi((\omega_{i}, y)\bullet x_i).
    \end{align*}
 Similarly, we will see $(\Phi\ot \Phi)\circ \Delta_L(x)=\Delta_R\circ \Phi(x)$. Indeed, on the one hand,
 \begin{align*}
     (\Phi\ot& \Phi)\circ \Delta_L(x)= (\Phi\ot \Phi)(x\ot 1+(x, \omega_{i})\ot x_i)\\
     =&y_j\Rbullet(x, \eta_j)\ot 1+(x, \omega_{i})\ot y_j\Rbullet(x_i, \eta_j),
 \end{align*}
 on the other hand,
 \begin{align*}
     \Delta_R\circ& \Phi(x)= \Delta_R(y_j)\Rbullet  \Delta_R(x,\eta_j)\\
     =&(1\ot y_j+y_k\ot (\eta_{k}, y_j))\Rbullet ((x,\omega_i)\ot (x_i, \eta_j))\\
     =&(x,\omega_i)\ot y_j\Rbullet (x_i, \eta_j)+y_k\Rbullet (x,\omega_i)\ot (\eta_{k}, y_j)\Rbullet(x_i, \eta_j)\\
     =&(x,\omega_i)\ot y_j\Rbullet (x_i, \eta_j)+y_k\Rbullet (x,\omega_i)\ot t_R(\ev^L(x_i, \eta_{k}))\\
     =&(x,\omega_i)\ot y_j\Rbullet (x_i, \eta_j)+y_k\Rbullet (x,\omega_i \ev^L(x_i, \eta_{k}))\ot 1\\
     =&(x,\omega_i)\ot y_j\Rbullet (x_i, \eta_j)+y_k\Rbullet (x,\eta_{k})\ot 1.
 \end{align*}
 Also, we have $\varepsilon_R\circ \Phi((\omega_{i}, y)\bullet x_i)=\varepsilon_R(y)=0$ and $\varepsilon_R\circ\Phi(x)=\varepsilon_R(y_j\Rbullet(x, \eta_j))=\varepsilon_R(t_R(\varepsilon_R(y))(x, \eta_j))=0$. Moreover,
 \begin{align*}
     \Phi^{-1}\circ \circledast(x)=\Phi^{-1}(x^\circledast)=(\omega_{i}, x^\circledast)\bullet x_i=(\eta_j^*, x^\circledast)\bullet y_j^{\circledast^{-1}}=(y_j\Rbullet(x, \eta_j))^{\circledast^{-1}}=\circledast^{-1}\circ \Phi(x).
 \end{align*}
\end{proof}
By Theorem~\ref{thmpairbar}, we have
\begin{corollary}\label{coro. invertible connection}
     If $A$ is a $*$-algebra and $(\Omega^1, \extd)$ is an fgp $*$-differential calculus then  ${}_{\CI\CL(
     \Omega^1)}\CM$  is a bar subcategory of ${}_{\cI\cL(\Omega^1)}\CJ_{\cI\cR(\Omega^1)}$. Moreover, ${}_{\cI\cL(\Omega^1)}\CM\cong \CM_{\cI\cR(\Omega^1)}\cong {}_A\cI\CE_A$ as bar categories.
\end{corollary}
\proof  We combine Theorem \ref{thmpairbar} with the above results and identify ${}_{\cI\cL(\Omega^1)}\CM$ with its image in ${}_{\cI\cL(\Omega^1)}\CJ_{\cI\cR(\Omega^1)}$. Recall that in Theorem \ref{thmpairbar}, the bar structure of $M\in {}_{\cI\cL(\Omega^1)}\CM$ is given by
\[X\la\overline{m}=\overline{\Phi^{-1}(X^\circledast)\la m},\]
    for any $x\in\cL$ and $m\in M\in {}_{\cI\cL(\Omega^1)}\CM$. We also need to check that the bar category structure agrees with that of ${}_A\cI\CE_A$, where $F:{}_{\cI\cL(\Omega^1)}\CM\to {}_A\cI\CE_A^L$ is the monoidal functor given by (\ref{prop. invertible bimodule connection}) with formula given by Proposition \ref{prop. relation between bimodule connection and left module}, then we can check
\begin{align*}
    \nabla_{F(\overline{M})}(\overline{m})=&\omega_{i}\ot x_i\la\overline{m}=\omega_{i}\ot \overline{m\ra x_{i}^\circledast}=\omega_i\ot\overline{(\id\ot \ev^R)((\sigma_{M})^{-1}\circ \nabla_M(m)\ot x_i^\circledast)}\\
    =&\omega_i\ot\overline{m^\alpha \ev^R(m^\beta, x_i^\circledast)}=\omega_i\ev^L(x_i, m^\beta{}^*)\ot\overline{m^\alpha }=m^\beta{}^*\ot \overline{m^\alpha }\\
    =&\nabla_{\overline{F(M)}}(\overline{m}),
\end{align*}
where $m^\alpha\ot m^\beta=(\sigma_{M})^{-1}\circ \nabla_{M}(m)$. We do not need to check it (it is implied) but this also works for the braiding:
\begin{align*}
     \dagger^{-1}\circ\sigma_{F(\overline{M})}(\overline{m}\ot \omega)=&\dagger^{-1}(\omega_i\ot (x_i,\omega)\la \overline{m})=\dagger^{-1}(\omega_i\ot\overline{m\ra (\omega^*, x_i^\circledast)})\\
     =&\overline{m\ra (\omega^*, x_i^\circledast)\ot \omega_i^*}=\overline{(\sigma_{E})^{-1}(\omega^*\ot m)}\\
     =&\dagger^{-1}\circ\sigma_{\overline{F(M)}}(\overline{m}\ot \omega).
\end{align*}
\endproof

\subsection{$*$-Hopf algebroid pair $\cI\cI\cL(\Omega^1),\cI\cI\cR(\Omega^1)$}\label{sec:difpiv}

In this section we assume that $\Omega^1$ is pivotal so that $\cX=\cX^R=\cX^L$. Then \cite{AryGho1} also constructed Hopf algebroids, which we denote  $\cI\cI\cL(\Omega^1), \cI\cI\cR(\Omega^1)$. Here,  $\cI\cI\cL(\Omega^1)$ is the quotient of $\cI\cL(\Omega^1)$ by the relations (\ref{equ. relations 2}) (now written with $\bullet$), while $\cI\cI\cR(\Omega^1)$ is the quotient of $\cI\cR(\Omega^1)$ by the relations
\begin{align}\label{equ. relations 6}
    (\eta_{j}, x)\Rbullet(y_{j},\omega)=\ev^{L}(x, \omega),\quad (x, \omega_{i})\Rbullet(\omega,x_{i})=\und{\ev^{R}(\omega, x)}.
\end{align}
By a similar method as Theorem \ref{thm star bialgebroid pair IL and IR}, we have:
\begin{proposition}\label{prop. star Hopf algebroid pair IL and IR}
    If $A$ is a $*$-algebra and $(\Omega^1, \extd)$ is an fgp $*$-differential calculus then
$(\cI\cI\cL(\Omega^1), \cI\cI\cR(\Omega^1))$ is a $*$-Hopf algebroid pair, with $\circledast$ given by Lemma \ref{lem. star related LX and RX} and additionally
\[(\omega, y)^\circledast=(y^{\circledast^{-1}}, \omega^*).\]
And $\Phi:\cI\cI\cL(\Omega^1)\to \cI\cI\cR(\Omega^1)$ given by
\[\Phi(x, \omega)=(x,\omega),\quad \Phi(\omega, y)=(\omega, y),\quad \Phi((\omega_i, y)\bullet x_i)=y,\quad \Phi(x)=y_j\Rbullet(x, \eta_j).\]
\end{proposition}
By Theorem \ref{thm. biinvertible module} and (\ref{prop. invertible bimodule connection}), we can see that:
\begin{corollary}\label{thm. biinvertible connections as bar category}
      If $A$ is a $*$-algebra and $(\Omega^1, \extd)$ is an fgp $*$-differential calculus then  ${}_A\cI\cI\CE_A\cong {}_{\cI\cI\cL(\Omega^1)}\mathcal{M}\cong \mathcal{M}_{\cI\cI\cR(\Omega^1)}$ as monoidal subcategories of ${}_{\CI\CI\CL(\Omega^1)}\CJ_{\CI\CI\CR(\Omega^1)}$ and as subcategories of   ${}_A\cI\CE_A$,  $ {}_{\cI\cL(\Omega^1)}\mathcal{M}$ and $\mathcal{M}_{\cI\cR(\Omega^1)}$ respectively.
\end{corollary}
\begin{proof}
    By Corollary \ref{coro. invertible connection}, we know ${}_A\cI\CE_A\cong{}_{\cI\cL(\Omega^1)}\mathcal{M}$ as bar category. And ${}_A\cI\cI\CE_A$, ${}_{\cI\cI\cL(\Omega^1)}\CM$ inherited the bar structure from them respectively. Let $F:{}_{\cI\cI\cL(\Omega^1)}\CM\to{}_A\cI\cI\CE_A$ be the functor with formula given by Proposition \ref{prop. relation between bimodule connection and left module}. We can check $ \tau_{F(\overline{E})}(x\ot \overline{e})=(x,\omega_i)\la \overline{e}\ot x_i=\tau_{\overline{F(E)}}(x\ot \overline{e})$ which is the same with the proof of Theorem \ref{thm. biinvertible module}.
\end{proof}

\subsection{The category ${}_{\cL(\Omega^1)}\cK_{\cR(\Omega^1)}$}\label{secK}

Here we give a route to the  $\CI\CL(\Omega^1)$ construction in terms of the $\odot$ construction in Proposition~\ref{propodot}. The starting point is a certain subcategory of $ {}_{\cL(\Omega^1)}\CJ_{\cR(\Omega^1)}$.

\begin{proposition}\label{propK} Let ${}_{\cL(\Omega^1)}\cK_{\cR(\Omega^1)}$  be the full monoidal subcategory of ${}_{\cL(\Omega^1)}\CJ_{\cR(\Omega^1)}$ where we impose the four relations
\begin{align}\label{equ. addition conditions}
(x_i\la e)\ra (\omega_{i}, y) &= e\ra y  \cr
 (x, \eta_i)\la(e\ra y_i) &= x\la e \cr
(x, \eta_i) \la (e\ra (\xi, y_i)) &= \ev^L(x, \xi)\, e \cr
((x_i,\eta)\la e)\ra(\omega_{i}, y) &= e\,\ev^R(\eta, y),
\end{align}
for any $e\in E\in {}_{\cL(\Omega^1)}\cK_{\cR(\Omega^1)}$. This category can be identified with the left modules of $\CL(\Omega^1)\odot \CR(\Omega^1)^{op}$ in Proposition \ref{propodot} modulo the further relations
\[ y=(\omega_i, y)\bullet  x_i,\quad   x=(x,\eta_i)\bullet y_i \]
\[ (x,\eta_{j})\bullet (\omega, y_{j})=\ev^{L}(x, \omega),\quad   (\omega_{i},y)\bullet (x_{i}, \omega)=\underline{\ev^{R}(\omega, y)}. \]
Moreover, this quotient can be identified with $\CI\CL(\Omega^1)$.
\end{proposition}
\proof (1) We first check that it is monoidal: for any $e\in E\in {}_{\cL(\Omega^1)}\cK_{\cR(\Omega^1)}$ and $f\in F\in {}_{\cL(\Omega^1)}\cK_{\cR(\Omega^1)}$, we have
    \begin{align*}
        (x_i\la& (e\ot f))\ra (\omega_{i}, y)\\
        =&(x_i\la e\ot f)\ra (\omega_{i}, y)+((x_i,\omega_{j})\la e\ot x_j\la f)\ra (\omega_{i}, y)\\
        =&(x_i\la e)\ra (\omega_{i}, y_j)\ot f\ra (\eta_{j}, y)+((x_i,\omega_{j})\la e)\ra (\omega_{i}, y_k)\ot (x_j\la f)\ra (\eta_{k}, y)\\
        =&e\ra y_j\ot f\ra (\eta_{j}, y)+e\, \ev^R(\omega_{j}, y_k)\ot (x_j\la f)\ra (\eta_{k}, y)\\
         =&e\ra y_j\ot f\ra (\eta_{j}, y)+e\, \ot (x_j\la f)\ra ((\eta_{k}, y)t_{R}(\ev^R(\omega_{j}, y_k)))\\
        =&e\ra y_j\ot f\ra (\eta_{j}, y)+e\, \ot (x_j\la f)\ra (\ev^R(\omega_{j}, y_k) \eta_{k}, y)\\
        =&e\ra y_j\ot f\ra (\eta_{j}, y)+e\, \ot (x_j\la f)\ra (\omega_{j}, y)\\
        =&e\ra y_j\ot f\ra (\eta_{j}, y)+e\, \ot f\ra y\\
        =&(e\ot f)\ra y.
    \end{align*}
    Also we have
    \begin{align*}
        (x, &\eta_i) \la ((e\ot f)\ra (\xi, y_i)) \\
        =& (x, \eta_i) \la(e\ra (\xi, y_j)\ot f\ra(\eta_{j}, y_i))\\
        =&(x, \omega_{k})\la (e\ra (\xi, y_j))\ot (x_k, \eta_i)\la (f\ra(\eta_{j}, y_i))\\
        =&(x, \omega_{k})\la (e\ra (\xi, y_j))\ot \ev^L(x_k, \eta_{j})\, f\\
        =&(x, \omega_{k} \ev^L(x_k, \eta_{j}))\la (e\ra (\xi, y_j))\ot  f\\
        =&(x, \eta_{j})\la (e\ra (\xi, y_j))\ot  f\\
        =&\ev^L(x, \xi)\, e\ot  f.
    \end{align*}
    The rest are similar.

    (2) The actions of $\CL(\Omega^1),\CR(\Omega^1)^{op}\subset \CL(\Omega^1)\odot\CR(\Omega^1)$ are
    \[  x.e=x\la e,\quad (x, \omega).e=(x, \omega)\la e, \quad y.e=e\ra y ,\quad (\eta, y).e=e\ra(\eta,y)  \]
   using which the restrictions on the objects translate into the relations stated. For example, the fourth relation in the subcategory can be written
 \[ (\omega_i,y).((x_i,\eta).e)=e\ev^R(\eta,y)=\underline{\ev^R(\eta,y)}.e\]
 using the notation (\ref{Aeconv}). The last two of the stated quotient relations are the same as those of $\CI\CL(\Omega^1)$ and note that if these hold then then the second relation follows from the first,
  \[ (x,\eta_i)y_i=(x,\eta_i)(\omega_j,y_i)x_j=\ev^L(x,\omega_j)x_j=x.\]
  Finally, the first quotient relation allows one to eliminate the $\cX^L$ generators, so that we have exactly $\CI\CL(\Omega^1)$. \endproof

On the other hand, the same category ${}_{\cL(\Omega^1)}\cK_{\cR(\Omega^1)}$ can be expressed as right modules of quotient of $\CL(\Omega^1)^{op}\und{\odot}\CR(\Omega^1)$ by relations
\[  y=x_i\Rbullet (\omega_i,y),\quad x=y_i\Rbullet (x, \eta_i)\]
\[ (\omega,y_i)\Rbullet (x,\eta_i)=\und{\ev^L(x,\omega)},\quad (x_i,\omega)\Rbullet (\omega_i,y)=\ev^R(\omega,y)\]
similarly using
\[ e.x= x\la e,\quad e.(x,\omega)=(x,\omega)\la e,\quad e.y= e\ra y,\quad e.(\eta,y)=e\ra(\eta,y)\]
and the specification of the subcategory. We see that we recover $\CI\CR(\Omega^1)$, using the second relation to replace
$\cX^R$ generators by $\cX^L$ ones, and the first relation then being redundant.

The merit of this treatment is that Proposition~\ref{propodot} tells us that, since we know by Lemma~\ref{lem. star related LX and RX}  that $\CL(\Omega^1), \CR(\Omega^1)$ are $*$-related, $\CL(\Omega^1)\odot\CR(\Omega^1)^{op}, \CL(\Omega^1)^{op}\und{\odot}\CR(\Omega^1)$ are a $*$-bialgebroid pair. One can then check that this is compatible with the relations that we are quotienting by, to recover Theorem~\ref{thm star bialgebroid pair IL and IR}. In categorical terms, this is equivalent to checking that ${}_{\cL(\Omega^1)}\cK_{\cR(\Omega^1)}$ is a bar subcategory, which we check.

\begin{proposition} If $A$ is a $*$-algebra and $(\Omega^1,\extd)$ an fgp $*$-differential calculus $\Omega^1$ then ${}_{\cL(\Omega^1)}\cK_{\cR(\Omega^1)}$ is a bar subcategory of ${}_{\cL(\Omega^1)}\CJ_{\cR(\Omega^1)}$ and isomorphic to ${}_{\cI\cL(\Omega^1)}\CM$ and $\CM_{\cI\cR(\Omega^1)}$ (and  hence ${}_A\cI\CE_A$) as bar categories.
\end{proposition}
\begin{proof}
    The bar structure of ${}_{\cL(\Omega^1)}\cK_{\cR(\Omega^1)}$ is inherited from ${}_{\cL(\Omega^1)}\CJ_{\cR(\Omega^1)}$. Now, we can check for any $e\in E\in {}_{\cL(\Omega^1)}\cK_{\cR(\Omega^1)}$.
    \begin{align*}
        (x_{i}\la \overline{e})\ra (\omega_{i}, y)=\overline{(y^{\circledast^{-1}}, \omega_{i}^*)\la(e\ra x_i^\circledast)}=\overline{y^{\circledast^{-1}}\la e}=\overline{e}\ra y.
    \end{align*}
    Similarly, we have $ (x, \eta_i)\la(\overline{e}\ra y_i)= x\la \overline{e}$. Next,
    \begin{align*}
        (x, \eta_i) \la (\overline{e}\ra (\xi, y_i))=&\overline{((y_i^{\circledast^{-1}}, \xi^*)\la e)\ra (\eta_i^*, x^\circledast)}=\overline{e \ev^R(\xi^*, x^\circledast)}= \ev^{L}(x, \xi)\, \overline{e}.
    \end{align*}
  Similarly for the fourth relation for an object.  So ${}_{\cL(\Omega^1)}\cK_{\cR(\Omega^1)}$ is a bar subcategory. The rest is then clear as we have already identified the category as left $\CI\CL(\Omega^1)$-modules and as right $\CI\CR(\Omega^1)$-modules, given the remarks above and Theorem~\ref{thmpairbar}.
\end{proof}

\subsection{The biparallelisable case}\label{sec:ex}

We note that if $\Omega^1$ is fgp and has a generalised (i.e. not necessarily symmetric) quantum metric in the sense\cite{BM} of $(\ ,\  ):\Omega^1\tens_A\Omega^1\to A$ and $\cg\in \Omega^1\tens_A\Omega^1$ obeying the snake identities then $\Omega^1$ is pivotal with
\begin{equation}\label{pivg} \cX=\Omega^1,\quad \ev^L(\omega, \eta)=(\omega, \eta)=\ev^R(\omega, \eta),\quad \coev^R=\cg=\coev^L.\end{equation}
One can check that the $\circledast$ on $\cX$ coincides with $*$ on $\Omega^1$ if and only if $\cg$ is $*$-compatible in the sense $\dagger(\cg)=\cg$ as in \cite{BM}, although we do not need to assume this nor to assume this form of pivotal structure.
Either way, there are by now many examples of quantum Riemannian geometry sufficient to illustrate  Ghobadi's $\CL(\Omega^1),\CI\CL(\Omega^1)$ and $\CI\CI\CL(\Omega^1)$ and exhibit them as $*$-bialgebroid and $*$-Hopf algebroid pairs.

In the rest of this section we focus on the simplest class of these where  calculus on $A$ has $\Omega^1$ free from both sides, in fact with a finite set $\{\omega_i\}$ which is both a left basis and a right basis over $A$. In this section, we analyse the above constructions at this level of generality. In this case exterior derivative can be written as for all $f\in A$ as
\[ \extd f=\del_i(f)\omega_i= \omega_i\del^R_i(f),\quad \del_i, \del^R_i:A\to A\]
as a definition of `partial derivatives' (this is what they would for a local coordinate basis in the classical case at this point). We sum over repeated indices. Our assumption is that
\[  \omega_j f= C_{ij}(f) \omega_i,\quad  f\omega_j =\omega_i C^{-1}_{ij}(f),\quad   C_{ij},C^{-1}_{ij}: A\to A,\]
for certain `commutation operators' which are mutually inverse as elements of $M_d({\rm Lin}_k(A))$, where $d$ is the order of the basis. The requirements of a differential calculus $(\Omega^1,\extd)$ translate to
\begin{align}\label{Cdel}
C_{ij}(fg)=C_{kj}(f) C_{ik}(g),\quad \del_i(fg)=\del_j(f)C_{ij}(g)+ f \del_i(g)  \nonumber\\
C^{-1}_{ij}(fg)=C^{-1}_{ik}(f)C^{-1}_{kj}(g),\quad \del_i^R(fg)=\del_i^R(f) g+ C_{ij}^{-1}\del_j^R(g)
\end{align}

We let $x_i$ and $y_i$ be corresponding dual bases so that
\[ \ev^L(x_i,\omega_j)=\delta_{ij}=\ev^R(\omega_i,y_j),\quad  \coev^L=\omega_i\tens x_i,\quad \coev^R=y_i\tens\omega_i\]
(so we set $\eta_i=\omega_i$ in the general theory).  The bimodule structures on $\cX^R$ and $\cX^R$ are then
\[ x_i f= C^{-1}_{ij}(f) x_j,\quad f y_i=y_jC_{ij}(f).\]
Also note that
\[ \ev^L(x_i,\extd f)=\del^R_i(f),\quad \ev^R(\extd f,y_i)=\del_i(f)\]
which allows is to identify $x_i=\del_i^R, y_i=\del_i$ corresponding to the conventions of classical differential geometry where vector fields are thought of as differential operators.

Next, $\CL(\Omega^1)$ is generated by $f\in A$, $\und f\in A^{op}$,
$t_{ij}=(x_i,\omega_j)$ and $x_i\in\cX^R$. These are subject to the relations for $A^e$ and (using $\bullet$ for the product in  $\CL(\Omega^1)$)
\begin{align*}
& f \bullet x_i = f\,x_i,\quad x_i\bullet f = x_i\, f + \ev^L(x_i,\extd f),\quad
x_i \bullet \und{f} =  \und{f} \bullet x_i + (x_i, \extd f),\cr
&f \bullet (x,\omega) = (f\,x,\omega) ,\quad (x, \omega)\bullet f = (xf, \omega),\quad \underline f \bullet(x, \omega) = (x, \omega f),\quad (x, \omega)\bullet \underline f = (x,f\omega).
\end{align*}
and the coring structure is
\begin{align*}
&\Delta_{L}(x_i) = x_i\tens 1 +t_{ij}\tens x_j,\quad \varepsilon_{L}(x_i)=0,\cr
&\Delta_{L}(f) =f\tens 1,\quad \Delta_{L}(\underline f) =1\tens\underline f,\quad
\varepsilon_{L}(f)=f,\quad \varepsilon_{L}(\underline f)=f,\cr
&\Delta_{L} (t_{ij}) = t_{ik}\tens t_{kj}   \ ,\quad
\varepsilon_{L}(t_{ij} )=\delta_{ij}.
\end{align*}
The source and target maps are $s_L(f)=f$ and $t_L(f)=\und f$.

Similarly on the other side, $\CR(\Omega^1)$ is generated by $f\in A$, $\und f\in A^{op}$,
${\und t}{}_{ij}=(\eta_i,y_j)$ and $y_i\in\cX^L$, subject to the relations for $A^e$ and using $\und\bullet$ for the product in  $\CR(\Omega^1)$
\begin{align*}
&f  \Rbullet (\omega, y) = (\omega, f\,y)  \ ,\quad (\omega, y) \Rbullet   f = (\omega, y\,f) \cr
&\underline f  \Rbullet (\omega, y)  = (\omega \, f, y)  \ ,\quad (\omega, y) \Rbullet   \underline f = (f\,\omega, y) \ ,\cr
& y_i \Rbullet  f = y_i\,f\ ,\quad f \Rbullet  y_i = f\, y_i + \ev^R(\extd f, y_i)\ ,\quad
 \underline f  \Rbullet   y_i = y_i  \Rbullet  \underline f + (\extd f , y_i)\ ,
\end{align*}
and the coring structure is
\begin{align*}
&\Delta_{R}(y_i) = 1\tens y_i + y_j\tens {\und t}{}_{ji}\ ,\quad \varepsilon_{R}(y_i)=0\ ,\cr
&\Delta_{R}(f) =1\tens f \ ,\quad \Delta_{R}(\underline f) =\underline f\tens 1\ ,\quad
\varepsilon_{R}(f)=f\ ,\quad \varepsilon_{R}(\underline f)=f\ ,\cr
&\Delta_{R}( {\und t}{}_{ij})= {\und t}{}_{ik}  \tens   {\und t}{}_{kj}  \ ,\quad
\varepsilon_{R}( {\und t}{}_{ij}) =\delta_{ij} \ .
\end{align*}
The source and target maps are $s_R(f)=f$ and $t_R(f)=\und f$.

Next we construct $\CI\CL(\Omega^1)$, which has additional relations from $\CR(\Omega^1)$ but these enter with the opposite product  so $Y_1\bullet Y_2=Y_2 \Rbullet  Y_1$ for $Y_i\in\CR(\Omega^1)$. As explained in Section~\ref{secK}, this arises from as a quotient of the left $A$-bialgebroid $\CL\odot \CR^{op}$ in Proposition~\ref{propodot} for $\cL=\CL(\Omega^1), \CR=\CR(\Omega^1)$. As part of the construction, to have consistent source and target maps for  $\CI\CL(\Omega^1)$, we eliminate all mention of $A^e$ coming from $\CR$ by identifying $\und f, \und g\in \CR$ with $f,g\in \CL$ (this is consistent with the product as
$\CL\odot \CR^{op}$ has the $\CR$ product reversed). Thus, we explicitly rewrite the $\CR$ relations above as follows:
\begin{align} \label{ILex1}
& (\omega, y)\bullet \und f = (\omega, f\,y) ,\quad\und f\bullet (\omega, y) = (\omega, y\,f), \cr
&(\omega, y) \bullet f = (\omega \, f, y) ,\quad f\bullet  (\omega, y) = (f\,\omega, y), \cr
&\und f\bullet y_i = y_i\,f,\quad  y_i \bullet \und f = f\, y_i + \ev^R(\extd f, y_i),\quad
   y_i \bullet f = f\bullet y_i + (\extd f , y_i),
\end{align}
as the relevant relations viewed in the $A$-bialgebroid  $\CL\odot \CR^{op}$.
The source and target maps are
\[s_{\CL\odot\CR^{op}}(f)=f,\quad t_{\CL\odot\CR^{op}}(f)=\und f.\]
Finally, as explained in Section~\ref{secK}, we impose relations  between $\CL$ and $\CR^{op}$ coming from the subcategory $\cK$,
\begin{align}
& (\omega_{i}, y_j)\bullet x_i  =  y_j ,\quad (x_j, \eta_i) \bullet y_i = x_j,  \cr
& (x, \eta_i)\bullet (\omega, y_i) = \ev^L(x, \omega) ,\quad (\omega_{i}, y)\bullet (x_i,\omega) = \underline{\ev^R(\omega, y)}. \label{ILex2}
\end{align}
We can use these to rewrite all occurrences of the generators $y_i$ to get $\cI\cL(\Omega^1)$ in the original Ghobadi form in Section~\ref{secIL}. This amounts to $\CL(\Omega^1)$ with additional generators $\und t_{ij}=(\eta_i,y_j)$ with relations given by the first two lines of (\ref{ILex1}) and the last line of (\ref{ILex2}). The latter on the matrix generators and the coalgebra are
\begin{equation}\label{coalgILex}  t_{ij}\bullet \und t_{kj}=\delta_{ik},\quad \und t_{ij}\bullet t_{ik}=\delta_{jk},\quad \Delta_{L}( {\und t}{}_{ij})= {\und t}{}_{ik}  \tens   {\und t}{}_{kj} ,\quad
\varepsilon_{L}( {\und t}{}_{ij}) =\delta_{ij} .  \end{equation}

To obtain $\cI\cR(\Omega^1)$ from the point of view of Section~\ref{secK}, we begin with the right bialgebroid  $\CL^{op} \, \und\odot \, \CR$ but this time we retain the product $\und\bullet$ of the right bialgebroid $\CR$, and write $X_1 \Rbullet  X_2=X_2\bullet X_1=$ for $X_i\in\CL$. Similarly we retain the $\CR$ copy of $A^e$ by identifying $\und f\in \CL$ with
$f\in \CR$, and $\und g\in \CL$ with $g\in \CR$. Thus, we explicitly rewrite the $\CL$ relations above as
$\CL^{op} \, \und\odot \, \CR$ relations:
\begin{align*}
&  x_i \Rbullet \und f = f\,x_i,\quad \und f \Rbullet x_i = x_i\, f + \ev^L(x_i,\extd f),\quad
f \Rbullet x_i  =  x_i \Rbullet f + (x_i, \extd f),\cr
&(x,\omega) \Rbullet \und f = (f\,x,\omega),\quad  \und f \Rbullet(x, \omega)= (xf, \omega),\quad  (x, \omega) \Rbullet f = (x, \omega f),\quad  f \Rbullet (x, \omega) = (x,f\omega).
\end{align*}
The source and target maps are
\[s_{  \CL^{op} \, \und\odot \, \CR }(f)=f,\quad t_{   \CL^{op} \, \und\odot \, \CR }(f)=\und f.\]
Then, as explained in Section~\ref{secK}, we  impose relations from the category $\cK$ as
\begin{align*}
& x_i  \Rbullet     (\omega_{i}, y_j) =  y_j ,\quad y_i  \Rbullet   (x_j, \eta_i) = x_j, \cr
& (\omega, y_i)   \Rbullet   (x, \eta_i)= \underline{\ev^L(x, \omega) } ,\quad
(x_i,\omega)     \Rbullet     (\omega_{i}, y) = \ev^R(\omega, y).
\end{align*}
We can use these to rewrite all occurrences of the generators $x_i$ to get the original description of  $\cI\cR(\Omega^1)$ in Section~\ref{secIL}. This has addtional generators $t_{ij}$ with relations and coalgebra
\[ \und t_{ij}\Rbullet  t_{kj}=\delta_{ik},\quad t_{ij}\Rbullet \und t_{ik}=\delta_{jk},\quad \Delta_{R}( {\und t}{}_{ij})= {\und t}{}_{ik}  \tens   {\und t}{}_{kj} ,\quad
\varepsilon_{R}( {\und t}{}_{ij}) =\delta_{ij}. \]

We now assume that $(\Omega^1,\extd)$ is a $*$-calculus over $\C$ and give the $*$-bialgebroid pair structure of $\CI\CL(\Omega^1),\CI\CR(\Omega^1)$ and starting with how they arise from the  $*$-bialgebroid pair  $\CL\odot \CR^{op}, \CL^{op} \, \und\odot \, \CR$  in Proposition~\ref{propodot}. This has  the antilinear anti-algebra map $\circledast:\CL\odot \CR^{op} \to
\CL^{op} \, \und\odot \, \CR$ given by
\begin{align*}
f^\circledast = f^*,\quad (\und f)^\circledast = \und f^*,\quad x^\circledast = x^*,\quad y^\circledast = y^*,\quad (x,\omega)^\circledast =
(\omega^*,x^*) ,\quad (\omega,y)^\circledast =
(y^*,\omega^*) .
\end{align*}
Here $\omega^*$ is the usual star of a 1-form, and $x^*\in\cX^L$ for $x\in\cX^R$ and $y^* \in\cX^R$ for $y\in\cX^L$ are defined by $\ev^R(\omega\tens x^*)=\ev^L(x\tens \omega^*)^*$ and
$\ev^L(y^*\tens \omega)=\ev^R(\omega^*\tens y)^*$ respectively  (these were denoted $\circledast$ for $\CL,\CR$ $*$-related but we reserve this symbol for the composite to avoid confusion). As a brief check of applying $\circledast$ to the coproduct (recalling that the coproduct on $\CL\odot \CR^{op}$ and $\CL^{op} \, \und\odot \, \CR$  are just given by $\Delta_L$ and $\Delta_R$) we have
\begin{align*}
\Delta_{\CL^{op} \, \und\odot \, \CR}(x^\circledast) &= \Delta_{\CL^{op} \, \und\odot \, \CR}(x^*) =
\Delta_R(x^*) =1\tens x^* + y_j \tens (\eta_j,x^*) \cr
\mathrm{flip}(\circledast \tens \circledast)\Delta_{ \CL\odot \CR^{op}   } (x) &=
\mathrm{flip}(\circledast \tens \circledast)\Delta_{ L  } (x) =
\mathrm{flip}(\circledast \tens \circledast)( x\tens 1 + (x,\omega_i)\tens x_i) \cr
&= 1\tens x^* + x_i{}^* \tens (\omega_i{}^*,x^*),
\end{align*}
which are equal as the dual bases are related by
$y_i\tens \eta_i=x_i{}^*\tens \omega_i{}^*$.
To see that this construction descends to $\cI\cL(\Omega^1)$ and $\cI\cR(\Omega^1)$ (corresponding to $\cK$ being a bar subcategory in Section~\ref{secK}),  we apply $\circledast$ to the $\cK$-relations in $\CL\odot \CR^{op}$ to get
\begin{align*}
& x_i{}^* \Rbullet (y_j{}^*,\omega_i{}^*) =  y_j{}^*,\quad y_i{}^* \Rbullet (\eta_i{}^*,x_j{}^*) = x_j{}^*, \cr
&(y_i{}^*,\xi^*) \Rbullet   ( \eta_i{}^*,x^*) = \ev^L(x, \xi)^*,
\quad (\eta^*,x_i{}^*)  \Rbullet  ( y^* , \omega_{i}{}^*)=  \underline{ \ev^R(\eta, y)^* }.
\end{align*}
These are the same as the $\cK$-relations in $\CL^{op} \, \und\odot \, \CR$. Moreover,  the bialgebroids $\CL\odot \CR^{op}$ and $\CL^{op} \, \und\odot \, \CR$ are  reflexive as in Proposition~\ref{propodot} by the invertible linear anitalgebra coalgebra map $\Phi:\CL\odot \CR^{op} \to
\CL^{op} \, \und\odot \, \CR$ given by
\begin{align*}
\Phi(f) = \und f,\quad  \Phi(\und f)=  f,\quad  \Phi( x) = x,\quad \Phi( y) = y,\quad \Phi(x,\omega)=
(x,\omega) , \quad \Phi(\omega,y) =
(\omega,y).
\end{align*}
To see that this descends to $\cI\cL(\Omega^1)$ and $\cI\cR(\Omega^1)$, we apply $\Phi$ to the $\cK$-relations in $\CL\odot \CR^{op}$ to get
\begin{align*}
& x_i \Rbullet (\omega_{i}, y_j)  =  y_j ,\quad y_i \Rbullet (x_j, \eta_i) = x_j, \cr
&(\xi, y_i) \Rbullet (x, \eta_i)  = \underline{\ev^L(x, \xi)},\quad (x_i,\eta) \Rbullet (\omega_{i}, y) = \ev^R(\eta, y).
\end{align*}
These are the same as the $\cK$-relations in $\CL^{op} \, \und\odot \, \CR$.

Now we move to the pivotal case we set $\cX=\cX^R$ and with $\ev^L,\coev^L$ as above and define $\cX^L=\cX$ also as a bimodule, with
\[  \ev^R(\omega_i, x_j)= \delta_{ij},\quad    \coev^R(1)=x_i\tens\omega_i\]
and viewed in the correct hom space via this $\ev^R$. Effectively, we set $\eta_i=\omega_i$ and $y_i=x_i$ in the above which works provided their bimodule structures to coincide. This happens when
\[  (C^T)^{-1}= ( C^{-1})^T\]
where $(C^T)_{ij}=C_{ji}$.  We then follow the construction in Section~\ref{secIL} where we define $\cI\cI\cL(\Omega^1)$ as the quotient $\cI\cL(\Omega^1)$ by the $\cI\cI\cL$-relations
 (\ref{equ. relations 2}), which on the generators amounts to the additional relations
 \begin{equation}\label{IILex} t_{ji}\bullet \und t_{jk}=\delta_{ik},\quad \und t_{ij}\bullet t_{kj}=\delta_{ik}.  \end{equation}
 Likewise quotienting $\CI\CR(\Omega^1)$, the additional relations for $\CI\CI\CR(\Omega^1)$ are from (\ref{equ. relations 6}) and on the generators amount to
 \[  \und t_{ji}\Rbullet t_{jk}=\delta_{ik},\quad t_{ij}\Rbullet \und t_{kj}=\delta_{ik}. \]
 The reader may check that the $\circledast$ and $\Phi$ maps send the additional $\cI\cI\cL$-relations to the additional $\cI\cI\cR$-relations so that we obtain a $*$-Hopf algebroid pair.

 Returning to $\CL(\Omega^1)$, an easy example of a bimodule connection is $E=A$ and $\nabla_E=\extd: A\to \Omega^1=\Omega^1\tens_A A$. Here $\sigma_E(f\tens\omega)=f\omega\tens 1$, i.e., $\sigma_E=\id$ when $\tens_A A$ and $A\tens_A$ are cancelled. This corresponds to a representation of $\CL(\Omega^1)$ where the actions are
 \[ x_i.f=\ev^L(x_i,\extd f)\tens 1=\del_i^R(f),\quad t_{ij}.f=(\ev^L\tens\id)(x_i\tens \sigma_E(f\tens \omega_j))=C_{ij}^{-1}(f),\quad \und g.f=fg,\]
 which one can check indeed represents the algebra.  Another question we can ask is the
 structure of the algebra of $\CL(\Omega^1)$. From the above, we see that $x_i, t_{ij}$ generate the free associative algebra $k\<x_i, t_{ij}\>$ on $d^3$ generators. This is actually a bialgebra over the field with coalgebra $\Delta_L,\varepsilon_L$ as above on these generators. All of $\CL(\Omega^1)$ factorises into this algebra and $A^e$, which we keep to the left. The cross-relations from the above are that $A, A^{op}$ commute and
 \[ x_i\bullet f=C^{-1}_{ij}(f)\bullet x_j+\del_i^R(f),\quad t_{ij}\bullet f=C^{-1}_{ik}(f)\bullet t_{kj}\]
 \[ x_i\bullet\und f=\und f\bullet x_i+\und{\del_j^R(f)}\bullet t_{ij},\quad  t_{ij}\bullet \und f=\und{C^{-1}_{kj}(f)}\bullet t_{ik}\]
 The part generated by $A$ and the $x_i, t_{ij}$ is a semidirect product $A\rtimes k\<x_i, t_{ij}\>$ by $x_i\la f=\del_i^R(f)$, $t_{ij}\la f= C^{-1}_{ij}(f)$ which one can check makes $A$ into a module algebra of the bialgebra. The part with $A^{op}$ can similarly be viewed as a semidirect or `smash' product $A^{op}\# k\<x_i, t_{ij}\>^{cop}$ for same left action on the vector space of $A^{op}$ and the opposite coproduct.  This simplifies further in the case of $\CI\CI\CL(\Omega^1)$, where we have additional generators $\und t_{ij}$ and relations `
 \[  \und t_{ij}\bullet f=C_{ki}(f) \bullet \und t_{kj}, \quad \und t_{ij}\bullet\und f=\und{C_{jk}(f)}\bullet \und t_{ik} \]
 from the above, as well as (\ref{coalgILex}) and (\ref{IILex}).

 \begin{proposition}\label{propGhoES} In the case of $\Omega^1$ free as above, $H=k\<x_i,t_{ij},St_{ij}\>\subset\CI\CI\CL(\Omega^)$ with relations such that $St_{ij}=\und t_{ji}$ is a Hopf algebra and with  $\Delta_L$ given by the expressions above but now regarded over $k$. Moreover, $\CI\CI\CL(\Omega^1)= A^e\# H$ a cocycle Hopf algebroid \cite[Prop.~5.2]{HM22} with trivial cocycle.
 \end{proposition}
 \proof Here the Hopf algebra $H$ is the free associative algebra on generators $x_i, t_{ij}, \und t_{ij}$  modulo the relations in (\ref{coalgILex}), (\ref{IILex}) and hence the antipode
 \[ St_{ij}=\und t_{ji},\quad S\und t_{ij}=t_{ji},\quad S x_i = - \und t_{ji}\bullet x_j\]
with the coalgebra as stated, now taken over the field. Next, the  algebra structure of $A^e\# H$ in the case of trivial cocycle  is
 \[ (b\tens b' \#h)(c\tens c'\#g)=b(h\o\la c)\tens c' (Sg\t \la b')\# h\t g\o\]
 The subalgebra $A\tens 1\tens H= A\# H$ has a left handed semidirect form for the action
 \[ x_i\la f=\del_i^R(f),\quad t_{ij}\la f= C^{-1}_{ij}(f),\quad \und t_{ij}\la f=C_{ji}(f)\]
 recovers the stated relations between $A$ and $H$. The subalgebra $1\tens A^{op}\#H$ has the structure of a right-handed semidirect product where the same left action is converted to a right action as $a\ra h:= Sh\la a$. In our case for $\und f\in A^{op}$, we have
 \[ \und f\bullet t_{ij}=t_{ik}\bullet (S t_{kj}\la \und f)=t_{ik}\bullet (\und t_{jk}\la \und f)= t_{ik}\bullet \und{C_{kj}(f)}\]
 which is equivalent to the required relation when applied to $C^{-1}_{jm}(f)$ in place of $f$ and summed over $j$. Similarly
 \[ \und f\bullet \und t_{ij}=t_{ik}\bullet (S \und t_{kj}\la \und f)=t_{ik}\bullet (t_{jk}\la \und f)= t_{ik}\bullet \und{C^{-1}_{jk}(f)}\]
 is equivalent to the required relation when applied to $C_{mj}(f)$ and summed over $j$, provided we are in the pivotal case where we assumed that $(C^T)^{-1}=(C^{-1})^T$. Finally, we have
 \[ \und f\bullet x_i=x_i\bullet \und f+ t_{ij}\bullet (Sx_j\la \und f)=x_i\bullet\und f- t_{ij}\bullet \und{C_{jk}(\del_k^R(f))}\]
 which gives the required commutation relations on using the $\und f, t_{ij}$ relations already proven. This gives the algebra structure of $\CI\CI\CL(\Omega^1)$. One can then check that bialgebroid structures match up as well.
 \endproof

 This is a special case of an Ehresmann-Schauenburg Hopf algebroid $\CL(A\#H,H)\cong A^e\#H$ as explained in \cite[Lem.~5.3]{HM22} for a cleft extension or `trivial bundle' quantum principal bundle, but in our case such that the cleaving map is an algebra map so that there is no cocycle. There is a similar construction for an extension $\widetilde{\CI\CL(\Omega^1)}$ where we keep on adjoining matrix generators $t^{(n+1)}_{ij}=St^{(n)}_{ji}$, where $t^{(0)}_{ij}=t_{ij}$ so that $t^{(1)}_{ij}=\und t_{ij}$.  These have matrix coalgebra $\Delta t^{(n)}_{ij}=t^{(n)}_{ik}\tens t^{(n)}_{kj}$, $\varepsilon(t^{(n)}_{ij})=\delta_{ij}$. The action on $A$ in the $A^e\#H$ construction is again defined iteratively,
 \[  t^{(n)}_{ij}\la f= C^{(n)}_{ij}(f),\quad  C^{(n+1)}= ({C^{(n)}}^{-1})^T,\quad C^{(0)}=C^{-1},\quad C^{(1)}=C^T. \]
 This no longer requires the pivotal construction but is infinitely generated. However, the geometric significance of the action of the higher generators is unclear.

 In the $*$-algebra and $*$-calculus case over $\C$,  if we suppose that $\omega_i^*=\omega_i$ then $x_i^*=y_i$ so that $t_{ij}^*=\und t_{ji}$ from the above $\circledast$.  Then in the pivotal case the above $H$ becomes a {\em flip} Hopf $*$-algebra as in \cite{BM} (where $\Delta$ commutes with $*$ with an extra flip). The axioms for this are such that $S\circ *=*\circ S$. Thus, the $*$-Hopf algebroid pair based on $\CI\CI\CL(\Omega^1)$ is different from the $*$-Hopf algebroid pairs that we obtained on the Ehresmann-Schauenburg Hopf algebroid pair in Proposition~\ref{propES}.

\begin{example}(Integer line graph) \rm
Let $A=\C(\Z)$ with $\Omega^1$ corresponding to the integer line graph. Using the group structure of $\Z$ there is a natural left basis of left-invariant forms which  $\omega_1=e^+,\omega_{-1}=e^-$ in the notations of \cite{BM}. We denote the dual basis of $\cX^R$ by $x_1,x_{-1}$, so that $\ev^L(x_i,\omega_j)=\delta_{ij} $ and $ \coev^L(1)=\omega_i\tens x_i$.
 The bimodule structure on $\Omega^1$ and $\cX^R$ is given by
\[
\omega_if=R_i(f)\, \omega_i,\quad x_if=R_{-i}(f)\, x_i
\]
for $f\in A$ with $R_{\pm1}(f)(i):=f(i\pm 1)$. We define
\[\extd f=(\partial_{1} f)\,\omega_{1} +(\partial_{-1} f)\, \omega_{-1},\]
where $\partial_{\pm1}:=R_{\pm1}-\id$. This is a $*$-calculus with $\omega_i^*=-\omega_{-i}$ when we work over $\C$. Then  $x_i^*=-x_{-i}$ and $y_i^*=-y_{-i}$ and $t_{ij}^*=\und t_{-j-i}=S t_{-i-j}$ which again gives a flip Hopf $*$-algebra.  Here, $C_{ij}(f)=\delta_{ij}R_{i}(f)$, which obeys the condition as assumed for a pivotal structure so that Proposition~\ref{propGhoES} for the structure of $\CI\CI\CL(\Omega^1)$ applies. \end{example}

Another well-studied quantum Riemannian geometry is the fuzzy sphere.

\begin{example} (Fuzzy sphere) \rm
Let $\lambda\ne \pm 1,0$ be a real parameter. The `fuzzy sphere' algebra $A=\C_\lambda[S^2]$ is generated by three generators $z^i, i=1,2,3$, subject to the relations
\[[z^i, z^j]=2\lambda \varepsilon_{ijk}z^k,\quad \sum_{i}(z^i)^2=1-\lambda^2\]
where $\eps_{123}=1$ is the totally antisymmetric symbol, repeated indices will be summed and $[\ ,\ ]$ denotes the commutator. In these conventions, when $\lambda=1/n$ for $n=1,2,\cdots$, there is a natural quotient which is isomorphic to $M_n(\C)$ viewed as matrix fuzzy spheres. The standard (rotationally invariant but 3-dimensional) $\Omega^1$ has central basis $\omega_i=\omega_i^*$ (denoted $s^i$ in \cite{BM}) and exterior derivative
\[\extd f=(\partial_{i} f)\omega_i,\quad  \partial_i(f):=\frac{1}{2\lambda}[z^i, f]\]
for any $f\in A$. This is a $*$-calculus with $\omega_i^*=\omega_i$ and hence $x_i^*=y_i$. Here $C_{ij}(f)=f$ so we are even more simply in the pivotal case where the above results apply and again $H$ is a flip Hopf $*$-algebra.
\end{example}

\end{document}